\documentclass[10pt]{amsart}

\makeatletter
\let\@wraptoccontribs\wraptoccontribs
\makeatother

\usepackage{amssymb}

\usepackage[all]{xy}

\usepackage{enumerate}
\usepackage{todonotes}

\usepackage{relsize}
\usepackage[bbgreekl]{mathbbol}
\usepackage{amsfonts}
\DeclareSymbolFontAlphabet{\mathbb}{AMSb} 
\DeclareSymbolFontAlphabet{\mathbbl}{bbold}
\newcommand{\Prism}{{\mathlarger{\mathbbl{\Delta}}}}

\usepackage{amsmath}
\usepackage{amsxtra}
\usepackage{amscd}
\usepackage{amsthm}
\usepackage{amsfonts}
\usepackage{amssymb}
\usepackage{eucal}
\usepackage{verbatim}
\usepackage[all]{xy}
\usepackage{mathrsfs}
\usepackage{lineno}
\usepackage{color}
\definecolor{deepjunglegreen}{rgb}{0.0, 0.29, 0.29}
\definecolor{darkspringgreen}{rgb}{0.09, 0.45, 0.27}

\usepackage{stmaryrd}

\makeatletter
\usepackage{etoolbox,needspace}
\pretocmd\section{\Needspace*{4\baselineskip}}{}{}
\makeatletter

\usepackage[pdftex,bookmarks=false,colorlinks=true,citecolor=darkspringgreen,debug=true,
  naturalnames=true,pdfnewwindow=true]{hyperref}

\DeclareMathOperator{\Gra}{{Gr}}

\usepackage{tikz-cd}

\newcommand{\rar}[1]{\stackrel{#1}{\longrightarrow}}

\newcommand{\hooklongrightarrow}{\lhook\joinrel\longrightarrow}

\newcommand{\bA}{{\mathbb A}}

\newcommand{\bF}{{\mathbb F}}
\newcommand{\bG}{{\mathbb G}}

\newcommand{\bN}{{\mathbb N}}

\newcommand{\bP}{{\mathbb P}}
\newcommand{\bQ}{{\mathbb Q}}

\newcommand{\bZ}{{\mathbb Z}}

\newcommand{\DFL}{\cD^b(\mathscr{MF}(W(k)))}
\newcommand{\BFL}{\mathscr{DMF}^{\mathrm{big}}(W(k))}
\newcommand{\Cris}{ \mathfrak{p}_{\mathrm{cris}}}
\newcommand{\Reesparameter}{ t}
\newcommand{\MapCmodGmtoWkNred}{\widetilde{\mathfrak{p}}_{\mathrm{cris}}}

\newcommand{\cA}{{\mathcal A}}

\newcommand{\cC}{{\mathcal C}}
\newcommand{\cD}{{\mathcal D}}
\newcommand{\cE}{{\mathcal E}}
\newcommand{\cF}{{\mathcal F}}
\newcommand{\cG}{{\mathcal G}}
\newcommand{\cH}{{\mathcal H}}
\newcommand{\cI}{{\mathcal I}}

\newcommand{\cL}{{\mathcal L}}
\newcommand{\cM}{{\mathcal M}}
\newcommand{\cN}{{\mathcal N}}
\newcommand{\cO}{{\mathcal O}}
\newcommand{\cP}{{\mathcal P}}

\newcommand{\cX}{{\mathcal X}}
\newcommand{\cY}{{\mathcal Y}}

\newcommand{\fX}{{\mathfrak X}}

\newcommand{\fm}{{\mathfrak m}}

\newcommand{\eff}{\text{eff}}

\newcommand{\colim}{\text{colim}}

\newcommand{\et}{\operatorname{\acute{e}t}}

\newcommand{\nc}{\newcommand}

\nc\wh{\widehat}

\nc\on{\operatorname}

\nc\Gr{\on{Gr}}

\nc\Fl{\on{Fl}}

 \DeclareMathOperator{\Spf}{{Spf}}

\DeclareMathOperator{\Mod}{{Mod}}

\DeclareMathOperator{\coker}{{coker}}

\newcommand{\limfrom}{{\displaystyle\lim_{\longleftarrow}}}
\newcommand{\limto}{{\displaystyle\lim_{\longrightarrow}}}
\newcommand{\rightlim}{\mathop{\limto}}

\newcommand{\leftlim}{\mathop{\displaystyle\lim_{\longleftarrow}}}
\newcommand{\limfromn}{\leftlim\limits_{\raise3pt\hbox{$n$}}}
\newcommand{\limton}{\rightlim\limits_{\raise3pt\hbox{$n$}}}

\newcommand{\rightlimit}[1]{\mathop{\lim\limits_{\longrightarrow}}\limits
                    _{\raise3pt\hbox{$\scriptstyle #1$}}}

\newcommand{\leftlimit}[1]{\mathop{\lim\limits_{\longleftarrow}}\limits
                    _{\raise3pt\hbox{$\scriptstyle #1$}}}

\newcommand{\epi}{\twoheadrightarrow}
\newcommand{\iso}{\buildrel{\sim}\over{\longrightarrow}}
\newcommand{\mono}{\hookrightarrow}

\DeclareMathOperator{\Id}{{Id}}

\DeclareMathOperator{\Coker}{{Coker}}
\DeclareMathOperator{\Cone}{{Cone}}

 \DeclareMathOperator{\Ext}{{Ext}}
\DeclareMathOperator{\Hom}{{Hom}}

\DeclareMathOperator{\Ker}{{Ker}} \DeclareMathOperator{\id}{{id}}
\DeclareMathOperator{\im}{{Im}} 

 \DeclareMathOperator{\op}{{op}}

\DeclareMathOperator{\Perf}{{Perf}}
\DeclareMathOperator{\Spec}{{Spec}}
\DeclareMathOperator{\Tor}{{Tor}} 

\DeclareMathOperator{\Vect}{{Vect}}

\DeclareMathOperator{\Pic}{{Pic}}
\DeclareMathOperator{\Coh}{{Coh}}

\DeclareMathOperator{\Fil}{{Fil}}

\makeatletter

\newcommand{\Rmnum}[1]{\expandafter\@slowromancap\romannumeral #1@}
\makeatother

\newtheorem*{theo}{Theorem}

\newtheorem{Th}{Theorem}
\newtheorem{pr}{Proposition}[section]
\newtheorem{lm}[pr]{Lemma}
\newtheorem{claim}[pr]{Claim}

\newtheorem{cor}[pr]{Corollary}

\theoremstyle{definition}

\newtheorem{df}[pr]{Definition}
\newtheorem{rem}[pr]{Remark}

\numberwithin{equation}{section}

\newcommand{\Frac}{\operatorname{Frac}}

\newcommand{\Fun}{\operatorname{Fun}}

\newcommand{\dR}{\mathrm{dR}}
\newcommand{\HT}{\mathrm{HT}}

\DeclareMathOperator{\ad}{{ad}}

\begin{document}

\title[Prismatic $F$-Gauges]
{Prismatic $F$-Gauges and Fontaine--Laffaille modules}
\author{Gleb Terentiuk}
\address{Department of Mathematics, University of Michigan, Ann Arbor, MI, USA}
\email{tgleb@umich.edu}

\author{Vadim Vologodsky}
\address{Department of Mathematics, Princeton University, Princeton, NJ, USA}
\email{vologod@gmail.com}

\author{Yujie Xu}
\address{Department of Mathematics, 
Columbia University, 
New York, NY, USA}
\email{xu.yujie@columbia.edu}

\begin{abstract}
    Let $k$ be a perfect field of characteristic $p$ and $W(k)$ its ring of Witt vectors. We construct an equivalence of categories between the full subcategory of the derived category of quasi-coherent sheaves on the syntomification of $W(k)$ spanned by objects whose Hodge-Tate weights are between $[0,p-2]$ and an appropriate derived category of Fontaine-Laffaille modules. 
\end{abstract}

\maketitle

\section{Introduction}
\subsection{Fontaine--Laffaille Theory}
Fix a perfect field $k$ of characteristic $p$, and let $W(k)$ be the ring of Witt vectors. For a smooth $p$-adic formal scheme\footnote{Recall that a $p$-adic formal scheme over $W(k)$ is a functor from the category of $p$-nilpotent $W(k)$-algebras to sets such that, for every positive integer $n$, its restriction to $W_n(k)$-algebras is represented by a scheme $X_n$ over $W_n(k)$. Thus, $X = \colim_n X_n$, where the colimit is taken in the category of functors from $p$-nilpotent $W(k)$-algebras to sets.} $X$ over $W(k)$ of dimension $\leq p-1$, we have a canonical isomorphism in the derived category of Zariski sheaves on $X$:
\begin{equation}\label{intro-Mazur-isom}
(\Omega_X^{\bullet},p\cdot d_{\dR})\xrightarrow{\sim}(\Omega_X^{\bullet},d_{\dR}). 
\end{equation}
Its composition with the inclusion of complexes $(\Omega_X^{\bullet}, d_{\dR})\hookrightarrow (\Omega_X^{\bullet},p\cdot d_{\dR})$ given by $\Omega_X^i \xrightarrow{p^i\Id}\Omega_X^i$ in degree $i$, is the crystalline Frobenius morphism.~A proof of this fundamental and beautiful result, as well as a more general logarithmic version, appears in \cite[Theorem 6.10]{Ogus-log-dR-Witt}. The idea behind the construction of the map \eqref{intro-Mazur-isom} can be traced back to \cite{Mazur-1973}\footnote{We refer the reader to Remark \ref{remark:Phi-Maz-tilde-isom-PhiMaz-Hsyn} for a sketch of the Mazur-Ogus argument.}.   
One infers from \eqref{intro-Mazur-isom} that the crystalline Frobenius $\varphi$ restricted to the $i$-th term of the ``filtration b\^ete'' is canonically divisible by $p^i$: 
\begin{equation}\label{eqn:intro-stupid-filtration}
    \frac{\varphi}{p^i}: (\Omega_X,d)^{\geq i}\to (\Omega_X,d).
\end{equation}
More generally, in \cite{Mazur-1973},
Mazur proved that, for a smooth $p$-adic  formal scheme $X$ over $W(k)$ of any dimension and any integer $i\geq 0$, the crystalline Frobenius $\varphi$  restricted to  $(\Omega_X,d)^{\geq i}$ is canonically divisible by $p^{[i]}$, where $[i]:=\min\limits_{m\geq i}\mathrm{ord}\frac{p^m}{m!}$; we refer the reader to Remark \ref{remark:Mazur-module} for an interpretation of this result. 
See \cite{Kato-87,Antieau-Mathew-Morrow-Nikolaus} for a construction of \eqref{eqn:intro-stupid-filtration} via quasi-syntomic descent. 

The reduction of \eqref{intro-Mazur-isom} modulo $p$ recovers the Deligne--Illusie isomorphism from \cite{Deligne-Illusie-87}, provided that $X$ has dimension $\leq p-2$.\footnote{In dimension $p-1$, the reduction of \eqref{intro-Mazur-isom} differs from the Deligne--Illusie isomorphism; see Remark \ref{rem:Mazur_vs_Deligne_Illusie}.}

Inspired by \cite{Fontaine-Laffaille}, we describe the linear-algebraic structure on the de Rham cohomology $\mathrm{R}\Gamma_{\dR}(X)$ induced by the isomorphism \eqref{intro-Mazur-isom} as follows: we view the complex $\mathrm{R}\Gamma_{\dR}(X)$, equipped with the Hodge filtration $F^{\cdot}$, 
as an object $\cM$ of the derived category $\cD_{qc}(\bA^1/\bG_m)$ of quasi-coherent sheaves on $\bA^1/\bG_m$, where $\bA^1$ is considered as a $p$-adic formal scheme over $W(k)$ with the standard $\bG_m$-action. The isomorphism \eqref{intro-Mazur-isom} can be rewritten as\footnote{Later in the text, we will write $\bA^1_-$ with coordinate given by $v_-$, instead of simply $v$.} 
\begin{equation}\label{intro-isom-M-fibrewise}
F^*(\cM_{v=p})\xrightarrow[]{\sim}\cM_{v=1}.
\end{equation}
Here $\cM_{v=a}$ denotes the (derived) fibre of $\cM$ at $\Spf W(k)\xrightarrow{a}\bA^1/\bG_m$ and $F:W(k)\to W(k)$ is the Frobenius. 
We define and denote by $\mathscr{DMF}^{\mathrm{big}}(W(k))$ the stable $\infty$-category consisting of $\cM\in \cD_{qc}(\bA^1/\bG_m)$ equipped with an isomorphism \eqref{intro-isom-M-fibrewise}. We refer to it as the category of \textit{big Fontaine-Laffaille modules} (see Definition \ref{Defn-big-DMF-category}). 

For any integer $n\geq 0$, we denote by $\mathscr{DMF}_{[0,n]}^{\mathrm{big}}(W(k))\subset \mathscr{DMF}^{\mathrm{big}}(W(k))$ the full subcategory consisting of objects whose Hodge filtration ranges from $0$ to $n$\footnote{{\it i.e.},~$F^m\to F^{m-1}$ are quasi-isomorphisms for $m\leq 0$ and $F^{\ell}$ are acyclic for $\ell >n$.}. The preceding construction describes a contravariant functor from the category $\widehat{\mathrm{Sm}}^{\leqslant p-1}_{W(k)}$ of smooth formal schemes of dimension $\leq p-1$ to $\mathscr{DMF}^{\mathrm{ big}}_{[0,p-1]}(W(k))$:
\begin{equation}\label{eqn:intro-functor-Sm-to-DMFbig}
    \widetilde{\Phi}_{\mathrm{Maz}}:\widehat{\mathrm{Sm}}^{\leqslant p-1}_{W(k)}\to \mathscr{DMF}^{\mathrm{ big}}_{[0,p-1]}(W(k)).
\end{equation}

The subcategory $\mathscr{MF}(W(k))$ of $\mathscr{DMF}^{\mathrm{big}}(W(k))$ with $\cM\in \Coh(\bA^1/\bG_m)$ is the category of finitely-generated (``small'') Fontaine-Laffaille modules studied in \cite[\S 1.5]{Fontaine-Laffaille} and in \cite[\S 1.3]{wintenberger}. 
In particular, it is proven in {\it loc. cit.} that $\mathscr{MF}(W(k))$ is an abelian category.  
The original work \cite{Fontaine-Laffaille} does not use the language of sheaves on stacks; see Remark \ref{remark:comparing-FL82-with-our-definition} for a comparison of the original definition in \cite{Fontaine-Laffaille} and our geometric reinterpretation. 
In Proposition \ref{FL-lemma} we show that the bounded derived category  $\cD^b(\mathscr{MF}(W(k)))$ of $\mathscr{MF}(W(k))$ is a full subcategory of $\mathscr{DMF}^{\mathrm{big}}(W(k))$ consisting of bounded complexes with coherent cohomology.

Functor \eqref{eqn:intro-functor-Sm-to-DMFbig} is an input for various comparison theorems  in $p$-adic Hodge theory. 
In particular,
for $p>2$, the subcategory $\mathscr{MF}^{\mathrm{tors}}_{[0,1]}(W(k))$ of $p$-power-torsion objects in $\mathscr{MF}_{[0,1]}(W(k))$ is equivalent to the category of commutative finite flat $p$-group schemes over $W(k)$ (see \cite{Fontaine-Laffaille})\footnote{
This statement is not true for $p=2$. In fact, the category of commutative finite flat $2$-group-schemes over $\bZ_2$ is not an abelian category: for example, the nontrivial homomorphism $\bZ/2\bZ\to\mu_2$ has trivial kernel and cokernel, but it is not an isomorphism.}. 
In \cite{Fontaine-Laffaille}, an exact functor $\mathcal{T}_{\et}^{FL}$ was constructed
from $\mathscr{MF}_{[0,p-2]}(W(k))$ to the category of $\bZ_p$-linear representations of the absolute Galois group $G_K$ of $K:=\mathrm{Frac}(W(k))$. A central result of \cite[Theorem 5.3]{Faltings-crystalline} is the following commutative diagram involving the category $\mathrm{SmPr}^{\leqslant p-2}_{W(k)}$ of smooth proper schemes over $W(k)$ of dimension $\leq p-2$, 
\[\begin{tikzcd}
&\cD^b(\mathscr{MF}_{[0,p-2]})\arrow[]{dd}{\mathcal{T}_{\et}^{FL}}\\
    \mathrm{SmPr}^{\leqslant p-2}_{W(k)}\arrow[]{ru}{\eqref{eqn:intro-functor-Sm-to-DMFbig}}\arrow[]{rd}[swap]{X\mapsto \mathrm{R}\Gamma_{\et}(X\otimes\overline{K},\bZ_p)}&\\
    &\cD^b(\mathrm{Rep}_{G_K}(\bZ_p))
\end{tikzcd}\]

\subsection{Prismatic $F$-gauges}
Given a bounded prism $(A,I,\delta)$ equipped with a map $W(k)\to A/I$,  Bhatt and Scholze constructed 
in \cite{Bhatt-Scholze-prisms} a new cohomology theory $X \mapsto R\Gamma_{\Prism}(X/A)$ for $p$-adic formal schemes over $W(k)$
with values in the derived category $\cD(A)$ of $A$-modules.  Bhatt--Lurie \cite{Bhatt-Lurie-2022absolute} and 
Drinfeld \cite{drinfeld2022prismatization} advanced this theory further  by organizing all prismatic cohomology theories, equipped with the prismatic Nygaard filtrations and the Frobenii,
 into a single contravariant functor 
\begin{equation}\label{eqn:intro-functor-SchWk-to-Dqcsyn}
    \mathcal{H}_{syn}: \widehat{\mathrm{Sch}}_{W(k)}\to \cD_{qc}(W(k)^{syn})
\end{equation}
from the category $\widehat{\mathrm{Sch}}_{W(k)}$ of $p$-adic formal schemes over $W(k)$ to a stable $\infty$-category $\cD_{qc}(W(k)^{syn})$. Moreover, the latter category can be interpreted as the derived category of quasi-coherent sheaves on a stack $W(k)^{syn}$, called the \textit{syntomification} of $W(k)$. The prismatic cohomology theory attached to $(A,I,\delta)$ arises from \eqref{eqn:intro-functor-SchWk-to-Dqcsyn} via the pullback along a map $\Spf A\to W(k)^{syn}$ constructed using $(I, \delta)$. The syntomification is defined  in \cite{drinfeld2022prismatization}, \cite{bhatt-lecture-notes} for any (nice) 
$p$-adic formal scheme $X$; the functor \eqref{eqn:intro-functor-SchWk-to-Dqcsyn} carries a $p$-adic formal scheme $X$ over $W(k)$ to the derived direct image of the structure sheaf under the morphism of stacks
\begin{equation}
    X^{syn} \to W(k)^{syn}
\end{equation}
constructed by functoriality. 

In this article, we address the following two questions: (notations as in \eqref{eqn:intro-functor-Sm-to-DMFbig})
\begin{enumerate}
    \item Given a smooth scheme $X$ of dimension $\leq p-1$, can one recover the (big) Fontaine-Laffaille module $\widetilde{\Phi}_{\mathrm{Maz}}(X)$ from $\mathcal{H}_{syn}(X)$?
    \item Does $\mathcal{H}_{syn}(X)$ contain more information than $\widetilde{\Phi}_{\mathrm{Maz}}(X)$?
\end{enumerate}
\subsection{Main results}
For a non-negative integer $n$, we denote by $\cD_{qc,[0,n]}(W(k)^{syn})$ the full subcategory of $\cD_{qc}(W(k)^{syn})$ spanned by objects with Hodge-Tate weights (see Definition \ref{defn:HT-weights}) in the interval $[0,n]$. In particular, the category $\cD_{qc,[0,n]}(W(k)^{syn})$ contains objects of the form $\mathcal{H}_{syn}(X)$ for a smooth $p$-adic formal scheme $X$ of dimension $\leq n$. The main result (Theorem \ref{main-thm}) of our article is the following:
\begin{theo}\label{intro-main-thm}
    There is a functor 
    \begin{equation}\label{eqn:intro-main-thm-PhiMaz}
    \Phi_{\mathrm{Maz}}:\cD_{qc,[0,p-1]}(W(k)^{syn}) \xrightarrow{} \mathscr{DMF}^{\mathrm{big}}_{[0,p-1]}(W(k))\end{equation}
    that induces an equivalence of $\infty$-categories
    \begin{equation}\label{eqn:intro-main-thm-Phi-equiv}
    \cD_{qc,[0,p-2]}(W(k)^{syn}) \xrightarrow{\sim} \mathscr{DMF}^{\mathrm{big}}_{[0,p-2]}(W(k)).\end{equation}
    \end{theo}
In fact, the composition of $\Phi_{\mathrm{Maz}}$ with $\mathcal{H}_{syn}$ is isomorphic to $\widetilde{\Phi}_{\mathrm{Maz}}$ (see Remark \ref{remark:Phi-Maz-tilde-isom-PhiMaz-Hsyn}). 
The functor $\Phi_{\mathrm{Maz}}$ (see \eqref{functor}) restricts to an equivalence 
    \begin{equation}\label{intro-perfect-complexes-PhiMaz}
    \Perf_{[0,p-2]}(W(k)^{syn}) \xrightarrow{\sim} \cD^b\left(\mathscr{MF}_{[0,p-2]}(W(k))\right),\end{equation}
    where the left-hand side in \eqref{intro-perfect-complexes-PhiMaz} is the full subcategory $\Perf_{[0,p-2]}(W(k)^{syn})$ of the category of perfect complexes $\Perf(W(k)^{syn})$ formed of objects of Hodge-Tate weights in $[0,p-2]$. Moreover, the functor \eqref{intro-perfect-complexes-PhiMaz} is $t$-exact, {\it i.e.},~it restricts to an equivalence between abelian subcategories\footnote{By definition, the category $\Coh_{[0,p-2]}(W(k)^{syn})$ is a full subcategory of the category $\Coh(W(k)^{syn})$ of coherent sheaves on $W(k)^{syn}$ that consists of objects lying in $\cD_{qc,[0,p-2]}(W(k)^{syn})$.} 
    \begin{equation}\label{intro-coherent-PhiMaz}
     \Coh_{[0,p-2]}(W(k)^{syn})\xrightarrow{\sim}\mathscr{MF}_{[0,p-2]}(W(k)).  
    \end{equation}
We also show that $p$-torsion objects in $\Coh_{[0,p-2]}(W(k)^{syn})$ are vector bundles on $W(k)^{syn}\otimes\bF_p$ (see Remark \ref{remark:vector-bundles}). 
We remark that $\Phi_{\mathrm{Maz}}$ is not an equivalence (see Remark \ref{remark:Phi-Maz-not-an-equivalence}).

Let us now describe some of the earlier works related to our main theorem. 

Firstly, in \cite{caisliu}, Cais and Liu constructed a functor from the category of $p$-torsionfree Fontaine--Laffaille modules with Hodge-Tate weights 
in $[0,p-2]$ to the category of Breuil--Kisin modules.~The abelian category of coherent prismatic $F$-gauges also admits a functor to the category of Breuil--Kisin modules.~We expect that our functors agree. 

Secondly, using results from \cite{Bhatt_Scholze_Galois}, Bhatt and Lurie established an equivalence of categories 
$$\mathcal{T}_{\et}^{BL}: \Coh^{\text{refl}}(W(k)^{syn}) \iso \mathrm{Rep}_{G_K}^{crys}(\bZ_p)$$
between the category of \textit{reflexive} coherent sheaves on $W(k)^{syn}$ and
the category of $\bZ_p$-lattices in $\bQ_p$-linear crystalline representations of the absolute Galois group $G_K$ of  $K:=\Frac (W(k))$ (see \cite[Theorem 6.6.13]{bhatt-lecture-notes}).
On the other hand, Fontaine and Laffaille (\cite[\S 7.14]{Fontaine-Laffaille}) constructed an equivalence  of categories  $\mathcal{T}_{\et}^{FL}$ between 
the full subcategory $\mathscr{MF}_{[0,p-2]}^f(W(k))\subset  \mathscr{MF}_{[0,p-2]}(W(k))$ of $p$-torsionfree Fontaine--Laffaille modules
and the full subcategory  $\mathrm{Rep}_{G_K, [0, p-2]}^{crys}(\bZ_p)\subset \mathrm{Rep}_{G_K}^{crys}(\bZ_p)$,
whose objects are lattices in crystalline representations with Hodge-Tate weights lying  in $[0,p-2]$\footnote{The essential surjectivity of $\mathcal{T}_{\et}^{FL}$ is not stated explicitly in \cite{Fontaine-Laffaille}. The observation that the results of {\it loc. cit.} imply the essential surjectivity is due to Breuil
\cite[Proposition 3]{breuil}.}.
In Remark \ref{rem:Lattices in crystalline representations} we prove that the Bhatt--Lurie and Fontaine--Laffaille functors agree via \eqref{intro-coherent-PhiMaz}:
\begin{equation}
    \mathcal{T}_{\et}^{FL} \iso \mathcal{T}_{\et}^{BL} \circ \Phi_{\mathrm{Maz}}^{-1}: \mathscr{MF}_{[0,p-2]}^f(W(k)) \to \mathrm{Rep}_{G_K, [0, p-2]}^{crys}(\bZ_p).
\end{equation}
In particular, the Bhatt--Lurie equivalence carries every object of $\mathrm{Rep}_{G_K, [0, p-2]}^{crys}(\bZ_p)$ to a \textit{vector bundle} over 
$W(k)^{syn}$, and not just a reflexive sheaf.

Thirdly, Theorem F (2) from \cite{Antieau-Mathew-Morrow-Nikolaus} suggested that the category of Fontaine--Laffaille
modules with small Hodge-Tate weights is a derived full subcategory of the category of prismatic $F$-gauges.~The latter result implies, in particular, that for a smooth $p$-adic formal 
scheme $X$ of dimension $\leq p-2$, the syntomic cohomology $R\Gamma(W(k)^{syn}, \mathcal{H}_{syn}(X)\{i\})$ is isomorphic, for $i\leq p-2$, to 
$\mathrm{RHom}(W(k), \widetilde{\Phi}_{\mathrm{Maz}}(X)(i))$ computed in the the category of Fontaine--Laffaille  modules.
The original proof from \cite{Antieau-Mathew-Morrow-Nikolaus} uses methods from algebraic topology.~An algebraic proof of the mod $p$ version is explained in \cite[Remark 6.5.15]{bhatt-lecture-notes}.~Our argument is closer in spirit to the latter.  We refer the reader to Remark \ref{rem:Syntomic cohomology in small weights} for further details and a generalization.

Lastly, a version of \eqref{intro-perfect-complexes-PhiMaz} with \textit{rational coefficients} holds in all  Hodge-Tate  weights. Indeed, using the results from \cite[\S 6]{bhatt-lecture-notes}, one can prove \cite{bhatt-private} that the category $\Perf(W(k)^{syn})\otimes \bQ_p$ of perfect prismatic $F$-gauges up to isogeny is equivalent to the bounded derived category of crystalline representations (with arbitrary Hodge-Tate  weights).
On the other hand, combining \cite[Proposition 7.8]{Fontaine-Laffaille} with \cite[Theorem A]{Colmez_Fontaine}, the category of crystalline representations is equivalent to the category of Fontaine--Laffaille modules up to isogeny. This yields an equivalence 
 \begin{equation}\label{intro-perfect-complexes-PhiMaz-isogeny}
    \Perf(W(k)^{syn})\otimes \bQ_p \xrightarrow{\sim} \cD^b\left(\mathscr{MF}(W(k))\right) \otimes \bQ_p.\end{equation}
    
\subsection{Overview of the proof}
Recall from \cite{bhatt-lecture-notes} that the construction of syntomification can be applied to any $p$-adic formal scheme over $\bZ_p$. In particular, we can apply it to $\Spec k$; by functoriality, we have a morphism $\mathfrak{p}_{\mathrm{cris}}: k^{syn}\to W(k)^{syn}$, and thus the pullback functor $\mathfrak{p}_{\mathrm{cris}}^*:\cD_{qc}(W(k)^{syn})\to\cD_{qc}(k^{syn})$. The target category has a concrete linear-algebraic description due to \cite{Ekedahl-diagonal-complexes-Fgauges, Fontaine-Jannsen-F-gauges, bhatt-lecture-notes}. In geometric context, $\mathfrak{p}_{\mathrm{cris}}^*\mathcal{H}_{syn}(X)$ recovers the crystalline cohomology of $X\otimes \bF_p$, equipped with the Nygaard filtration and divided Frobenii. In addition, one has a morphism $\mathfrak{p}_{\bar{\dR}}:\bA^1/\bG_m\to W(k)^{syn}$, which in geometric context recovers the de Rham cohomology equipped with the Hodge filtration. We consider the Cartesian square
\begin{equation}\label{diagram:intro-fibre-product-commdiagram}
\begin{tikzcd}
    &k^{syn}\times_{W(k)^{syn}}\bA^1/\bG_m\arrow[]{ld}{}\arrow[]{rd}{}&\\
    k^{syn}\arrow[]{rd}{}&&\bA^1/\bG_m\arrow[]{ld}{}\\
    &W(k)^{syn}&
\end{tikzcd}
\end{equation}
and the induced functor 
\begin{equation}\label{eqn:intro-map-to-fibre-product-over-Drinfeld-frakD}
\cD_{qc}(W(k)^{syn})\to \cD_{qc}(k^{syn})\times_{\cD_{qc}(\mathfrak{D})} \cD_{qc}(\bA^1/\bG_m),
\end{equation}
where $\mathfrak{D}$ is the fibre product in diagram \eqref{diagram:intro-fibre-product-commdiagram}.~A key input for our construction of the functor $\Phi_{\mathrm{Maz}}$ is an explicit description of the fibre product $\mathfrak{D}$ communicated to us by Drinfeld (see~\S\ref{subsection:description-fibre-product}\footnote{\S \ref{subsection:Drinfeld-lemma} contains a description of a certain closed substack of $\mathfrak{D}$ which is sufficient for all our applications.}).

To describe $\mathfrak{D}$, recall from the end of \S \ref{subsec:preliminary-syntomification-section} that $k^{syn}$ is obtained from\footnote{Note that $A$ is not $p$-complete. By our convention, we write $\Spf A$ for the $p$-adic formal scheme $\colim_n  \Spec A/p^n$.} $k^\cN = \Spf A/ \bG_m$, where $A= W(k)[v_-, v_+]/(v_-v_+-p)$,
by identifying two open points $\begin{tikzcd}j_{\pm}:\Spf W(k) \arrow[shift left]{r}{}\arrow[shift right]{r}[swap]{}&k^{\cN}\end{tikzcd}.$  Let $B^{\flat}$ be the  PD-envelope of the ideal  $(v_+)$ in $A$. In Lemma \ref{lem:3.10} and Proposition \ref{prop:fibre-product-tildefrakT} we show that
 $\mathfrak{D}$ is isomorphic to $\Spf B^\flat /\bG_m$ such that the map $\mathfrak{D} \to k^{syn} \times \bA^1/\bG_m$
 is given by the natural algebra homomorphisms $A\to B^\flat$ and $W(k)[v]\xrightarrow{v\mapsto v_-} B^\flat.$ 
 We also consider a ``baby'' version $B$ of 
 $B^\flat$ defined to be the quotient of $B^\flat$ by the $p$-power-torsion elements (see the beginning of \S \ref{subsection:Drinfeld-lemma} for a more concrete description of $B$).  
 Our main geometric input is the functor 
\begin{equation}\label{eqn:intro-map-to-fibre-product-over-Drinfeld-B}
\cD_{qc, [0, \infty]}(W(k)^{syn})\to \cD_{qc, [0, \infty]}(k^{syn})\times_{\cD_{qc}(\Spf B/\bG_m )} \cD_{qc, [0, \infty]}(\bA^1/\bG_m).
\end{equation}
obtained from  \eqref{eqn:intro-map-to-fibre-product-over-Drinfeld-frakD} by applying $\cD_{qc}(\mathfrak{D}) \to \cD_{qc}(\Spf B/\bG_m )$.
 The target category has a concrete linear-algebraic description.~In geometric context, functor \eqref{eqn:intro-map-to-fibre-product-over-Drinfeld-frakD} extracts from the $F$-gauge $\mathcal{H}_{syn}(X)$ the information about the Nygaard filtered crystalline cohomology with its divided Frobenii, the Hodge filtered de Rham cohomology ({\it i.e}.,~the cohomology theories known before the invention of prismatic cohomology), together with a comparison isomorphism relating Nygaard and Hodge filtrations (that, in particular, recovers \eqref{eqn:intro-stupid-filtration}). See Remark \ref{remark:cris-Hodge-filtration-recovered-from-Nygaard} for more details. 
We check (see Remark \ref{remark-construction-functor-Psi-Mazur}) that the functor from
 $\cD_{qc, [0, p-1]}(k^{syn})\times_{\cD_{qc}(k^\cN)} \cD_{qc, [0, p-1]}(\bA^1/\bG_m) $
 to 
 $ \cD_{qc, [0, p-1]}(k^{syn})\times_{\cD_{qc}(\Spf B/\bG_m)} \cD_{qc, [0, p-1]}(\bA^1/\bG_m)$, induced by
 the map $\Spf B/\bG_m \to k^\cN$,
 is an equivalence and identifies the left-hand side with $\mathscr{DMF}^{\mathrm{big}}_{[0,p-1]}(W(k))$. 
This defines $\Phi_{\mathrm{Maz}}$ from \eqref{eqn:intro-main-thm-PhiMaz}. The proof of equivalence \eqref{eqn:intro-main-thm-Phi-equiv} uses a concrete description of the reduced locus $W(k)^{syn}_{red}$ of $W(k)^{syn}$ from~\cite{drinfeld2022prismatization} that leads to an equivalence $\cD_{qc,[0,p-1]}(W(k)^{syn}_{red}) \xrightarrow{\sim} \mathscr{DMF}^{\mathrm{big}}_{[0,p-1]}(W(k))\otimes \bF_p$.~Finally, via a deformation theory argument (suggested to us by Bhatt), we prove that the pullback along $W(k)^{syn}_{red} \hookrightarrow W(k)^{syn}\otimes \bF_p$ induces an equivalence 
\begin{equation}\label{eq:restriction.intro}
 \cD_{qc,[0,p-2]}(W(k)^{syn}\otimes \bF_p) \iso \cD_{qc,[0,p-2]}(W(k)^{syn}_{red}).   
\end{equation}
The last step involves cohomology computations from \cite[\S 6]{bhatt-lecture-notes}. We infer that \eqref{eqn:intro-main-thm-Phi-equiv} induces an equivalence after taking the tensor product with $\bF_p$, which, by $p$-completeness, completes the proof.

Though functor  \eqref{eqn:intro-map-to-fibre-product-over-Drinfeld-B} is not an equivalence in large weights, it can be used to construct some interesting invariants of $F$-gauges. Following a suggestion of Drinfeld's, we define in 
Remark \ref{remark:Mazur-module} a functor from the right-hand side of \eqref{eqn:intro-map-to-fibre-product-over-Drinfeld-B}
to a category of Mazur modules $ \mathrm{Mod}_{\mathrm{Maz}}(W(k))$, giving  
$$ {\bf \Phi}_{\mathrm{Maz}}:   \cD_{qc,[0, \infty]}(W(k)^{syn})\to \mathrm{Mod}_{\mathrm{Maz}}(W(k)).$$
We refer the reader to Remark \ref{remark:Phi-Maz-tilde-isom-PhiMaz-Hsyn} for an interpretation of this functor in geometric context.

As another application of Drinfeld's description of  $\mathfrak{D}$, we construct in Remark \ref{remark:refined-crystalline-Frob-categories} 
a refinement of the crystalline Frobenius functor on the category of crystals on a smooth $p$-adic formal scheme $X$ over $\Spf W(k)$. This construction is related to the Cartier transform studied in \cite{Ogus-Vologodsky} and especially to its $p$-adic lift constructed in \cite{Xu}.  

\subsection{Further directions} In a sequel we plan to extend our main results to the case of $F$-gauges over a smooth $p$-adic formal scheme $X$ over $\Spf W(k)$.~In \cite[pp 30--38]{Faltings-crystalline}, Faltings defined (at least for $p>2$) an abelian category\footnote{see Theorem 2.3 {\it loc. cit.}. Faltings' notation for this category is $\mathscr{MF}_{[0,p-1]}^{\nabla}(X)$.}   $\mathscr{MF}_{[0,p-1]}(X)$ of Fontaine-Laffaille modules over $X$.~We expect that there is an exact fully faithful functor $\mathscr{MF}_{[0,p-2]}(X) \mono \Coh (X^{syn})$; however, this functor is not {\it derived} fully faithful\footnote{{\it i.e.},~the induced functor $\cD^b\left(\mathscr{MF}_{[0,p-2]}(X)\right) \to \cD_{qc}(X^{syn})$ is not fully faithful.}: the category $\mathscr{MF}_{[0,p-2]}(X)$ is too small to capture the syntomic cohomology even in small weights.~For example, the pullback functor $\mathscr{MF}_{[0,p-2]}(W(k)) \to \mathscr{MF}_{[0,p-2]}(\bP^1)$ is an equivalence.~In Remark \ref{The_category_of_Fontaine_Laffaille_modules_over_a_scheme}, we propose a definition for $\mathscr{DMF}^{\mathrm{big}}_{[0,p-1]}(X)$.~We expect to prove that the corresponding subcategory of the latter is equivalent to $\cD_{qc,[0,p-2]}(X^{syn})$.

\subsection{Notations}
Fix a prime number $p$. For any stack $X$, we denote by $\cD_{qc}(X)$ the stable $\infty$-category of quasi-coherent sheaves on $X$. For a stack $X$ over $W(k)$, denote by $X^{\wedge}$ the corresponding $p$-adic formal stack, {\it i.e.},~for any ring $R$, we have that $X^{\wedge}(R)=X(R)$ if $p$ is nilpotent in $R$, and $X^{\wedge}(R)=\varnothing$ otherwise.~In other words, $X^{\wedge}=X\times\Spf\bZ_p$. 
In particular, $\cD_{qc}(X^{\wedge})=\underset{n}{\varprojlim} \cD_{qc}(X\otimes\bZ/p^n\bZ)$. 

For a commutative ring $R$, we will write $\Spf R$ for the formal spectrum of the $p$-completion of $R$ with respect to the ideal $(p)$, {\it i.e.},~$\Spf R=(\Spec R)^{\wedge}$. 
For a finitely generated ideal $I\subset R$, we shall write $\widehat{\cD}(R)$ for the full subcategory of $\cD(R)$ spanned by $I$-complete objects (see for example \cite[\S 1.9]{Bhatt-Lurie-2022absolute}).

\subsection*{Acknowledgements.} This paper owes its existence to Bhargav Bhatt and Vladimir Drinfeld. Bhatt suggested to us that the restriction functor \eqref{eq:restriction.intro} is an equivalence and gave us numerous hints during the initial stage of the project. Drinfeld explained to us his description of the fiber product $\mathfrak{D}$ in \eqref{diagram:intro-fibre-product-commdiagram}, which is a key input in this work, and we thank him for allowing us to include his result in this paper (see \S\ref{subsection:Drinfeld-lemma} and \S \ref{subsection:description-fibre-product}). 
We would also like to thank Akhil Mathew, Jacob Lurie, Arthur Ogus, and Alexander Petrov for their interest in our work and numerous discussions. In particular, Lurie shared with us his perspective on Theorem \ref{intro-main-thm}. He also suggested an alternative description of $\mathfrak{D}$ explained in \S\ref{subsection:description-fibre-product}.  Mathew supplied a proof of a key assertion in Remark  \ref{rem:Lattices in crystalline representations}.

G.T.~was supported by NSF Grant DMS \#1801689.~V.V.~was supported by a Simons Foundation Investigator Grant No.~622511 through Bhargav Bhatt.~V.V.~would like to thank the University of Chicago and Princeton University where the work has been done.~Y.X.~was supported by the U.S. National Science Foundation under Award No.~2202677.

\section{Preliminaries}

\subsection{Quasi-coherent sheaves on $\bA^1/\bG_m$}\label{subsection:QCoh-on-A1modGm}
We recall the standard dictionary between the derived category of quasi-coherent sheaves on $\bA^1/\bG_m$ and the filtered derived category of graded $W(k)$-modules.~See \cite[\S 2.2]{bhatt-lecture-notes} for additional details. 

Let $z$ denote the coordinate function on $\bG_m$.~Consider the standard action of $\bG_m$ on $\bA^1=\Spec W(k)[v_+]$ given by 
    $\bG_m\times \bA^1\to \bA^1, v_+\mapsto z\otimes v_+$.~We shall denote the $p$-adic formal scheme associated to $\bA^1$ with the above $\bG_m$-action by $\bA^1_+$, {\it i.e.},~$\bA^1_+:=\Spf W(k)[v_+]:= (\Spec W(k)[v_+])^{\wedge}$.~Similarly, we define an action of $\bG_m$ on $\bA^1_-:=\Spf W(k)[v_-]$ given by $v_-\mapsto z^{-1}\otimes v_-$. 
We write $\bA^1_{\pm}/\bG_m$ for the quotient $p$-adic formal stacks over $B\bG_m$. Thus,
for a $p$-nilpotent $W(k)$-algebra $R$, the groupoid of $R$-points of $\bA^1_{+}/\bG_m$  (resp.~$\bA^1_{-}/\bG_m$) classifies pairs $(L, v_+)$ (resp. $(L, v_-)$), where $L$ is a line bundle over $S:=\Spec R$ and $v_+: \cO_S \to L$ is a section of $L$ (resp. $v_-: L \to \cO_S$ is a section of $L^*$). We denote by  $(\cO(1), v_+)$ (resp. $(\cO(1), v_-)$ ) the universal 
pair over $\bA^1_{+}/\bG_m$ (resp. $\bA^1_{-}/\bG_m$). We have a map  $\bA^1_{\pm}/\bG_m \to B\bG_m$ classifying the line bundle $\cO(1)$
\footnote{Note that $\bA^1_+/\bG_m$ is isomorphic to $\bA^1_-/\bG_m$ as abstract stacks, but not as stacks over $B\bG_m$.}. 
We identify the derived category of quasi-coherent sheaves on $\bA^1_{\pm}/\bG_m$ with the derived category of $p$-complete graded modules over $W(k)[v_{\pm}]$\footnote{A graded complex $\oplus_i F^i$ of 
        $W(k)[v_{-}]$-modules is said to be $p$-complete if each $F^i$ is $p$-complete.}: 
\begin{equation}\label{identifying-CohA1modGm-with-modules}
\cD_{qc}(\bA^1_{\pm}/\bG_m)\simeq \widehat{\cD}_{gr}(W(k)[v_{\pm}]).
\end{equation}
The functor $\cD_{qc}(\bA^1_-/\bG_m)\to \widehat{\cD}_{gr}(W(k)[v_-])$ is given by 
\begin{equation}\label{identifying-CohA1modGm-with-modules-minus}
\cM\mapsto (\oplus F^i, F^i\xrightarrow{v_-}F^{i-1}, \deg F^i=i),
\end{equation} 
where $F^i:=\mathrm{R}\Gamma(\bA^1_-/\bG_m,\cM\otimes\cO(i))$ with $\cO(i)$ the line bundle pulled back from $B\bG_m$. The map $v_-: F^i\to F^{i-1}$ comes from the global section $v_-\in \Gamma(\bA^1_-/\bG_m,\cO(-1))$. 
Similarly, the functor $\cD_{qc}(\bA^1_+/\bG_m)\to \widehat{\cD}_{gr}(W(k)[v_+])$ is given by 
\begin{equation}\label{identifying-CohA1modGm-with-modules-plus}
\cM\mapsto (\oplus G_i, G_i\xrightarrow{v_+}G_{i+1},\deg G_i=i)
\end{equation}
where $G_i:=\mathrm{R}\Gamma(\bA^1_+/\bG_m,\cM\otimes\cO(i))$. 
\begin{df}\label{Defn-Dqc-ab-A1modGm-plusminus}
    For $a,b\in\bZ\cup \{+\infty,-\infty\}$ such that $a\leq b$, let $\cD_{qc,[a,b]}(\bA^1_{\pm}/\bG_m)$ be a full subcategory of $\cD_{qc}(\bA^1_{\pm}/\bG_m)$ formed of objects $\cM$ such that, under the identification \eqref{identifying-CohA1modGm-with-modules-plus} (resp.~\eqref{identifying-CohA1modGm-with-modules-minus}), $G_j$ (resp.~$F^j$) is acyclic for $j<a$ (resp.~$j> b$) and $v_+:G_i\to G_{i+1}$ (resp.~$v_-:F^i\to F^{i-1}$) is a quasi-isomorphism for $i\geq b$ (resp.~$i\leq a$). Similarly, one defines a full subcategory $\Coh_{[a,b]}(\bA^1_{\pm}/\bG_m)\subset\Coh(\bA^1_{\pm}/\bG_m)$.  
\end{df}

\begin{rem}
    Consider the map $B\bG_m\to \bA^1_{\pm}/\bG_m$ given by inclusion of the origin into $\bA^1_{\pm}$. We identify $\cD_{qc}(B\bG_m)$ with the derived category of $p$-complete graded complexes of $W(k)$-modules. For any $\cM\in \cD_{qc,[a,b]}(\bA^1_{\pm}/\bG_m)$, its restriction to $B\bG_m$ is acyclic in grading degrees outside of $[a,b]$. Thus we have a well-defined functor 
    \[\cD_{qc,[a,b]}(\bA^1_{\pm}/\bG_m)\to \cD_{qc,[a,b]}(B\bG_m).\] 
\end{rem}

\begin{df}\label{effective-Coh}
A sheaf $\cM\in\cD_{qc}(\bA^1_{\pm}/\bG_m)$ is \textit{effective} if $\cM\in\cD_{qc,[0,+\infty]}(\bA^1_{\pm}/\bG_m)$. 
\end{df}

\subsection{Fontaine-Laffaille modules}\label{subsection:Fontaine-Laffaille-modules}
Denote by $F: W(k)\to W(k)$ the Frobenius endomorphism. For $a\in W(k)$, denote by $i_a: \Spf W(k)\mono \bA^1_- \epi \bA^1_-/\bG_m$ the point given by equation $v_{-}=a$. Given $\cM\in \cD_{qc}(\bA^1_-/\bG_m)$, we denote by $\cM_{v_-=a}:=i_a^* \cM\in \cD_{p\text{-}comp}(W(k))$  the {\it derived} pullback;
for  $\cM\in\Coh(\bA^1_-/\bG_m)$, we write $ L_0 i_a^* \cM \in \Mod(W(k))$ for the {\it non-derived} fiber.  
\begin{df}\label{Defn:FLmodules-as-sheaves}
(cf.~\cite[\S 1.5]{Fontaine-Laffaille}, \cite[\S 1.3]{wintenberger})
    A finitely generated (``small'') Fontaine--Laffaille module over $W(k)$ is $\cM\in\Coh(\bA^1_-/\bG_m)$ together with an isomorphism between the fibres
    \[\varphi: F^*( L_0 i_p^* \cM)\xrightarrow{\sim}  L_0 i_1^* \cM.\] 
Morphisms between Fontaine--Laffaille modules are morphisms between coherent sheaves on $\bA^1_-/\bG_m$ compatible with the isomorphisms $\varphi$. We denote the category of Fontaine--Laffaille modules over $W(k)$ by $\mathscr{MF}(W(k))$. 
\end{df}

\begin{rem}\label{remark:comparing-FL82-with-our-definition}The full subcategory $\mathscr{MF}^{\mathrm{tor}}(W(k))$ of $\mathscr{MF}(W(k))$ spanned by $p$-power-torsion objects is equivalent to the category $\underline{MF}^f_{tor}$ introduced in  \cite[\S 1.5]{Fontaine-Laffaille}.
However, the definition in \cite{Fontaine-Laffaille} does not use the language of stacks. To compare Definition \ref{Defn:FLmodules-as-sheaves} with the one given in \cite{Fontaine-Laffaille}, observe an equivalence between $\Coh(\bA^1_-/\bG_m)$ and the category of finitely generated graded modules $\bigoplus F^i$ over $W(k)[v_-]$ as explained in \S \ref{subsection:QCoh-on-A1modGm}. Under this equivalence, the fibre 
$ L_0 i_1^* \cM$  of   $\cM$ over the point $v_-=1$ is sent to $F^{-\infty}:=\underset{v_-}{\varinjlim}\,\ F^i$. Giving a morphism $\varphi: F^*( L_0 i_p^* \cM)\to  L_0 i_1^* \cM$ is equivalent to giving a collection of $F$-linear maps $\varphi_i: F^i\to F^{-\infty}$ such that $\varphi_{i-1}v_-=p\varphi_i$. Assume that $\cM$ is of $p$-power torsion. Then by ~\cite[Theorem 2.1]{Faltings-crystalline} (see also Proposition \ref{FL-lemma} below for the implication in the forward direction), the fact that $\varphi$ is an isomorphism is equivalent to the combination of the following two statements:
\begin{enumerate}
    \item  Each $v_-: F^i\to F^{i-1}$ is a {\it split} injection; and 
    \item $F^{-\infty}=\sum\limits_i\im (\varphi_i)$. 
\end{enumerate}
These two are precisely the conditions given in \cite[\S 1.5]{Fontaine-Laffaille} except for the adjective {\it split} in (1) which is missing in {\it loc. cit}. However, using that the $W(k)$-module $F^{-\infty}$ has finite length, one shows 
(see \cite[Proposition 1.4.1 (ii)]{wintenberger} or \cite[Theorem 2.1]{Faltings-crystalline}) that the splitting condition for $v_-: F^i\to F^{i-1}$ is  redundant.

The whole category  $\mathscr{MF}(W(k))$ can be  described similarly: an object of  $\mathscr{MF}(W(k))$ consists of a finitely generated 
$W(k)$-module $F^{-\infty}$ equipped with an exhaustive filtration by $W(k)$-submodules $F^{\bullet}$ together with 
$F$-linear maps $\varphi_i: F^i\to F^{-\infty}$, $\varphi_{i-1}v_-=p\varphi_i$, such that the conditions (1) and (2) above hold\footnote{In the  nontorsion case the adjective {\it split} in (1) is essential.}. This category was introduced in  \cite[\S 1.3]{wintenberger}.
The category of strongly divisible lattices that appeared in  \cite[\S 7.11]{Fontaine-Laffaille} is equivalent to the full subcategory of $p$-torsionfree objects of  $\mathscr{MF}(W(k))$. 
\end{rem}

\begin{df}\label{Defn-big-DMF-category}
Let $\mathscr{DMF}^{\mathrm{big}}(W(k))$ be the stable $\bZ_p$-linear $\infty$-category formed of objects $\cM\in \cD_{qc}(\bA^1_-/\bG_m)$ together 
with an isomorphism $\varphi:F^*(\cM_{v_-=p})\xrightarrow{\sim}  \cM_{v_-=1}$, {\it i.e.},~we have 
$$\mathscr{DMF}^{\mathrm{big}}(W(k)):=\mathrm{Eq}\begin{tikzcd}\Big(\cD_{qc}(\bA^1_-/\bG_m)\arrow[shift left]{r}{i_1^*}\arrow[shift right]{r}[swap]{F^*\circ i_p^*}&\cD_{qc}(\Spf W(k))\Big)\end{tikzcd},$$
where $i_1^*$ denotes the pullback to the point $v_-=1$ and $F^*\circ i_p^*$ denotes the pullback to the point $v_-=p$ post-composed with the Frobenius.    

\end{df}
Part (1) of the following Proposition, with  $ \mathscr{MF}(W(k))$ replaced by  $\mathscr{MF}^{\mathrm{tor}}(W(k))$, is proven in \cite[\S 1.8]{Fontaine-Laffaille}, and in \cite[Proposition 1.4]{wintenberger}  in general. Part (2) is proven in \cite[Theorem 2.1]{Faltings-crystalline}. 
\begin{pr}\label{FL-lemma}\noindent
\begin{enumerate}
    \item For every $(\cM, \varphi)\in \mathscr{MF}(W(k))$ and every point $i_a: \Spf W(k) \to \bA_-^1/\bG_m$, the derived pullback $Li^*_a\cM$ is concentrated in cohomological degree $0$.~In particular, for every $j\in\bZ$, the map $F^j\xrightarrow{v_-}F^{j-1}$ is injective\footnote{Indeed, $i^*_0\cM$ is isomorphic to $\bigoplus_j \Cone(F^j\xrightarrow{v_-}F^{j-1})$.}. 
  \item The category $\mathscr{MF}(W(k))$ is abelian. Moreover, the functor $\mathscr{MF}(W(k))\to \Coh(\bA_-^1/\bG_m)$ given by $(\cM,\varphi)\mapsto\cM$ is exact. 
    \item Let  $\cD^b(\mathscr{MF}(W(k)))$ be the bounded derived category  of $\mathscr{MF}(W(k))$.
    Then the functor $\cD^b(\mathscr{MF}(W(k)))\to \mathscr{DMF}^{\mathrm{big}}(W(k))$ is fully faithful, and its essential image is formed of pairs $(\cM,\varphi)$ with $\cM\in \cD^b(\Coh(\bA^1_-/\bG_m))$. 
    \end{enumerate}
\end{pr}
    \begin{proof}     
Though the first two assertions are known, for the reader's convenience, we include a proof.

To prove (1) we first verify, by induction on $n$, that if $(\cM,\varphi)$ is a Fontaine-Laffaille module with $p^n \cM=0$, then $L_1 i^*_a \cM =0$.~For $n=1$, we check that the pullback of $\cM$ to $\bA^1_- \otimes k$ is a vector bundle; this would imply the assertion.~Indeed, by $\bG_m$-equivariance it suffices to show $\dim L_0i^*_p \cM=\dim L_0 i^*_1 \cM$\footnote{Indeed, a coherent sheaf $\cF$
on $\bA^1_- \otimes k$ is a vector bundle if, for every closed point $a$ of  $\bA^1_- \otimes k$, the dimension of its fiber over $a$ is equal to the dimension of its fiber over the generic point. If $\cF$ is $\bG_m$-equivariant, then $\cF$ restricted to 
$(\bA^1_- \otimes k)\backslash\{0\}$ is automatically a vector bundle. Thus $\cF$
is a vector bundle if $\dim L_0i^*_0 \cM=\dim L_0i^*_p \cM$ is equal to $\dim L_0i^*_1 \cM$.}. But this follows from the existence of $\varphi$.~For the induction step, set $\cM'= \Ker (\cM \to \cM/p)$ and consider the commutative diagram 
\begin{equation}
\begin{tikzcd}
0\arrow[]{r}{}&F^*L_0i_p^*\cM'\arrow[]{r}{}\arrow[]{d}{\varphi_{\cM'}}&F^*L_0i_p^*\cM\arrow[]{r}{}\arrow[]{d}{\simeq}&F^*L_0i_p^*(\cM/p)\arrow[]{r}{}\arrow[]{d}{\simeq}&0\\
0\arrow[]{r}{}&L_0i_1^*\cM'\arrow[]{r}{}&L_0i_1^*\cM\arrow[]{r}{}&L_0i_1^*(\cM/p)\arrow[]{r}{}&0
\end{tikzcd}
\end{equation}
  Using that $L_1i^*_p (\cM/p)=0$, the rows are exact.~It follows that the left downward arrow is an isomorphism. Thus, $(\cM', \varphi_{\cM'})$ is a  Fontaine-Laffaille module.~Applying the induction hypothesis to $\cM'$ and $\cM/p$ we infer the assertion for $\cM$.
  
To complete the proof of (1), let us check that $L_1 i^*_a \cM =0$, for {\it any}~Fontaine-Laffaille module.~
Indeed, let $M$ denote the space of global sections of $\cM$ pulled back to $\bA^1_-$. Since $\cM$ is coherent, $M$ is a finitely generated module
over $ W(k)[v_-]^{\wedge}$. In particular, $M$ is $p$-complete in the classical (non-derived) sense, 
$M\iso R^0\,\limfrom \, M/p^n \iso R\,\limfrom \, M/p^n$.
The derived fiber $i^*_a \cM$ is isomorphic
to $\Cone (M\rar{v_-  -a} M)$. We have  $\Cone (M\rar{v_-  -a} M)=  R\,\limfrom \,\Cone (M/p^n \rar{v_-  -a} M/p^n)$. Since $\cM /p^n$ has a structure of Fontaine-Laffaille module, $\Cone (M/p^n \rar{v_-  -a} M/p^n)$ is supported in cohomological degree $0$, as was shown in the previous paragraph. Thus, $L_n i^*_a \cM =0$, for $n>0$.

To prove (2), we start with a preliminary construction. Let $\mathscr{MF}'(W(k))$ be the category formed of objects $\cM\in \Coh(\bA^1_-/\bG_m)$ together with a morphism $\varphi:F^* L_0 i^*_p \cM \to L_0i^*_1\cM $.~Since the functor $L_0i^*_1\colon \Coh(\bA^1_-/\bG_m)\to \Coh (\Spf W(k))$ is exact and $L_0 i^*_p$ is right exact, the category  $\mathscr{MF}'(W(k))$ is abelian, and the forgetful functor $\mathscr{MF}'(W(k)) \to  \Coh(\bA^1_-/\bG_m)$ is exact and conservative. The category  $\mathscr{MF}(W(k))$ is a full subcategory of the abelian category $\mathscr{MF}'(W(k))$.~Thus to show that $\mathscr{MF}(W(k))$ is abelian, it suffices to show that, for any morphism $f$ in $\mathscr{MF}(W(k))$, its kernel and cokernel computed in $\mathscr{MF}'(W(k))$ lie in 
    $\mathscr{MF}(W(k))$. The assertion on cokernel is clear from the right exactness properties of $L_0i^*_p$ and $L_0i^*_1$. Let us check that $\Ker f \in \mathscr{MF}(W(k))$.
  Indeed, using the vanishing of $L_1 i^*_p $ for  Fontaine-Laffaille modules 
  we conclude that $L_0i^*_p \Ker(\cM \rar{f} \cM') \iso \Ker(L_0i^*_p \cM \rar{f} L_0i^*_p\cM')$ which implies that $\phi: F^* L_0i^*_p \Ker(\cM \rar{f} \cM') \to L_0i^*_0 \Ker(\cM \rar{f} \cM')$ is an isomorphism.
  
To prove (3): by part (1), for any $\cM \in \mathscr{MF}(W(k))$, we have a canonical isomorphism $F^*i^*_p \cM\simeq F^* i^*_p M$, where the pullbacks are {\it derived}. This defines an exact functor $\mathscr{MF}(W(k)) \to \BFL$ that extends uniquely to an exact functor 
\begin{equation}\label{eq:functor_from_small_to_big}
    \DFL \to \BFL. 
\end{equation}
Let us check that \eqref{eq:functor_from_small_to_big} is fully faithful. It suffices to show that, for any two objects 
$(\cM, \varphi_\cM),  (\cN, \varphi_\cN) \in \mathscr{MF}(W(k))$,
the map of mapping spectra
\begin{equation}\label{eq:fully_faith}
  \Hom_{\DFL} (\cM, \cN) \to \Hom_{\BFL} (\cM, \cN)  
\end{equation}
is an equivalence. First, we check this for a $p$-torsionfree $\cM$. Using part (1), this condition is equivalent to $\cM$ being a locally free $\cO$-module. By definition of 
the equalizer category, the spectrum 
$ \Hom_{\BFL} (\cM, \cN)$ is given by 
\begin{equation}\label{eq:proof_of_prop_on_FL_modules}
  \mathrm{Fib} \left(\Hom_{\cD_{qc}(\bA^1_-/\bG_m)} (\cM, \cN ) \to \Hom_{\cD_{qc}(W(k))} (F^* i^*_p\cM, i^*_1\cN)\right),  
\end{equation}
where the map carries $f\in \Hom_{\cD_{qc}(\bA^1_-/\bG_m)} (\cM, \cN )$ to $i^*_1(f) \circ \varphi_{\cM}-  \varphi_{\cN} \circ F^* i^*_p(f)$. In particular, if $\cM$ is a locally free $\cO$-module, $ \pi_i \Hom_{\BFL} (\cM, \cN)=0$  unless $i=0$ or $-1$: each
mapping spectrum appearing in \eqref{eq:proof_of_prop_on_FL_modules} is isomorphic to an abelian group.
Using \cite[Lemma~4.4]{Bloch-Kato}, the same is true for 
$ \Hom_{\DFL} (\cM, \cN)$. Since $\mathscr{MF}(W(k)) \to \BFL$ is fully faithful and its  essential image is closed under extensions (that is, for any morphism $\cN \to \cM[1]$ in $\BFL$, its fiber is isomorphic to an object of $\mathscr{MF}(W(k))$) it follows that \eqref{eq:fully_faith} is an equivalence for torsion-free $\cM$.
In general, for any $\cM$, we have a fiber sequence $\cM_{tor} \mono \cM \epi \cM/\cM_{tor}$, where $\cM_{tor}$
is a torsion  Fontaine-Laffaile module and $\cM/\cM_{tor}$ is torsion-free. Thus, it remains to prove that 
\eqref{eq:fully_faith} is an equivalence for $\cM =\cM_{tor} $. By d\'evissage, we may assume that $p \cM=0$. In this case the assertion reduces to the torsion-free case using the following observation.
\begin{lm}
    For any $p$-torsion Fontaine--Laffaille module $(\cM, \varphi)$, there exists a  torsion-free Fontaine--Laffaille module $(\tilde \cM, \tilde \varphi)$ with $(\tilde \cM, \tilde \varphi) \otimes \bF_p \iso (\cM, \varphi)$.
\end{lm}
\begin{proof}
Any vector bundle over $\bA^{1}_-/\bG_{m} \otimes \bF_p$ is  isomorphic to a direct sum of line bundles of the form $\cO(i)$, $i\in \bZ$. In particular, it lifts to a vector bundle over $\bA^1_-/\bG_m$. 
Pick a vector bundle $\tilde \cM$ over $\bA^1_-/\bG_m$ that lifts $\cM$ and then choose any $\tilde \varphi: F^*(\tilde \cM_{v_-=p})\xrightarrow{\sim}\tilde  \cM_{v_-=1}$ that lifts $\varphi: F^*(\tilde \cM_{v_-=p}) \otimes \bF_p\xrightarrow{\sim}\tilde  \cM_{v_-=1}\otimes \bF_p$.
\end{proof}
It remains to prove that the essential image of the functor \eqref{eq:functor_from_small_to_big}
consists of all objects $(\cM, \varphi)$ whose underlying complex $\cM\in \cD_{qc} (\bA^1_-/\bG_m)$
is bounded and has coherent cohomology.
We induct on length of the complex. If $\cM$ is supported in a single cohomological degree the assertion is clear. For the induction step, we assume that $\cM$ is connective and $H^0(\cM) \ne 0$. Note that the isomorphism 
$\varphi: F^*i^*_p \cM \xrightarrow{\sim} i^*_1 \cM$ induces $H^0(\varphi): F^*L_0i^*_p H^0(\cM) \xrightarrow{\sim} L_0 i^*_1 H^0(\cM)$. Thus $(H^0(\cM), H^0(\varphi))\in \mathscr{MF}(W(k))$.  By part (1) of the Proposition $ L_j i^*_p H^0(\cM)=  L_j i^*_1 H^0(\cM)  =0$ for $j>0$. It follows that  the map $\cM \to H^0 (\cM)$ in $\cD_{qc}(\bA^1_-/\bG_m)$ lifts  
to a morphism in  $\BFL$. Applying the induction assumption to its fiber we complete the proof.
 \end{proof}

Let $\mathscr{MF}_{[a,b]}(W(k))$ (resp.~$\mathscr{DMF}^{\mathrm{big}}_{[a,b]}(W(k))$\,) be a full subcategory of $\mathscr{MF}(W(k))$ (resp.~$\mathscr{DMF}^{\mathrm{big}}(W(k))$\,) formed of objects $(\cM,\varphi)\in \mathscr{MF}(W(k))$ (resp.~$\mathscr{DMF}^{\mathrm{big}}(W(k))$\,) such that $\cM\in \cD_{qc,[a,b]}(\bA^1_-/\bG_m)$. 

\subsection{Syntomification via filtered Cartier-Witt divisors}\label{subsec:preliminary-syntomification-section}
Recall from \cite{drinfeld2022prismatization},  \cite[\S 5.3]{bhatt-lecture-notes} definitions of the stacks $W(k)^{\Prism}$, $W(k)^{\cN}$, $W(k)^{syn}$, and $X^{syn}$.  

Throughout this paper we let $W$ be the $p$-typical Witt ring scheme pulled back to  $\Spec \bZ_{(p)}$.
For a $p$-nilpotent $W(k)$-algebra $R$, set $W_R := W \times \Spec R$. A {\it $W_R$-module} $M$ is an affine group scheme over $R$ with an action of the ring scheme $W_R$.
We say that $M$ is invertible if, Zariski locally on $\Spec R$, $M$ is isomorphic  to the free $W_R$-module $W_R$. Denote by  $W(k)^{\Prism}(R)$ the groupoid
whose objects are  pairs $(M, \xi)$, where $M$ is an invertible $W_R$-module and $\xi:M\to W_R$ is a morphism of $W_R$-modules such that  Zariski locally on $\Spec R$ the pair $(M, \xi)$ is isomorphic to $(W_R, w\Id)$, $w=(w_1, w_2, \cdots)\in W(R)$, with $w_1\in R$ nilpotent and $w_2\in R^*$.   

For a $W_R$-module $M$, precomposing the $W_R$-action on $M$ with the  Frobenius $F: W_R\to W_R$ we obtain a new 
$W_R$-module $F_*M$ whose underlying group scheme  is that of $M$. Applying this to the free $W_R$-module $W_R$, we have an exact sequence 
\begin{equation}\label{eq:short_exact_sequence}
    0\to \bG_{a,R}^{\sharp} \rar{} W_R \rar{F} F_*W_R  \to 0
\end{equation}
of  $W_R$-modules\footnote{To justify our notation for the kernel of $F$ recall  a fundamental
observation of Drinfeld \cite[Lemma 3.2.6]{drinfeld2022prismatization} and Bhatt-Lurie \cite{Bhatt-Lurie-2022absolute} that the (nonunital) ring scheme $ W^{(F)}_R:= \ker ( W_R \rar{F} W_R)$ is isomorphic to the PD-hull $\bG_{a,R}^{\sharp}$  of $\bG_{a,R}$ at $0$.}.
Note that the action of $W_R$ on $\bG_{a,R}^{\sharp}$ factors through the ``first coordinate'' homomorphism $W_R\to \bG_{a,R}$.
In particular, we have a natural action of the group $\bG_{m,R}$ on the $W_R$-module $\bG_{a,R}^{\sharp}$. Consequently, given an invertible $R$-module $L$ we can twist  $\bG_{a,R}^{\sharp}$ by $L$; the resulting $W_R$-module is denoted by $ \mathbf{V}(L)^{\sharp}$.
A  $W_R$-module $M$ is called {\it admissible}   if $M$ fits into a sequence of $W_R$-modules of the form
\begin{equation}\label{eq:defofadm}
 0 \to \mathbf{V}(L)^{\sharp} \to M \to F_* M' \to 0,   
\end{equation}
where $L \in \Pic(R)$ and $M'$ is an invertible $W_R$-module. One checks (see \cite[Lemma 3.12.7]{drinfeld2022prismatization}) that $L$, $M'$, and the sequence  \eqref{eq:defofadm} are functorial in $M$. In particular, every morphism $\xi:M\to W_R$ of $W_R$-modules extends uniquely to a morphism between short exact sequences \eqref{eq:defofadm} and \eqref{eq:short_exact_sequence}. 
Consider the category $W(k)^{\cN, c}(R)$ consisting of pairs $(M, \xi)$, called filtered Cartier-Witt divisors, where $M$ is an admissible $W_R$-module and $\xi:M\to W_R$ is a morphism of $W_R$-modules
such that the pair $(M', \bar \xi: M' \to W_R)$, constructed from the morphism of short exact sequences, is an object $W(k)^{\Prism}(R)$.
Morphisms between $(M, \xi)$ and $(M', \xi')$ in $W(k)^{\cN, c}(R)$ are given by morphisms $f: M \to M'$ of $W_R$-modules such that $\xi' \circ f = \xi$.
The groupoid $W(k)^{\cN} (R)$ is obtained from $W(k)^{\cN, c}(R)$ by discarding all morphisms which are not isomorphisms. By definition, the groupoid of $R$-points of the stack   $W(k)^{\cN}$ is  $W(k)^{\cN} (R)$. The functor $W(k)^{\cN, c}$ from the category of  $p$-nilpotent $W(k)$-algebras to the $2$-category of categories that carries an algebra $R$ to $W(k)^{\cN, c}(R)$ is an example of what Drinfeld calls in \cite{drinfeld2022prismatization}  a {\it $c$-stack}. We will not use this notion in the main body of the paper except  for Remark  \ref{rem:Drinfeld-commutativity-interpretation}.

Given an object $(M, \xi)$ of $W(k)^{\cN}(R)$, the morphism of $W_R$-modules $\mathbf{V}(L)^{\sharp}\to \bG_{a,R}^{\sharp}$ derived from $\xi$
lifts uniquely to a morphism of $R$-modules $L\to R$  (see \cite[Lemma 3.12.4, (ii)]{drinfeld2022prismatization}). This defines a morphism of stacks called the Rees map:
\begin{equation}\label{eqn:Rees-map-WkN}
t_{W(k)}: W(k)^{\cN} \to \bA^1_-/\bG_m\,.
\end{equation} Sending $(M, \xi)\in W(k)^{\cN}(R)$ to $(M', \bar \xi: M' \to W_R)\in W(k)^{\Prism}(R)$ determines a morphism $\pi : W(k)^{\cN} \to W(k)^{\Prism}$
called the structure morphism.  
Every invertible $W_R$-module is admissible and, moreover, every point $(M, \xi)\in W(k)^{\Prism}(R)$ is also a point of $W(k)^{\cN}(R)$
(see \cite[\S 5.3]{drinfeld2022prismatization}). This defines a map
$j_+: W(k)^{\Prism} \to W(k)^{\cN}$ which exhibits $W(k)^{\Prism}$ as an open substack of $ W(k)^{\cN}$ (see \cite[Lemma 5.3.1]{drinfeld2022prismatization}). As in  \cite[\S 5.6]{drinfeld2022prismatization}, we define another open embedding 
$j_-: W(k)^{\Prism} \to W(k)^{\cN}$ as follows. For $(M, \theta) \in W(k)^{\Prism} (R)$, define a filtered Cartier-Witt divisor $j_-(M, \theta):= (N, \xi)\in W(k)^{\cN}(R)$ by the pullback diagram 
\begin{equation}\label{dia:pullbackalongF}
\begin{tikzcd}
    N \arrow[]{r}{}\arrow[]{d}{\xi}&F_* M\arrow[]{d}{F_*(\theta)}\\
    W_R \arrow[]{r}{F}&F_*W_R
\end{tikzcd},
\end{equation}
 One can check that open substacks defined by $j_-$ and $j_+$ do not intersect (see \cite[Lemma 5.6.3]{drinfeld2022prismatization}). The stack $W(k)^{syn}$ is obtained from $W(k)^{\cN}$ by gluing two copies of $W(k)^{\Prism}$ using the maps $j_-$ and $j_+$.

For every $(M, \xi)\in W(k)^{\cN}(R)$ one considers the $W(k)$-algebra stack $\Cone(M \xrightarrow[]{\xi} W_R)$ (we refer the reader
to \cite[\S 1.3-1.4]{drinfeld2022prismatization} and the references therein for a discussion of the notion of ring stack). As a group stack,
this is simply the quotient stack $[W_R/M]$. The algebra structure on the quotient comes the natural DG-algebra structure on $M \xrightarrow[]{\xi} W_R$ given by the $W_R$-module structure on $M$  (see \cite[\S 1.3.3]{drinfeld2022prismatization}).
By ``\textit{transmutation}'' (\cite[Remark 2.3.8]{bhatt-lecture-notes},  \cite[\S 1.4.2]{drinfeld2022prismatization}), this defines a \textit{Nygaardization} functor $X \mapsto X^{\cN}$ from the category 
of bounded $p$-adic formal schemes  over $\Spf W(k)$ to stacks: for a $p$-nilpotent $W(k)$-algebra $R$, an $R$-point of the groupoid  $X^{\cN}(R)$ consists 
of $(M, \xi)\in W(k)^{\cN}(R)$ together with a morphism $\Spec (\Cone(M \xrightarrow[]{\xi} W_R)(R)) \to X$ of derived schemes over $W(k)$.
For every $(M, \theta)\in W(k)^{\Prism}(R)$, setting $j_+ (M, \theta)= (N_+, \xi_+)$, $j_- (M, \theta)= (N_-, \xi_-)$, one constructs (see \cite[\S 5.8]{drinfeld2022prismatization})) a natural isomorphism 
$\Cone(N_- \xrightarrow[]{\xi_-} W_R) \iso \Cone(N_+ \xrightarrow[]{\xi_+} W_R)$ of $W(k)$-algebra stacks (the map is given by the rightward arrows in diagram \eqref{dia:pullbackalongF}). Consequently, the fibers of $X^\cN \to W(k)^\cN$
over the two copies of $W(k)^{\Prism}$ inside $ W(k)^{\cN}$ are naturally identified (and denoted by\footnote{Bhatt's notations for
$j_+$ and $j_-$ are $j_{\text{HT}}$ and $j_{\dR}$, respectively.} $j_+, j_-: X^\Prism \mono X^\cN$). The stack $X^{syn}$ is obtained from $X^\cN$ by gluing these two open substacks.  This procedure defines a functor $X \mapsto X^{syn}$ from bounded $p$-adic formal schemes over $\Spf W(k)$ to stacks.  In particular, if $f: X \to \Spf W(k)$ is such a scheme, we have by functoriality a map $f^{syn}: X^{syn} \to W(k)^{syn}$. We refer to $\mathcal{H}_{syn}(X):= Rf_*^{syn}(\cO_{ X^{syn}})\in \cD_{qc}(W(k)^{syn})$ as the F-gauge associated to $X$.

Recall from \cite[Lemma 5.13.4 (i)]{drinfeld2022prismatization} that we have a morphism 
\begin{equation}\label{p-composed-to-Wk}
\mathfrak{p}_{\bar{\dR}}: \bA^1_-/\bG_m
\to W(k)^{\cN}
\end{equation} 
constructed as follows. Recall that an $R$-point of $\bA^1_-/\bG_m$ is a pair $(L ,v_-)$, where $L$ is an invertible $R$-module   
equipped with an $R$-linear  morphism $v_-: L \to R$. 
The map $\mathfrak{p}_{\bar{\dR}}$ is given by 
    $\mathfrak{p}_{\bar{\dR}}(L,v_-)=(M_{\mathfrak{p}_{\bar{\dR}, R}},\xi
    )$, 
where $M_{\mathfrak{p}_{\bar{\dR},R}}= \mathbf{V}(L)^{\sharp} \oplus F_*W_R$, and \begin{equation}\label{eqn:defining-xi-dR}
\xi:M_{\mathfrak{p}_{\bar{\dR},R}}=\mathbf{V}(L)^{\sharp} \oplus F_*W_R\to W_R
\end{equation} 
is given by $\xi:=(v_-^\sharp,V)$. Here  $V$ is the Verschiebung and $v_-^\sharp: \mathbf{V}(L)^{\sharp} \to  W_R$ is the composition of the map  $\mathbf{V}(L)^{\sharp} \to \bG_{a,\,R}^\sharp$ induced by $v_-$ and the embedding $\bG_{a,\,R}^\sharp \mono W_R$
from \eqref{eq:short_exact_sequence}.
One can verify (see \cite[Theorem 2.5.6]{bhatt-lecture-notes}) that, for a smooth $p$-adic formal scheme $X$ over $\Spf W(k)$, the complex 
$\mathfrak{p}^*_{\bar{\dR}}(\mathcal{H}_{syn}(X)) \in \cD_{qc}(\bA^1_-/\bG_m)$ recovers the Hodge filtered de Rham cohomology of $X$.

The stack $k^{\cN}$ can be described explicitly as follows. Set 
\begin{equation}\label{eqn:defining-A}
A:=W(k)[v_+,v_-]/(v_+v_-=p).
\end{equation}
We endow $A$ with a grading such that $\deg v_+=1$ and $\deg v_-=-1$, and consider the corresponding action of $\bG_m$ on $\Spf A$. Then $k^{\cN}$ is identified with $\Spf A/\bG_m$.~We shall just explain a construction of the map $\Spf A/\bG_m \to k^{\cN}$, referring the reader to \cite[\S 5.4]{bhatt-lecture-notes} for a proof of the isomorphism property.
First, we describe the composite map $\Cris:\Spf A/\bG_m\xrightarrow{\sim} k^{\cN}\to W(k)^{\cN}$ explicitly. Here the map $k^{\cN}\to W(k)^{\cN}$ comes from the morphism $\Spec k\to\Spf W(k)$ by functoriality. An $R$-point of $\Spf A/\bG_m$ is given by $R \xrightarrow{v_+}L \xrightarrow{v_-} R$, where $L$ is  an invertible $R$-module and $v_+v_-=p$. The map $\Cris$ takes $(R\xrightarrow{v_+} L \xrightarrow{v_-}R)$ and sends it to $(M_{\mathfrak{p}_{\mathrm{cris}},R},\xi)$, where 
$M_{\mathfrak{p}_{\mathrm{cris}},R}:=\coker(\bG_{a,R}^{\sharp}\xrightarrow{(v_+^{\sharp},-can)} \mathbf{V}(L)^{\sharp} \oplus W_R)$, and $\xi$ is induced by the map 
\begin{equation}\label{eqn:defining-xi-cris}
    \mathbf{V}(L)^{\sharp} \oplus W_R\xrightarrow{(can\,\circ\,v_-^{\sharp},\cdot p)}W_R. 
\end{equation}
The notation ``can'' stands for the canonical embedding $\bG_{a,R}^{\sharp}\hookrightarrow W_R$, and $\cdot p$ stands for multiplication by $p$. This defines the map $\Cris:\Spf A/\bG_m\to W(k)^{\cN}$. To lift this map to $k^{\cN}$ we need to endow $\Cone(M_{\mathfrak{p}_{\mathrm{cris}},R} \xrightarrow[]{\xi} W_R)$ with the structure of a $k$-algebra stack. 
In fact, we have a map of quasi-ideal pairs
$ (W_R \rar{\cdot p} W_R) \to (M_{\mathfrak{p}_{\mathrm{cris}},R} \xrightarrow[]{\xi} W_R)$ given by $ W_R\xrightarrow[]{(0, \Id)} \mathbf{V}(L)^{\sharp} \oplus W_R \to M_{\mathfrak{p}_{\mathrm{cris}},R}$ on the source and by $\Id$ on the target and, for any $W(k)$-algebra $D$, the animated ring $\Cone(D \xrightarrow[]{\cdot p} D)$ receives a map from $k\iso \Cone(W(k) \xrightarrow[]{\cdot p} W(k))$. Under the isomorphism  $\Spf A/\bG_m \iso k^{\cN}$, the open embedding
$j_+: k^\Prism \mono k^\cN$ (resp.~$j_-: k^\Prism \mono k^\cN$) is identified with $\Spf W(k) \rar{F} \Spf W(k) \xrightarrow[]{v_+=1, v_-=p} \Spf A \to \Spf A/\bG_m$  (resp. $\Spf W(k) \xrightarrow[]{v_+=p, v_-=1} \Spf A \to \Spf A/\bG_m$). Let us just construct an isomorphism of
{\it $k$-algebra} stacks obtained by pulling back $\bG_{a, k}^{\cN}$ along $j_{\pm}: \Spf W(k) \to  \Spf A/\bG_m$.
The restriction of $\bG_{a, k}^{\cN}$ to $\Spf W(k) \xrightarrow[]{v_+=1, v_-=p} \Spf A/\bG_m$ is given by $\Cone( W \rar{\cdot p} W)$; its restriction along  $\Spf W(k) \xrightarrow[]{v_+=p, v_-=1}  \Spf A/\bG_m$ is $\Cone(F_* W \rar{\cdot p} F_* W)$. As a $k$-algebra stack the latter is obtained from the former by precomposing the action of $k$ with the Frobenius. Equivalently,
the $k$-algebra stack $\Cone(F_* W \rar{\cdot p} F_* W)$ is isomorphic to the pullback of $\Cone( W \rar{\cdot p} W)$ along the Frobenius on $\Spf W(k)$. We refer the reader to \cite[\S 3.3]{bhatt-lecture-notes} and especially diagram (3.3.2) in {\it loc. cit.} for details.

For a smooth $p$-adic formal scheme $X$ over $\Spf W(k)$,  $\Cris^*(\mathcal{H}_{syn}(X))$ recovers the crystalline cohomology of $X\otimes \bF_p$ equipped with
the Nygaard filtration (see \cite[\S 3.3]{bhatt-lecture-notes}).

\subsection{The reduced locus of $W(k)^{\cN}$}
Let $C_2:=\Spec k[v_+^p,v_-]/(v_+^pv_-)$. We endow $C_2$ with a $\bG_m$-action given by the grading: $\deg v_+^p=p$ and $\deg v_-=-1$. The inclusion of $k[v_+^p,v_-]/(v_+^pv_-)\hookrightarrow k[v_+,v_-]/(v_+v_-)$ gives the morphism $k^{\cN}\otimes\bF_p\to C_2/\bG_m$. 
Recall from \cite[\S~5.16.10]{drinfeld2022prismatization} a factorization of $\mathfrak{p}_{\mathrm{cris}}$:
\begin{equation}
 k^{\cN}\otimes\bF_p\to C_2/\bG_m\xrightarrow{\MapCmodGmtoWkNred} W(k)^{\cN}_{red}\to W(k)^{\cN}\otimes\bF_p.
\end{equation}
Let us just explain a construction of the composite 
\begin{equation}\label{map-C2-to-WkNygaard}
C_2 \epi C_2/\bG_m\xrightarrow{\MapCmodGmtoWkNred} W(k)^{\cN},
\end{equation}
given by a pair $(M_{C_2},\alpha:M_{C_2}\to W_{C_2})$, where $M_{C_2}$ is an admissible $W_{C_2}$-module and $\alpha$ is a $W_{C_2}$-module homomorphism.~Explicitly, $M_{C_2}$ is given by the pullback diagram 
\begin{equation}
\begin{tikzcd}
    M_{C_2}\arrow[]{r}{}\arrow[]{d}{}&F_*W_{C_2}\arrow[]{d}{[v_+^p]}\\
    W_{C_2}\arrow[]{r}{F}&F_*W_{C_2}
\end{tikzcd},
\end{equation}
where the right downward arrow $[v_+^p]$ takes a Witt vector $w\in W_{C_2}$ to the product $[v_+^p]w$.~The morphism $\alpha$ is defined by 
$M_{C_2}=W_{C_2}\times_{F_*W_{C_2}}F_*W_{C_2}\xrightarrow{([v_-],V)}W_{C_2}$. 
The isomorphism between the composition 
\[
C:=\Spec k[v_+,v_-]/(v_+v_-)\twoheadrightarrow C/\bG_m=k^{\cN}\otimes\bF_p\xrightarrow{\mathfrak{p}_{\mathrm{cris}}} W(k)^{\cN}\]
and $C\to C_2\xrightarrow{\eqref{map-C2-to-WkNygaard}} W(k)^{\cN}$ 
is given by an isomorphism of $W_C$-modules 
\begin{equation}\label{eqn:preliminary-isom-admissible-modules-PsiDI}
\coker\left(W_C^{(F)}\xrightarrow{[v_+],-\Id} W_C^{(F)}\times W_C\right)\xrightarrow{\sim}W_C\times_{F_*W_C}F_*W_C=:C\times_{C_2}M_{C_2}
\end{equation}
defined by the matrix 
$\begin{pmatrix}
  \Id &    [v_+]\\
 0&  F
\end{pmatrix}$. 
In \cite[Corollary 7.5.1]{drinfeld2022prismatization}, Drinfeld refines the morphism $\widetilde{\mathfrak{p}}_{\mathrm{cris}}:C_2/\bG_m\to W(k)^{\cN}_{red}$ to an isomorphism 
\begin{equation}\label{eqn:isom-BH-WK-Nygaard-red}
BH\xrightarrow{\sim} W(k)^{\cN}_{red},
\end{equation}
where $H$ is an affine group scheme over $C_2/\bG_m$ constructed as follows. 
Let $\widetilde{H}:=W^{(F)}\times C_2$. For $(w,t), (w',t)\in \widetilde{H}$, the formula 
$(w,t)*(w',t)=(w+w'-[v_-^p(t)]ww',t)$ defines a map $\widetilde{H}\times\widetilde{H}\to \widetilde{H}$, making $\widetilde{H}$ a group scheme over $C_2$. The restriction of $\widetilde{H}$ to the open subscheme of $C_2$ given by $v_-\neq 0$ is $\bG_m^{\sharp}$, and its restriction to the closed subscheme given by $v_-=0$ is $\bG_a^{\sharp}$. Define a $\bG_m$-equivariant structure on $W\times C_2$ by 
\begin{equation}
    \bG_m\times W\times C_2\to W\times C_2, \quad (\lambda,w,t)\mapsto ([\lambda^p]w,\lambda t).
\end{equation}
This restricts to a $\bG_m$-equivariant structure on  $\widetilde{H}\hookrightarrow W\times C_2$ compatible with the group structure. This defines the promised group scheme $H$ over $C_2/\bG_m$. 

We write $(C_2)_{v_+^p=0}$, $(C_2)_{v_-=0}$ for the closed subschemes given by the respective equations. Denote by $D_{\HT}$ (resp.~$D_{\dR}$) the restrictions of $C_2/\bG_m$-stack $BH$ to $(C_2)_{v_-=0}/\bG_m$ (resp. $(C_2)_{v_+^p=0}/\bG_m$). Thus, via \eqref{eqn:isom-BH-WK-Nygaard-red}, one has $W(k)^{\cN}_{red}=D_{\HT}\cup D_{\dR}$. We refer to $D_{\HT}$ (resp.~$D_{\dR}$) as the \textit{Hodge-Tate} (resp.~\textit{de Rham}) component of $W(k)^{\cN}_{red}$. 
We denote by $D_{Hod}$ the fibre of $BH$ over $(C_2)_{v_+^p=v_-=0}/\bG_m=B\bG_m$. 
Geometrically, one pictures $W(k)^{\cN}_{red}$ as a union of two components ({\it i.e.},~$D_{\HT}$ and $D_{\dR}$) meeting \textit{transversally} at $D_{Hod}$. The restriction gives a functor 
\begin{equation}\label{Wk-Nygaard-red-fibre-product-DHT-DdR}
\mathcal{R}^*: \cD_{qc}(W(k)^{\cN}_{red})\to \cD_{qc}(D_{\HT})\times_{\cD_{qc}(D_{Hod})}\cD_{qc}(D_{\dR}). 
\end{equation}
The following result is explained in \cite[\S 6]{bhatt-lecture-notes} as a consequence of a general statement from \cite[Theorem 16.2.0.2]{Lurie-sag}.
\begin{lm}\label{lm:ffembedding}
   Functor \eqref{Wk-Nygaard-red-fibre-product-DHT-DdR} is fully faithful and admits right adjoint $\mathcal{R}_*$ given by \[ (\cF_{\HT},\cF_{\dR},\alpha: \cF_{\HT}|_{D_{Hod}}\simeq \cF_{\dR}|_{D_{Hod}})\mapsto \mathrm{Fib}(i_{HT*}\cF_{\HT}\oplus i_{\dR*}\cF_{\dR}\to i_{Hod*}(\cF_{\HT}|_{D_{Hod}})),\] 
    where $i_{?}$ is the closed embedding of $D_{?}$ into $W(k)^{\cN}_{red}$.
\end{lm}
The following result describes the essential image of $\mathcal{R}^*$ in \eqref{Wk-Nygaard-red-fibre-product-DHT-DdR}. Let $\cF:=(\cF_{\HT},\cF_{\dR},\alpha: \cF_{\HT}|_{D_{Hod}}\simeq \cF_{\dR}|_{D_{Hod}})$ be an object of the right-hand side of \eqref{Wk-Nygaard-red-fibre-product-DHT-DdR}. 
\begin{lm}\label{lemma:essential-image-restriction-functor-fibre-product}
Then $\cF$ lies in the essential image of $\mathcal{R}^*$ if and only if its pullback to 
\begin{equation}\label{fibre-product-C2-equations}
\cD_{qc}\left((C_2)_{v_-=0}/\bG_m\right){\times}_{\cD_{qc}\left((C_2)_{v_+^p=v_-=0}/\bG_m\right)}\cD_{qc}\left((C_2)_{v_+^p=0}/\bG_m\right)
\end{equation}
comes from an object of $\cD_{qc}(C_2/\bG_m)$. 
\end{lm}
\begin{proof}
$\cF$ lies in the essential image if and only if the canonical morphism $\mathcal{R}^*\mathcal{R}_*\cF\to \cF$ is an isomorphism. Since the map $\MapCmodGmtoWkNred: C_2/\bG_m\to W(k)^{\cN}_{red}$ is faithfully flat, it suffices to check this after pulling back along $\MapCmodGmtoWkNred$ to $C_2/\bG_m$. 
\end{proof}
In \cite[\S6.2]{bhatt-lecture-notes}, Bhatt and Lurie give a convenient description of $D_{\HT}$, {\it i.e.},~$D_{\HT}\simeq (\bA^{1,\dR}_+/\bG_m)\otimes\bF_p$.~To reconcile this isomorphism with the above description, we recall that the Frobenius morphism $\bA^1_+\otimes\bF_p\to \Spec k[v_+^p]=(C_2)_{v_-=0}$ factors through $(\bA^1_+\otimes\bF_p)^{\dR}$.~Moreover, a choice of $\delta$-structure on $\bA^1_+$, such that $\delta(v_+)=0$, gives a $\bG_m$-equivariant isomorphism $(\bA^1_+\otimes\bF_p)^{\dR}=B\bG_a^{\sharp}\times (C_2)_{v_-=0}$, where the action of $\bG_m$ on $\bG_a^{\sharp}$ is given by $\lambda*z=\lambda^pz$.~Using the identification of the Cartier dual to $\bG_a^{\sharp}$ and $\widehat{\bG}_a=\Spf k[[D^p]]$, we obtain an equivalence $\cD_{qc}(D_{\HT})\simeq \cD_{gr,D^p\text{-nilp}}(k[v_+^p,D^p])$, where the right hand side is a 
full subcategory of the derived category $\cD_{gr}(k[D^p, v_+^p])$ of graded modules over the polynomial algebra $k[D^p,v_+^p]$ with $\deg v_+^p=-\deg D^p=p$.~Objects of this full subcategory $\cD_{gr,D^p\text{-nilp}}(k[D^p,v_+^p])$ consist of complexes $\cM$ such that the action of $D^p$ on $\bigoplus\limits_iH^i(\cM)$ is locally nilpotent.
\subsection{Coherent sheaves and vector bundles over $X^{syn}$}\label{subsection:coherent_sheaves_and_vector_bundles}
Let $Q$ be a $p$-complete Noetherian regular local ring. In \cite[Remark 5.5.19]{bhatt-lecture-notes}, it is shown that the category $ \Perf(Q^{syn})$ of perfect complexes on $Q^{syn}$ has a unique $t$-structure whose category of connective objects  $\Perf^{\leq 0}(Q^{syn})$ consists of all $\cF\in \Perf(Q^{syn})$ such that, for any $p$-nilpotent ring $R$ and a map $f:\Spec R \to Q^{syn}$,
the pullback $f^*\cF$ is in $D^{\leq 0}(R)$. We refer to the heart of this $t$-structure as the category $\Coh(Q^{syn})$ of coherent sheaves on $Q^{syn}$. The latter has a more concrete description. 
Let  $(\tilde Q, (d), \delta)$ be a $p$-torsion prism with $\tilde Q/(d)=Q$. Then there exists a canonical faithfully flat map
\begin{equation}\label{eq:chart}
    \Spf \tilde{Q}[v_-, v_+]/(v_-v_+-d) \to Q^{syn},
\end{equation}
where $ \Spf \tilde{Q}[v_-, v_+]/(v_-v_+-d)$ is the $(p,d)$-adic formal scheme 
$\colim_N \Spec \tilde{Q}[v_-, v_+]/(v_-v_+-d, d^N, p^N)$; see \cite[Remark 5.5.19]{bhatt-lecture-notes}, {\it cf.} the end of \S \ref{subsec:preliminary-syntomification-section}. Then $\cF \in 
\cD_{qc}(Q^{syn})$ belongs to  $\Coh(Q^{syn})$ if and only of its pullback along \eqref{eq:chart}
is coherent, that is, a finite module over\footnote{Note that since the ring is regular, every finite module over it is perfect.} $\tilde{Q}[v_-, v_+]/(v_-v_+-d)$.

Recall that, for any stack $\cX$, the category $\Vect(\cX)$ of vector bundles over $\cX$ is defined to be
$\lim_{\Spec R \to \cX} \Vect(\Spec R)$, where the limit is taken over the category all affine schemes over $\cX$. An object $\cF \in 
\cD_{qc}(Q^{syn})$ is a vector bundle if and only if its pullback along \eqref{eq:chart} is a finite projective module over $\tilde{Q}[v_-, v_+]/(v_-v_+-d)$.

\subsection{Hodge-Tate weights}\label{subsection:HT-weights}
Recall that $k^{\cN}$ can be identified with $\Spf A/\bG_m$ \cite[\S 5.4]{bhatt-lecture-notes}. Consequently, the category $\cD_{qc}(k^{\cN})$ is equivalent to the subcategory $\widehat{\cD}_{gr}(A)$ of the derived category $\cD_{gr}(A)$ of graded $A$-modules spanned by $p$-complete objects. Given a gauge $\cF\in \cD_{qc}(k^{\cN})$, we will denote by 
\begin{equation}\label{identifying-CohFpN-with-modules}
(\cN^{\cdot}=\bigoplus\limits_i \cN^i, v_-:\cN^i\to\cN^{i-1}, v_+:\cN^i\to \cN^{i+1},v_-v_+=v_+v_-=p)
\end{equation}
the corresponding object in $\widehat{\cD}_{gr}(A)$. 
\begin{df}\label{defn:HT-weights}
We say that a gauge $\cF\in \cD_{qc}(k^{\cN})$ has Hodge-Tate weights $\geq a$ if, under the identification \eqref{identifying-CohFpN-with-modules}, the maps $v_-$ in 
\begin{tikzcd} \cN^a\arrow[shift left]{r}{v_-}&\cN^{a-1}\arrow[shift left]{r}{v_-}&\cN^{a-2}\arrow[shift left]{r}{v_-}&\cdots\end{tikzcd} are all quasi-isomorphisms.

A gauge $\cF$ is said to be \textit{effective} if it has Hodge-Tate weights $\geq 0$. 
We say that a gauge $\cF$ has Hodge-Tate weights $\leq b$ if 
under the identification \eqref{identifying-CohFpN-with-modules}, the maps $v_+$ in 
\begin{tikzcd} \cdots\cN^{b+3}&\cN^{b+2}\arrow[]{l}{v_+}&\cN^{b+1}\arrow[]{l}{v_+}& \cN^b\arrow[]{l}{v_+}\end{tikzcd} are all quasi-isomorphisms.

We say that a gauge $\cF\in \cD_{qc}(W(k)^{\cN})$ has Hodge-Tate weights $\geq a$ (resp. $\leq b$) if its pullback to $k^{\cN}$ has the corresponding property. Likewise, we say that an $F$-gauge $\cF\in \cD_{qc}(W(k)^{syn})$ has Hogde-Tate weights $\geq a$ (resp. $\leq b$) if its pullback to $k^{\cN}$ has the corresponding property. We denote by $\cD_{qc,[a,b]}(k^{\cN})$ the full subcategory of $\cD_{qc}(k^{\cN})$ that consists of objects with Hodge-Tate weights in $[a,b]$. Similarly, we define $\cD_{qc,[a,b]}(W(k)^{\cN})$. 
\end{df}

\begin{rem}\label{HT-weights-kN-remark}
    We observe that $\cF\in\cD_{qc}(k^{\cN})$ has Hodge-Tate weights in $[a,b]$ if and only if the pullbacks of $\cF$ along both composites 
    $\Spec k[v_{\pm}]/\bG_{m,k}\to k^{\cN}\otimes \bF_p\to  k^{\cN}$ lie in $\cD_{qc,[a,b]}((\bA^1_{\pm}/\bG_m)\otimes \bF_p)$ in the sense of Definition \ref{Defn-Dqc-ab-A1modGm-plusminus}. 
\end{rem}

\begin{lm}\label{c*-preserve-weights}
    Suppose $\cF\in \cD_{qc,[a,b]}(\bA^1_-/\bG_m)$, then $t^*\cF\in \cD_{qc,[a,b]}(k^{\cN})$.  
\end{lm}
\begin{proof}
The map $i_{\pm}:(\bA^1_{\pm}/\bG_m)\otimes \bF_p=\Spec k[v_{\pm}]/\bG_{m, k}\mono k^{\cN}$ given by the equation $v_{\mp}=0$ is a closed embedding. Since $v_+$ is a nonzero divisor in $W(k)[v_+,v_-]/(v_+v_--p)$, the restriction functor along the above embedding corresponds algebraically to taking the cone of $v_+$ on the corresponding graded module. Then saying that $\cG\in\cD_{qc}(k^{\cN})$ has weights $\leq b$ is equivalent to saying that the restriction $i^*_{-}\cG$ has weights $\leq b$. The composition $t\circ i$ is the embedding $\Spec k[v_-]/\bG_m\to \Spf W(k)[v_-]/\bG_m$ given by the equation $p=0$. Since, for any $\cF\in \cD_{qc,(-\infty,b]}(\bA^1_-/\bG_m)$,  its restriction to the special fiber $(\bA^1_-/\bG_m)\otimes \bF_p$ has weights $\leq b$, we conclude that $t^*\cF\in \cD_{qc, (-\infty ,b]}(k^{\cN})$. Saying that $\cG\in\cD_{qc}(k^{\cN})$ has weights $\geq a$ is equivalent to saying that the restriction $i^*_{+} \cG$ has weights $\geq a$. The composition $t \circ i_{+}$ is isomorphic to $(\bA^1_{+}/\bG_m) \otimes \bF_p \to B\bG_m \otimes \bF_p \to \bA^1_{-}/\bG_m$ and the lemma follows.
\end{proof}

\begin{lm}\label{rho-dR-preserve-weights}
 Suppose $\cF\in\cD_{qc,(-\infty,b]}(W(k)^{\cN})$, then $\mathfrak{p}_{\bar{\dR}}^*\cF\in \cD_{qc,(-\infty, b]}(\bA^1_-/\bG_m)$.   
\end{lm}
\begin{proof}
By $p$-completeness, it suffices to show $\mathfrak{p}_{\bar{\dR}}^*\cF\otimes \bF_p\in \cD_{qc,(-\infty,b]}\left((\bA^1_-/\bG_m)\otimes \bF_p\right)$. It remains to observe that the composition $(\bA^1_-/\bG_m)\otimes \bF_p\xrightarrow{}\bA^1_-/\bG_m\xrightarrow{\mathfrak{p}_{\bar{\dR}}}W(k)^{\cN}$ is isomorphic to $(\bA^1_-/\bG_m)\otimes \bF_p\xrightarrow{i_{-}} k^{\cN}\xrightarrow{\Cris}W(k)^{\cN}$, where $i_{-}$ is defined in the first line of the proof of Lemma \ref{c*-preserve-weights}; see, for example, \cite[\S 2.8.3]{bhatt-lecture-notes}.
\end{proof}

\begin{rem}\label{remark-two-notions-effectivity}
If $\cF \in \cD_{qc}(k^{\cN})$ is effective, then its restriction along the map $B\bG_{m,k}\mono k^{\cN}$ is given by $B\bG_{m,k}\overset{v_-=0
}{\hookrightarrow}(\bA^1_-/\bG_m)\otimes \bF_p \overset{v_+=0}{\hookrightarrow}k^{\cN}$, viewed as a $\bZ$-graded object $M^{\bullet}$ of 
$\cD_{qc}(k)$, has weights $\geq 0$, {\it i.e.},~$M^i$ is acyclic for every $i<0$. 
If $\cF$ is perfect, then the converse is true. 
\end{rem}
    
\begin{lm}\label{lm:existenceofrightadjoint}
The embedding $e: \cD_{qc,(-\infty,b]}(k^{\cN}) \mono \cD_{qc}(k^{\cN})$ admits a right adjoint functor 
$w_{\leq b}:  \cD_{qc}(k^{\cN})  \to \cD_{qc,(-\infty,b]}(k^{\cN})$.
\end{lm}
\begin{proof}
    We define $w_{\leq b}$ by sending an object $(M^{\bullet},v_-,v_+)\in \cD_{qc}(k^{\cN})$ to the object $({M'}^{\bullet},v'_-,v'_+)\in \cD_{qc,(-\infty,b]}(k^{\cN})$, where ${M'}^i=M^i$ for $i<b$ and ${M'}^i=M^b$ otherwise. Here $v'_+: {M'}^i\to {M'}^{i+1}$ is equal to $v_+$ for $i<b$ and equal to $\Id$ otherwise; here $v'_-: {M'}^i\to {M'}^{i-1}$ is equal to $v_-$ for $i\leq b$ and equal to $p\cdot\Id$ otherwise. Indeed, we have a canonical map $f(\bullet)$ sending $({M'}^{\bullet},v'_-,v'_+)$ to $(M^{\bullet},v_-,v_+)$ given by 
    $f(i)=\Id$ for $i\leq b$ and $f(i)=v_+^{i-b}$ otherwise. This defines a natural transformation of functors $e\circ w_{\leq b}\to \Id$. On the other hand, we have an isomorphism $w_{\leq b}\circ e\iso \Id$. 
\end{proof}

\section{The fiber product $k^{\cN}\times_{W(k)^{\cN}} \bA^1_-/\bG_m$}
\subsection{Drinfeld Lemma}\label{subsection:Drinfeld-lemma} 
Recall from \eqref{eqn:defining-A} that $A:=W(k)[v_+,v_-]/(v_+v_-=p)$ with a grading such that $\deg v_+=1$ and $\deg v_-=-1$.~Let $A^i$ be the subgroup of elements of degree $i$.~Let $B\subset A\otimes\mathrm{Frac}(W(k))$ be the $W(k)$-subalgebra generated by $v_-$ and $\frac{v_+^n}{n!}$ for $n\geq 0$.~The grading on $A$ induces a grading on $B$ and we write $B=\bigoplus\limits_iB^i$.

\begin{df}\label{def:Mazur numbers} 
For $n\in\bN_{>0}$, its associated \textit{Mazur number} is $[n]:=\min\limits_{m\geq n}\mathrm{ord}\frac{p^m}{m!}$. 
\end{df}
Note that $[n]=n$ for $n<p$, and $[n+1]-[n]=1$ or $0$ for every $n\in\bN_{>0}$.~For every positive integer $n$, the ideal $(p^{[n]})\subset \bZ_p$ is the $n$-th divided power of the ideal $(p)$.
\begin{lm}\label{Bi-free-module}
   For every $i\geq 0$,  $B^i$ is the free $W(k)$-module on $\frac{p^{[i]}v_+^i}{p^i}$. For every $i\leq 0$, $B^i$ is the free $W(k)$-module generated by $v_-^i$. 
\end{lm}
\begin{cor}\label{cor:AtoB}
The embedding $A\hookrightarrow B$ is an isomorphism in degrees $<p$. In particular, $B$ as a graded $W(k)[v_-]$-module is effective in the sense of Definition \ref{effective-Coh}.    
\end{cor}
\begin{rem} Lemma \ref{Bi-free-module} shows that $B$ is the Rees algebra of the filtration on $W(k)$ given by divided powers of the ideal $(p)\subset W(k)$. 
\end{rem}
\begin{rem}\label{rem:the_category_of_B-modules} 
 Recall from \cite[\S 3.3]{bhatt-lecture-notes} that the stable $\infty$-category $\cD_{qc}(k^\cN)= \widehat{\cD}_{gr}(A)$ of graded $p$-complete $A$-modules can be identified with 
the derived category of filtered $p$-complete modules over the filtered algebra $W(k)$, where the latter is equipped with the filtration by powers of the ideal 
$(p)\subset W(k)$. Likewise, the category $\widehat{\cD}_{gr}(B)$ of graded $p$-complete $B$-modules can be identified with 
the derived category of filtered $p$-complete modules over the filtered algebra $(W(k), (p)^{[\cdot]})$, where $(p)^{[\cdot]}$ denotes the filtration by {\it divided} powers of the ideal 
$(p)$.
\end{rem}
Recall the maps $\mathfrak{p}_{\bar{\dR}}$ and $\Cris$ from \S\ref{subsec:preliminary-syntomification-section}. Consider the following diagram
\begin{equation}\label{Drinfeld-diagram}
\begin{tikzcd}
    &\Spf B/\bG_m\arrow[]{d}[swap]{a}&\\
    &k^{\cN}=\Spf A/\bG_m\arrow[]{ld}[swap]{\Cris}\arrow[]{rd}{\Reesparameter}&\\
    W(k)^{\cN}&& \bA^1_-/\bG_m\arrow[]{ll}[swap]{\mathfrak{p}_{\bar{\dR}}}
\end{tikzcd}\end{equation}
where $a$ and $\Reesparameter$ are induced by homomorphisms of graded $W(k)$-algebras: the map $\Reesparameter$ is induced by the map $W(k)[v_-]\to A$ sending $v_-\mapsto v_-$, and the map $a$ is induced by the obvious embedding $A\hookrightarrow B$. Warning: $\Cris\not\simeq \mathfrak{p}_{\bar{\dR}}\circ \Reesparameter$.  
\begin{Th}[Drinfeld]\label{Drinfeld-prop}
 There exists a unique isomorphism $\Psi_{\mathrm{Maz}}:\Cris\circ a\simeq \mathfrak{p}_{\bar{\dR}}\circ\Reesparameter\circ a$.   
\end{Th}
We will need the following lemma in the proof of the Theorem. 
\begin{lm}\label{divisibility-W(B)-lemma}
The element $[v_+^p]$ in $W(B)$ is uniquely divisible by $p$.    
\end{lm}
\begin{proof}
Uniqueness is obvious because $p$ is not a zero-divisor in $B$ (hence in $W(B)$). 
It is enough to show that $[v_{+}]^p$ is divisible by $p$ in the ring of big Witt vectors. Recall that the additive group of the big Witt vectors is $(1+xB[[x]])^*$ and $[v_{+}]^p$ corresponds to $1-v_{+}^p x$ and to show that it is divisible by $p$ we need to show that there exists a series $g \in (1+B[[x]])^*$ such that $g^p = 1-v_{+}^p x.$ Explicitly, $g=\exp(\frac{1}{p} \log(1-v_{+}^p x))$. However, we need to explain why this formula makes sense. Note that $\log(1-v_{+}^p x)=-\sum_{n \ge 1} \dfrac{v_{+}^{pn} x^n}{n}$ makes sense since $v_{+}^{pn}$ is divisible by $(pn)!$ and, moreover, this series is divisible by $p$ since all the summands are. To make sense of $\exp(\frac{1}{p} \log(1-v_{+}^p x))=\exp(-\sum_{n \ge 1} \dfrac{v_{+}^{pn} x^n}{pn})$ we need to explain why this element which a priori lives in $(1+x\bQ[[x]])^*$ actually lives in $(1+xB[[x]])^*.$ That is, we need to explain why $\Big(\sum_{n \ge 1} \dfrac{v_{+}^{pn} x^n}{pn}\Big)^m$ is divisible by $m!$ and it is enough to show that $\Big(\dfrac{v_{+}^{pn} x^n}{pn}\Big)^m=\dfrac{v_{+}^{pnm} x^{nm}}{(pn)^m}$ is divisible by $m!$. Recall that $v_{+}^{pnm}$ is divisible by $(pnm)!$ so we are reduced to show that $\dfrac{(pnm)!}{(pn)^m m!}$ is a $p$-adic integer. This follows from a known fact that $\dfrac{(nm)!}{(n)^m m!} \in \bZ$.     
\end{proof}
\begin{rem}\label{rem:another_proof}
   Let us sketch another proof of  Lemma \ref{divisibility-W(B)-lemma}. There  exists a unique element $y' \in \bG_a^\sharp(B)$ whose image in $\bG_a(B)$ is $v_+$. Let $y\in W^{(F)}(B)$ be the image of $y'$ under the isomorphism $\bG_a^\sharp \iso W^{(F)}$ (see \cite[Lemma 3.2.6]{drinfeld2022prismatization}). Then the first ghost 
   coordinate of $[v_+] -y$ is $0$. Thus, there exists a unique $z\in  W(B)$ with $V(z)= [v_+] -y$. We claim that $p\cdot z = [v_+^p]$. Indeed,
   $p\cdot z = FV(z)=F([v_+] -y)= [v_+^p]$.
     \end{rem} 
\begin{rem}\label{rem:G_m_equivariance}
    The Witt vector $\frac{[v_+^p]}{p}\in W(B)$ determines a map $\Spec B \xrightarrow{d} W$. We claim that
    the following diagram is commutative
\begin{equation}\label{dia:G_m_equivariance}
\begin{tikzcd}
\bG_m\times \Spec B \arrow[]{d}{\alpha}\;\arrow[]{r}{\Id\times d}&\bG_m\times W \arrow[]{d}{\beta}\\
    \Spec B\arrow[]{r}{d}&W\,.
    \end{tikzcd}
    \end{equation}
    Here the action $\alpha$ of $\bG_m$ on $\Spec B$ is given by the grading and the action  of $\bG_m$ on $W$ is given by the formula $\beta(\lambda, w)=[\lambda^p] w$, where $\lambda \in \bG_m, w \in W$, and $[\cdot]$ refers to the Teichm\"uller representative.  Indeed, since $\bG_m\times \Spec B $ is flat over $\bZ_p$ it suffices to check 
    that $ d \circ \alpha = \beta \circ (\Id\times d)$ after post-composing with $W\rar{p}W$ which is clear.    

   Using \eqref{dia:G_m_equivariance} the map $d$ descends to a map  
   $$  \bar d: \Spec B/\bG_m \to W/\bG_m $$ 
   of the quotient stacks.
     \end{rem} 
     
\begin{proof}[Proof of Theorem \ref{Drinfeld-prop}] 
Fix an $S$-point $\mathcal{P}$ of $\Spf B/\bG_m$. Using the morphism $a$, we have an $S$-point of $\Spf A/\bG_m$, {\it i.e.},~$\cO_S\xrightarrow{v_+}L\xrightarrow{v_-}\cO_S$, where $L$ is a line bundle over $S$ and $v_+v_-=p$. Define a morphism of $W_S$-modules $w:W_S\to \mathbf{V}(L)^{\sharp}\mono [L]$, where the invertible $W_S$-module $[L]$ is the Teichm\"uller lift of $L$ (see \cite[3.11]{drinfeld2022prismatization}). By Lemma \ref{divisibility-W(B)-lemma}, the homomorphism $[v_+^p]: W_S\to [L^{\otimes p}]$ of $W_S$-modules is canonically divisible by $p$. 
Indeed, we first define 
\begin{equation}\label{eq:divided}
 \frac{[v_+^p]}{p}: W_S\to [L^{\otimes p}]   
\end{equation}
 locally
on $S$ assuming a lift $\tilde{ \mathcal{P}}: S\to \Spf B $ of $\mathcal{P}$.
 The choice of a lift trivializes $L$ and $[L]$. Composing $\tilde {\mathcal{P}}$ with
$\Spec B \rar{d} W$ from Remark \ref{rem:G_m_equivariance} we obtain a Witt vector $\frac{[v_+^p]}{p} \in W(S)$. The multiplication by $\frac{[v_+^p]}{p}$ defines the desired map $\frac{[v_+^p]}{p}: W_S\to W_S= [L^{\otimes p}]$. Diagram 
\eqref{dia:G_m_equivariance}
shows that \eqref{eq:divided}
does not depend on the choice of a trivialization\footnote{Indeed, consider the map  $  \bar d: \Spec B/\bG_m \to W/\bG_m $ from Remark \ref{rem:G_m_equivariance}.
Note that an $S$-point of $W/\bG_m$ is a line bundle $L'$ over $S$ together with a map $W_S$-modules $W_S\to [L^{' \otimes p}]$.
The image
of $\cP$ under $\bar d$ is our \eqref{eq:divided}.} and, thus, it is well-defined globally. 
Applying the Verschiebung, we obtain a homomorphism $V(\frac{[v_+^p]}{p}): W_S\to [L]$ of $W_S$-modules. Set $w:=[v_+]-V(\frac{[v_+^p]}{p}): W_S\to [L]$. We claim that $w$ lands in $\mathbf{V}(L)^{\sharp}$. 
Indeed, the composition $F\circ w: W_S\to [L^{\otimes p}]$ is zero. 

We want to construct an isomorphism $\Cris\circ a(\mathcal{P})\simeq \mathfrak{p}_{\bar{\dR}}\circ\Reesparameter\circ a(\mathcal{P})$. Define a morphism $\Psi_{\mathrm{Maz}}:M_{\mathfrak{p}_{\mathrm{cris}},S}:=\coker\left(\bG_{a,S}^{\sharp}\xrightarrow{(v_+^{\sharp},-can)} \mathbf{V}(L)^{\sharp}\oplus W_S\right)\longrightarrow \mathbf{V}(L)^{\sharp}\oplus F_*W_S:=M_{\mathfrak{p}_{\bar{\dR},S}}$
as follows: let the map $\mathbf{V}(L)^{\sharp}\oplus W_S\to \mathbf{V}(L)^{\sharp}\oplus F_*W_S$ be given by the matrix 
$\begin{pmatrix}
    1& w\\
    0& F
\end{pmatrix}$. 
We claim that $\Psi_{\mathrm{Maz}}\circ(v_+^{\sharp},-can)$ is zero. Indeed, $F\circ can$ is zero. We also have that $v_+^{\sharp}-w\circ can:\bG_{a,S}^{\sharp}\to \mathbf{V}(L)^{\sharp}$ is zero because $V(\frac{[v_+^p]}{p})\circ can =0\footnote{We have that $V(y F(x))= x V(y)$, for any $x,y$ in the ring of Witt vectors. In particular, if $x$ is in Frobenius kernel then $V(y) x=0$.}$. This defines $\Psi_{\mathrm{Maz}}$. The following commutative diagram shows that $\Psi_{\mathrm{Maz}}$ is an isomorphism.  
\[\begin{tikzcd}
    0\arrow[]{r}{}&\mathbf{V}(L)^{\sharp}\arrow[]{r}{}\arrow[equal]{d}[swap]{}&M_{\mathfrak{p}_{\mathrm{cris}},S}\arrow[]{r}{(0,F)}\arrow[]{d}{\Psi_{\mathrm{Maz}}}&F_*W_S\arrow[]{r}{}\arrow[equal]{d}{}&0\\
    0\arrow[]{r}{}&\mathbf{V}(L)^{\sharp}\arrow[]{r}{}
    &M_{\mathfrak{p}_{\bar{\dR}},S}\arrow[]{r}{(0,\Id)}
    &F_*W_S\arrow[]{r}{}
    &0
\end{tikzcd}\]
Finally, one verifies that $\Psi_{\mathrm{Maz}}$ commutes with $\xi$ given by \eqref{eqn:defining-xi-dR} and \eqref{eqn:defining-xi-cris}. This proves the existence of the isomorphism in Theorem \ref{Drinfeld-prop}. 

For the uniqueness, it suffices to show that the point $\mathfrak{p}_{\bar{\dR}}\circ t\circ a: \Spf B/\bG_m\to W(k)^{\cN}$ has no nontrivial automorphisms. Since $\Spf B\to \Spf B/\bG_m$ is a faithfully flat cover, it suffices to show this for the corresponding point $s:\Spf B\to W(k)^{\cN}$. 
Explicitly, $s$ is given by the morphism $\bG_a^{\sharp}\oplus F_*W_B\xrightarrow{(v_-,V)}W_B$ of $W_B$-modules. 
Using  \cite[Proposition 5.2.1]{bhatt-lecture-notes}, any automorphism $\alpha$ of $W_B$-module $\bG_a^{\sharp} \oplus F_* W_B$ has the form $\begin{bmatrix}
t & x \\
0 & y
\end{bmatrix}$ for $t \in \bG_a(B), x \in \ker(\bG_a^{\sharp}(B) \to \bG_a(B)) = 0, y \in F_* W(B)$. The group of automorphisms of $s$ consists of $\alpha$ satisfying $(v_-, V) \circ \alpha = (v_-, V).$ In particular, for any $(a, b) \in \bG_a^{\sharp}(B) \oplus F_* W(B),$ we must have $v_{-} ta + V(by) = v_{-} a + V(b).$ Substitutions $a = 0, b=1$ and $a=[v_{+}], b=0$ give $t=y=1$ {\it i.e.}, $\alpha=\Id$.
\end{proof}
\subsection{A description of $k^{\cN}\times_{W(k)^{\cN}} \bA^1_-/\bG_m$}\label{subsection:description-fibre-product} 

\noindent Consider the Cartesian square
\begin{equation}\label{diagram:fibre-product-commdiagram}
\begin{tikzcd}
    &k^{\cN}\times_{W(k)^{\cN}}\bA^1_-/\bG_m\arrow[]{ld}{}\arrow[]{rd}{}&\\
    k^{\cN}\arrow[]{rd}[swap]{\mathfrak{p}_{\mathrm{cris}}}&&\bA^1_-/\bG_m\arrow[]{ld}{\mathfrak{p}_{\bar{\dR}}}\\
    &W(k)^{\cN}&
\end{tikzcd}
\end{equation}
Using Theorem \ref{Drinfeld-prop}, we have a morphism 
\begin{equation}\label{eqn:fibre-product-cN-version}
(a,t\circ a):\Spf B/\bG_m\to k^{\cN}\times_{W(k)^{\cN}}\bA^1_-/\bG_m:=\mathfrak{D}.\end{equation}

Note that the composition $\bA^1_-/\bG_m\rar{\mathfrak{p}_{\bar{\dR}}} W(k)^{\cN}\rar{t_{W(k)}} \bA^1_-/\bG_m$ is the identity. Indeed, recall that $\mathfrak{p}_{\bar{\dR}}$ takes $x=(L, v_-) \in \bA^1_{-}/\bG_m(R)$ to an admissible module $(v_{-}^{\sharp}, V): \mathbf{V}(L)^{\sharp} \oplus F_*W_R \to W_R$ as in \eqref{eqn:defining-xi-dR}. This gives rise to a morphism of admissible sequences \eqref{eq:defofadm} where the left vertical map is $v_{-}^{\sharp}: \mathbf{V}(L)^{\sharp} \to \bG_{a, R}^{\sharp}.$ In particular, the Rees map sends $\mathfrak{p}_{\bar{\dR}}(x)$ to $x.$ Now, giving a point an $S$-point of $\mathfrak{D}$ is equivalent to giving $y \in \bA^1_{-}/\bG_m(S)$ and $x \in k^{\cN}(S)$ together with $\mathfrak{p}_{\bar{\dR}}(y) \simeq 
 \mathfrak{p}_{cris}(x)$. Applying to it $t_{W(k)}$ we get $t(x) \simeq y.$ Thus, giving an $S$-point of $\mathfrak{D}$ is equivalent to giving an $S$-point $x$ of $k^{\cN}$ together with an isomorphism $\mathfrak{p}_{\bar{\dR}}\circ t(x)\simeq \mathfrak{p}_{\mathrm{cris}}(x)$. 

Let us define a closed substack of $\mathfrak{D}$. Note that both $\mathfrak{p}_{\bar{\dR}}$ and $\mathfrak{p}_{\mathrm{cris}}$, precomposed with the structure map $\pi: W(k)^{\cN} \to W(k)^{\Prism}$, factor through the de Rham point $\mathfrak{p}_{\dR}: \Spf W(k) \to W(k)^{\Prism}$, classifying the Cartier-Witt divisor given by $W \xrightarrow{\cdot p} W.$ This defines a map $\tilde{\pi}: \mathfrak{D} \to \Spf W(k) \times_{W(k)^{\Prism}} \Spf W(k).$ 
The stack on the right has a canonical point given by the diagonal map $\Spf W(k)\to \Spf W(k) \times_{W(k)^{\Prism}} \Spf W(k)$.
Consider the fiber over this point of the morphism $\tilde{\pi}$ and denote it by $\mathfrak{D}_{0}.$

The following remark will not be used in the remaining part of the paper. 
\begin{rem}\label{rem:Drinfeld-commutativity-interpretation}
    Drinfeld suggested to us another interpretation of $\mathfrak{D}$ as well as its substack $\mathfrak{D}_{0}$ that makes use of
    the $c$-stack enhancement of $W(k)^{\cN}$ introduced in \cite{drinfeld2022prismatization} and reviewed in \S \ref{subsec:preliminary-syntomification-section}.
    For a scheme $S$, consider the category $\mathfrak{D'}(S)$ consisting of pairs $(t, f)$, where $t: S \to \bA_{-}^1/\bG_m$ and $f: p_+ \to \mathfrak{p}_{\bar{\dR}}(t)$ is a morphism in the category $W(k)^{\cN, c}(S)$ from \S \ref{subsec:preliminary-syntomification-section}. Here $p_+$ is defined as the composition $S \to k^{\cN} \xrightarrow{\mathfrak{p}_{cris}} W(k)^{\cN}$, where the first morphism is given by $A \to H^0 (S, \mathcal{O}_{S})$ sending $v_-$ to $p$ and $v_+$ to $1$. We claim that the underlying groupoid of this category is equivalent to $\mathfrak{D}(S).$ This follows from the fact (\cite[\S 5.5.3]{drinfeld2022prismatization}) that 
    the Rees maps $t_{W(k)}:  W(k)^{\cN, c} \to (\bA_{-}^1/\bG_m)^c$, $k^{\cN, c} \to (\bA_{-}^1/\bG_m)^c $ are left fibrations of $c$-stacks. Indeed, since $W(k)^{\cN, c} \to (\bA^1_{-}/\bG_m)^c$ is a left fibration, the category $W(k)^{\cN, c}(S)_{p_+/}$ is identified with $(\bA^1_{-}/\bG_m)^c(S)_{p/} \simeq k^{\cN, c}(S).$\footnote{An object of $(\bA^1_{-}/\bG_m)^c(S)_{p/}$ consists of $(L \rar{v_{-}} \cO_{S}) \in (\bA^1_{-}/\bG_m)^c(S)$ together with a morphism $f: (\cO_{S} \rar{p} \cO_{S})\to (L \rar{v_{-}} \cO_{S})$ in $(\bA^1_{-}/\bG_m)^c(S)$. Giving such an $f$ is equivalent to specifying a map $v_{+}: \cO_{S} \to L$ with $v_{-}v_{+}=p.$ } Under this identification, the natural functor $W(k)^{\cN, c}(S)_{p_+/} \to W(k)^{\cN, c}$ is identified with $\mathfrak{p}_{cris}.$ The category $\mathfrak{D}'(S)$ is, by definition, equivalent to $W(k)^{\cN, c}(S)_{p_+/} \times_{W(k)^{\cN, c}(S)} (\bA^1_{-}/\bG_m)^c(S)$. Its underlying groupoid is $\mathfrak{D}(S)$ by the definition of the latter \eqref{diagram:fibre-product-commdiagram}. To describe 
    the substack  $\mathfrak{D}_{0},$ note that any object $(t, f)$ of this category provides us with  a triangle in $W(S)^{\cN, c}$ 
\begin{equation}\label{diagram:drinfeld}
\begin{tikzcd}
	{p_+}\arrow[]{rr}{}\arrow[dotted]{rd}{f} && {p_-} 
	\\
	& {\mathfrak{p}_{\bar{\dR}}(t)}\arrow[]{ru}{}&
\end{tikzcd}\end{equation} 
where $p_-$ is defined  by  $A \to H^0 (S, \mathcal{O}_{S})$,  $v_- \mapsto 1$, $v_+ \mapsto p$, and  
the solid arrows are the canonical morphisms (constructed by observing that, for any $S$, $p_+$ (resp. $p_-$) is the initial  (resp. final) object of  the category $k^{\cN, c}(S)$ and $p_-$ is the final object of  $(\bA_{-}^1/\bG_m)^c(S)$).  
The triangle is not commutative in general;  points that correspond to commutative triangles are classified by a substack of $\mathfrak{D}$; 
this substack is equivalent to $\mathfrak{D}_{0}.$ 

Recall from \cite[\S 1.8]{drinfeld2022prismatization} that every {\it effective} gauge $\cF \in \cD_{qc, [0, +\infty]} (W(k)^{\cN})$ gives
rise to a complex of contravariant $\cO$-modules on  $W(k)^{\cN, c}$. In particular, for every $S$-point of $\mathfrak{D}_{0}$
 the canonical morphism $\cF|_{p_-} \to \cF|_{p_+}$ factors through $\cF|_{\mathfrak{p}_{\bar{\dR}}(t)}$.
We also note that, for an effective $F$-gauge $\cF$, we have an isomorphism $F^*(\cF|_{p_+}) \simeq \cF|_{p_-}$ whose composition with  $\cF|_{p_-} \to \cF|_{p_+}$ is the crystalline Frobenius.
\end{rem}
\vspace{3mm} 

In this subsection, we give an explicit description of $\mathfrak{D}$, which was explained to us by Drinfeld and Lurie. The authors of this paper are responsible for any possible mistakes in the exposition of their results.

We endow $W$ with the $\bG_m$-action given by the formula: $\lambda*w=[\lambda]w$, where $[\cdot]$ refers to the Teichm\"uller representative. This gives an action of $\bG_m$ on $\ker(W\xrightarrow{F}W)=:W^{(F)}\hookrightarrow W$. We consider the fibre product $\Spf A\times_{\bA^1_+}W^{(F)}= :\widetilde{\mathfrak{D}}'$, where $\Spf A$ is viewed as a scheme over $\bA^1_+$ via the natural embedding $W(k)[v_+]\hookrightarrow A$, and $W^{(F)}\to\bA^1_+$ is given by the first Witt vector coordinate\footnote{{\it i.e.},~the map sends a Witt vector $w= \sum_{i\geq 0} V_i[w_i]$ to $w_0$.}. We endow $\Spf A\times_{\bA^1_+} W^{(F)}$ with the diagonal $\bG_m$-action.  
Let $\widetilde{\mathfrak{D}}'_0\hookrightarrow\Spf A\times_{\bA^1_+}W^{(F)}$ be the closed subscheme given by equation\footnote{
Set $p-V(1)= \sum_{i\geq 0} V_i[a_i]$, $w= \sum_{i\geq 0} V_i[w_i]$. Then the equation \eqref{eqn:defining-Spf-Btilde} reads as follows: $v_-^{p^{i}}w_i=a_i$, for every $i\geq 0$.} 
\begin{equation}\label{eqn:defining-Spf-Btilde}
[v_-]w=p-V(1).
\end{equation}
Recall from \cite[Lemma 3.2.6]{drinfeld2022prismatization} the unique isomorphism $W^{(F)}\simeq \bG_a^{\sharp}$ characterized by the commutative diagram 
\begin{equation}
\begin{tikzcd}
    W^{(F)}\arrow[]{r}{\simeq}\arrow[]{d}{}&\bG_a^{\sharp}\arrow[]{ld}{}\\
    \bG_a&
\end{tikzcd}\end{equation}
where the vertical map is induced by the projection $W \to \bG_a$ to the first coordinate. 
Thus under the isomorphism $W^{(F)}(\bZ_p)\simeq \bG_a^{\sharp}(\bZ_p)=p\bZ_p\subset \bZ_p$, the element $p-V(1)$ is sent to the number $p$. Consequently, under the identification $\Spf A\times_{\bA^1_+} W^{(F)}\simeq \Spf A\times_{\bA^1_+} \bG_a^{\sharp}$, the equation \eqref{eqn:defining-Spf-Btilde} has the form 
$v_-\cdot z=p$, 
where $z$ is a point of $\bG_a^{\sharp}$, and the product on the left-hand side refers to the $\bG_a$-module structure on $\bG_a^{\sharp}$. 
This yields a concrete description of  $\widetilde{\mathfrak{D}}'_0$ as the formal spectrum  of the  $W(k)$-algebra generated by $v_-$ and $\gamma_{p^n}(v_+)$, $n\geq 0$,
subject to the following relations:
\begin{equation}\label{eq:equationforPDenvelope}
\frac{p^n!}{(p^{n-1}!)^p} \gamma_{p^n}(v_+)= (\gamma_{p^{n-1}}(v_+))^p,  \quad   v_-^{p^n} \gamma_{p^n}(v_+)= \frac{p^{p^n}}{p^n!}.
\end{equation}
\begin{lm}\label{lem:3.10}
\begin{enumerate}
\noindent\item Let $B^{\flat}$ be the  PD-envelope of the ideal  $(v_+)$ in $A$, and let
$W(k)\langle v_+\rangle[v_-]$ be the  PD-envelope of the ideal  $(v_+)\subset W(k)[v_+,v_-]$. 
Then $B^{\flat} \iso W(k)\langle v_+\rangle[v_-]/{(v_+v_- -p)}$ and   $\widetilde{\mathfrak{D}}' \xrightarrow{\sim} \Spf B^{\flat}$.
\item Let $B^{\dagger}$ be the  PD-envelope of the ideal  $(v_+)$ in $A$ relative to $(\bZ_p,(p))$ \cite[07H7]{stacks-project}.
  Then $\widetilde{\mathfrak{D}}'_0 \xrightarrow{\sim} \Spf B^{\dagger}$.
  \end{enumerate}
\end{lm}
\begin{proof}
  To prove the first part, set  $\widetilde A^\sharp:= W(k)\langle v_+\rangle[v_-]$, $\widetilde A: = W(k)[v_+,v_-]$,
   and let $J\subset \widetilde A^\sharp$ be the PD ideal generated by $\frac{v_+^n}{n!}$, $n\geq 1$. Set $I:= (v_-v_+ -p) \subset \widetilde A^\sharp$. Then\footnote{Indeed, given any $f =(v_-v_+ -p)g 
\in I\cap J$, we conclude, using that $v_-v_+ -p$ and $v_+$  are coprime in $\tilde A \otimes \bQ$, that $g$ is divisible by $v_+$ in $\tilde A \otimes \bQ$. Hence, $g\in J$.} $I\cap J = I \cdot J $. In particular,  $I\cap J\subset J$ is a sub-PD-ideal. Consequently, there exists a unique PD structure on the image $\bar J$ of $J$ under the projection
$\widetilde A^\sharp \to \widetilde A^\sharp/I$. We claim that the map $(A, (v_+)) \to (\widetilde A^\sharp/I, \bar J)$ exhibits the right-hand side as the PD-envelope  $B^\flat$ of $(v_+)\subset A$. Indeed, by the universal property we have a PD-homomorphism $B^\flat \to \widetilde A^\sharp/I$. On the other hand, the projection $\widetilde A \to A$ gives a PD map 
$\widetilde A^\sharp \to B^\flat$ sending $I$ to $0$. By looking at generators, one readily checks that these are mutually inverse homomorphisms. Finally, we note that
$\widetilde A^\sharp/I = A \otimes _{W(k)[v_+]} W(k)\langle v_+\rangle$. Using the isomorphism  $\bG_a^\sharp \xrightarrow{\sim} W^{(F)}$, this implies the claim. 

To prove the second assertion, observe  that $B^{\dagger}$ is identified with  the tensor product 
$B^{\dagger} \otimes_{D_{\bZ_p}((p))} \bZ_p,$ where $D_{\bZ_p}((p))$ is the PD envelope of $(p)\subset \bZ_p$ and the homomorphism $D_{\bZ_p}((p)) \to \bZ_p$ takes each 
$\gamma_n(p) \in D_{\bZ_p}((p))$ to $\frac{p^n}{n!}$. 
Using \eqref{eq:equationforPDenvelope} the claim follows.
\end{proof}
\begin{pr}\label{prop:fibre-product-tildefrakT}

There is a $\bG_m$-equivariant isomorphism 
$\widetilde{\mathfrak{D}}:=\mathfrak{D}\times_{\bA^1_-/\bG_m}\bA^1_- \xrightarrow{\sim} \widetilde{\mathfrak{D}}'$ that identifies the substack $\widetilde{\mathfrak{D}}_{0} := \mathfrak{D}_{0}\times_{\bA^1_-/\bG_m}\bA^1_-$ with $\widetilde{\mathfrak{D}}'_0$.
Here the action of $\bG_m$ on $\widetilde{\mathfrak{D}}$ comes from the $\bG_m$-action on $\bA^1_-$.

\end{pr}
\begin{proof}
We start with a preliminary observation.~The line bundle $\mathcal{O}_{\bA^1_-/\bG_m}(-1)$ determines via the Rees map $t_{W(k)}: W(k)^{\cN} \to \bA_{-}^1/\bG_m$     
a $\bG_m$-torsor $ {W(k)}^{\cN, r} \to W(k)^{\cN}$.~Thus an $S$-point of ${W(k)}^{\cN, r}$ is a pair consisting of an $S$-point of ${W(k)}^{\cN}$ together with a trivialization of the pullback of $\mathcal{O}_{\bA^1_-/\bG_m}(-1)$ to $S$.~Observe $\bG_m$-equivariant isomorphisms $ {W(k)}^{\cN, r} \times_{{W(k)}^{\cN}} k^\cN \xrightarrow{\sim} \Spf A$
and $ {W(k)}^{\cN, r} \times_{{W(k)}^{\cN}} \bA^1_-/\bG_m \xrightarrow{\sim} \bA^1_-$.
 Consequently, we obtain an isomorphism 
$\Spf A \times_{{W(k)}^{\cN, r}} \bA_{-}^{1}  \xrightarrow{\sim} \widetilde{\mathfrak{D}}$ of $\bG_m$-torsors over $\mathfrak{D}$.
In particular, we conclude that the projection $\mathfrak{D} \to \Spf A/\bG_m$ lifts to a map $\widetilde{\mathfrak{D}} \to \Spf A$. 

Now we prove the proposition.
Let $R$ be a $p$-nilpotent ring.~Recall from \eqref{diagram:fibre-product-commdiagram} that a point in $\mathfrak{D}(R)$ consists of a point $x \in k^{\cN}(R)$ together with $\mathfrak{p}_{cris}(x) \simeq \mathfrak{p}_{\bar{\dR}} \circ t(x).$ Since $\widetilde{\mathfrak{D}} := \mathfrak{D} \times_{\bA^1_{-}/\bG_m} \bA^1_{-}$, we see that a point of $\widetilde{\mathfrak{D}}(R)$ consists of 
a homomorphism $g:A\to R$ together with an isomorphism $\beta$ of $W_R$-modules 
\begin{equation}\label{eqn:defining-Mcris-beta-MdR}
 \coker(\bG_{a,R}^{\sharp}\xrightarrow{(g(v_+),-\mathrm{can})} \bG_{a,R}^{\sharp}\oplus W_R):=M_{\mathfrak{p}_{\mathrm{cris},R}}\underset{\simeq}{\xrightarrow{\beta}} M_{\mathfrak{p}_{\bar{\dR},R}}:=\bG_{a,R}^{\sharp}\oplus F_*W_R
\end{equation} 
such that the following two diagrams commute
\begin{equation}\label{eqn:prop-fibre-product-commdiag-beta-bis}
 \begin{tikzcd}
    0 \arrow[]{r}{} & \bG_{a,R}^{\sharp} \arrow[]{r}{(\Id, 0)}[swap]{}\arrow[]{d}[swap]{\Id}& M_{\mathfrak{p}_{\mathrm{cris},R}} \arrow[]{d}{\beta}\arrow[]{r}{(0, F)}[swap]{}& F_*W_R
    \arrow[]{d}[swap]{\nu} \arrow[]{r}{}[swap]{}& 0\\
    0 \arrow[]{r}{}  &\bG_{a,R}^{\sharp} \arrow[]{r}{(\Id, 0)}[swap]{}   &    M_{\mathfrak{p}_{\bar{\dR},R}} \arrow[]{r}{(0, \Id)}[swap]{} &  F_*W_R  \arrow[]{r}{}[swap]{}& 0
 \end{tikzcd}
 \end{equation}
\begin{equation}\label{eqn:prop-fibre-product-commdiag-beta}
 \begin{tikzcd}
     M_{\mathfrak{p}_{\mathrm{cris},R}}\arrow[]{r}{\beta}[swap]{}\arrow[]{d}[swap]{(g(v_-)\circ\ \mathrm{can},p)}&M_{\mathfrak{p}_{\bar{\dR},R}}\arrow[]{d}{(g(v_-)\circ\ \mathrm{can},V)}\\
     W_R\arrow[equal]{r}{\Id}&W_R,
 \end{tikzcd}
 \end{equation}
for some $\nu \in \Hom_{W_R}(W_R,W_R)= W(R)$.
 Explicitly, $\beta': \bG_{a,R}^{\sharp}\oplus W_R\to M_{\mathfrak{p}_{cris, R}} \xrightarrow{\beta} \bG_{a,R}^{\sharp}\oplus F_*W_R$ is given by $\begin{pmatrix}
    1 &w\\
    0&\nu \circ F
\end{pmatrix}$, where $w=(w_1,w_2,\cdots)\in \Hom_{W_R}(W_R,\bG_{a,R}^{\sharp})=W^{(F)}(R)$ is subject to the equation $w_1=g(v_+)$ and $[g(v_-)]w=p-V(\nu)$. 
Sending a point of $\widetilde{\mathfrak{D}}(R)$ to $(g\in (\Spf A)(R), w\in W^{(F)}(R))$  defines a functorial map 
\begin{equation}\label{eq:constuction_of_the_map}
\widetilde{\mathfrak{D}}(R)\to (\Spf A\times_{\bA^1_+} W^{(F)}) (R).
\end{equation} 
We next show that \eqref{eq:constuction_of_the_map} is an isomorphism. To do this we shall construct its inverse.

Given a point $(g,w) \in (\Spf A\times_{\bA^1_+} W^{(F)})(R)$, observe that since $w_1=g(v_+)$ the first coordinate of $p-[g(v_-)]w$ is $0$. Therefore this element can be uniquely written as $V(\nu)$, for some $\nu \in W(R)$.
Then the matrix $\begin{pmatrix}
    1&w\\
    0&\nu \circ F
\end{pmatrix}$ 
defines a homomorphism  $M_{\mathfrak{p}_{\mathrm{cris},R}} {\xrightarrow{\beta}} M_{\mathfrak{p}_{\bar{\dR},R}}$ making the diagrams 
 \eqref{eqn:prop-fibre-product-commdiag-beta-bis} and \eqref{eqn:prop-fibre-product-commdiag-beta} commutative. 
It remains to check that $\beta$ is an isomorphism. Using the diagram \eqref{eqn:prop-fibre-product-commdiag-beta-bis} it is enough to check that $\nu \in W(R)$ is invertible. 
Let $\nu_1$ be the first ghost coordinate of $\nu$. If $\nu$ is not invertible 
we can find a maximal ideal $\mathfrak{m} \subset R$ which contains $\nu_1.$ Since $W^{(F)}(R/\mathfrak{m})=0$ and $w \in W^{(F)}(R)$, we conclude from the equation $[g(v_{-})]w=p-V(\nu)$ that $\nu_1=1$ in $R/\mathfrak{m}$ contradicting to $\nu_1 \in \mathfrak{m}$.~Thus $\beta$ is an isomorphism.~This defines a functorial map 
$$(\Spf A\times_{\bA^1_+} W^{(F)}) (R) \to \widetilde{\mathfrak{D}}(R)$$ which is inverse to \eqref{eq:constuction_of_the_map}.

For the second part, note that $\beta = \begin{pmatrix}
    1&w\\
    0&\nu \circ F
\end{pmatrix}$ $\in \widetilde{\mathfrak{D}}(R)$ lands in $\widetilde{\mathfrak{D}}_{0}$ if and only if  $\nu =1$ and the result follows.
\end{proof}
\begin{rem}
Another way to compute the fiber product $k^{\cN} \times_{W(k)^{\cN}} \bA^1_{-}$ relies on the computation of $p$-adically completed derived de Rham cohomology $\widehat{\dR}(k/W(k))$ by Bhatt \cite[Proposition 8.5]{bhatt2012padic}, namely $\widehat{\dR}(k/W(k)) = \widehat{W(k)\langle x \rangle} /(x-p)$, where $\widehat{W(k)\langle x \rangle}$ stands for the $p$-completed PD-envelope of $(x)\subset W(k)[x]$. Moreover, he defines a descending filtration on $\widehat{W(k)\langle x \rangle}$ whose $n$-th term is topologically generated by elements of the form $\gamma_i (x)$ with $i \ge n$. The induced  filtration $\operatorname{Fil}^{\bullet}$ on $\widehat{W(k)\langle x \rangle} /(x-p)$  matches the derived Hodge filtration $\operatorname{Hod}^{\bullet}$ under this isomorphism. Using the above result we shall construct a $\bG_m$-equivariant isomorphism  between the derived fiber product $k^{\cN} \times_{W(k)^{\cN}}^{\bf L} \bA^1_{-}$ 
and  the derived $p$-adic formal stack represented by the derived $p$-completion of  $B^{\dagger}$.
Recall that, for any derived $p$-adic formal stack $X$ over $W(k)$, the Hodge filtered derived de Rham stack $(X/W(k))^{\dR, +}$ is identified with the fiber product $X^{\cN} \times_{W(k)^{\cN}}^{\bf L} \bA^1_{-}/\bG_m$.  We apply this to $X =\Spec k$.   
To complete the proof observe that  $(k/W(k))^{\dR, +}$ is affine over $\bA^1_{-}/\bG_m$ in the sense of \cite[\S 7.3]{bhatt2022prismatization} and, hence, $(k/W(k))^{\dR, +} \times_{ \bA^1_{-}/\bG_m}^{\bf L} \bA^1_{-}$ is represented by the algebra $R\Gamma \big((k/W(k))^{\dR, +} \times_{ \bA^1_{-}/\bG_m}^{\bf L} \bA^1_{-}, \cO\big )$. The latter is the derived $p$-completion of the 
Rees algebra $\operatorname{Rees}_{\operatorname{Hod}^{\bullet}} \widehat{\dR}(k/W(k))$ which is identified with the derived $p$-completion
of $B^{\dagger}$ using Bhatt's result explained above. 

Note that the derived $p$-completion of $B^{\dagger}$ is the quotient of  the $p$-completion of $W(k)\langle v_+ \rangle [v_-]$
by the ideal $(v_+v_- -p)$. Using \cite[Remark 6.16]{Bhatt2018} it follows that this quotient is not  $p$-adically separated and, in particular, it is not 
$p$-complete. Thus, the derived stack $k^{\cN} \times_{W(k)^{\cN}}^{\bf L} \bA^1_{-}$ is not classical.
\end{rem}

\begin{rem}
    Proposition \ref{prop:fibre-product-tildefrakT} gives another proof of Theorem \ref{Drinfeld-prop}. Indeed, since $B$ is $p$-torsionfree and $v_+$ has divided powers in $B$,
    the embedding $A\hookrightarrow B$ extends uniquely to a homomorphism of graded $W(k)$-algebras $B^{\dagger} \to B$. This defines a  $\bG_m$-equivariant map
    $\Spf B  \to \widetilde{\mathfrak{D}}'_0 \iso \widetilde{\mathfrak{D}}_0$. In particular, we obtain a morphism $\Spf B/\bG_m \to  k^{\cN}\times_{W(k)^{\cN}} \bA^1_-/\bG_m$ whose composition with the projection to $ k^{\cN}$ is given by $A\hookrightarrow B$. This is equivalent to Theorem \ref{Drinfeld-prop}.
    \end{rem}
\begin{rem}\label{rem:k-syn-times-A1modGm}
Observe that the morphism
\begin{equation}\label{eq:nvrsyn}
   \mathfrak{D}=k^{\cN}\times_{W(k)^{\cN}}\bA^1_-/\bG_m \to  k^{syn}\times_{W(k)^{syn}}\bA^1_-/\bG_m
\end{equation}
induced by maps $k^{\cN} \to k^{syn}$ and $W(k)^{\cN} \to W(k)^{syn}$ is an isomorphism. Indeed, since the maps $k^{\cN} \to k^{syn}$, 
$W(k)^{\cN} \to W(k)^{syn}$ are \'etale,  the morphism \eqref{eq:nvrsyn} is also \'etale. Thus it suffices to check that \eqref{eq:nvrsyn} induces an equivalence on points with values in every algebraically closed field $K$ of characteristic $p$. This follows from \cite[Propositions 5.16.5 and 8.10.4 ({\rm iv})]{drinfeld2022prismatization} which assert that the maps $k^\cN \to W(k)^\cN$, $k^{syn} \to W(k)^{syn}$ induce equivalences on $K$-points. 
\end{rem}

\section{Construction of the functor $\cD_{qc,[0,p-1]}(W(k)^{syn}) \to \mathscr{DMF}^{\mathrm{big}}_{[0,p-1]}(W(k))$}

\subsection{Corollary of Drinfeld's Lemma}

The isomorphism $\Cris\circ a\simeq \mathfrak{p}_{\bar{\dR}}\circ\Reesparameter\circ a$ from 
Theorem \ref{Drinfeld-prop} gives an isomorphism 
of functors $\cD_{qc}(W(k)^{\cN})\to \cD_{qc}(\Spf B/\bG_m)$
\begin{equation}\label{cor_of_Drinfeld-prop}
a^*t^*\mathfrak{p}^*_{\bar{\dR}}\simeq a^*\Cris^*.
\end{equation}
We shall see that after restriction to $\cD_{qc,  [0,p-1]}(W(k)^{\cN})$ a stronger statement holds.
\begin{Th}\label{thm-crho-equal-b-pullback} There exists
a unique isomorphism 
$t^*\mathfrak{p}^*_{\bar{\dR}}\simeq \Cris^*$ 
of functors $\cD_{qc,  [0,p-1]}(W(k)^{\cN})\to \cD_{qc, [0,p-1]}(k^\cN)$ 
whose pullbacks  
recover 
\eqref{cor_of_Drinfeld-prop}.
\end{Th}
\begin{proof} Observe that indeed the functors $t^*\mathfrak{p}^*_{\bar{\dR}}$ and 
$\Cris^*$ carry $\cD_{qc,  [0,p-1]}(W(k)^{\cN})$ to the subcategory $\cD_{qc,  [0,p-1]}(k^{\cN})\subset \cD_{qc}(k^{\cN})$: the assertion about $t^*\mathfrak{p}^*_{\bar{\dR}}$ follows from 
Lemmas \ref{c*-preserve-weights} and \ref{rho-dR-preserve-weights} combined with Remark \ref{remark-two-notions-effectivity}; the assertion about $\Cris^*$ 
is true by definition of weights.

We will prove the theorem by showing that 
\begin{equation}\label{eq:fully_faithful}
 a^*:\cD_{qc, [0,p-1]}(k^\cN) \to \cD_{qc}(\Spf B/\bG_m) 
\end{equation}
is fully faithful. 
We start with the following.
\begin{lm}\label{effectivity-lemma}
    Let $\begin{tikzcd}\cdots\cN^i\arrow[shift left]{r}{v_-}&\cN^{i-1}\arrow[shift left]{l}{v_+}\arrow[shift left]{r}{v_-}&\cdots\arrow[shift left]{l}{v_+}\end{tikzcd}$ be an effective object of $\cD_{qc}(k^{\cN})$. Then the map\footnote{For a complex $M$ of abelian groups, we denote by $M^{\wedge}$ its derived $p$-adic completion.} $\cN^i\to\left((\cN^{\bullet}\overset{\mathbb{L}}{\otimes}_AB)^{\wedge}\right)^i$ induced by the homomorphism of graded algebras $A\to B$ is an isomorphism for $i<p$. 
\end{lm}
\begin{proof}
Equivalently, it suffices to show that $\left((\cN^{\bullet}\overset{\mathbb{L}}{\otimes}_A B/A)^{\wedge}\right)^i$ is acyclic for $0\leq i\leq p-1$. Moreover, by derived $p$-adic completeness, it suffices to prove that $(\cN^{\bullet}\overset{\mathbb{L}}{\otimes}_A (B/A\overset{\mathbb{L}}{\otimes}_{W(k)}k))^i$ is acyclic for $0\leq i\leq p-1$. The complex $B/A\overset{\mathbb{L}}{\otimes}_{W(k)}k$ is supported in two cohomological degrees: $0$ and $-1$. Using Lemma \ref{Bi-free-module}, the $A \otimes k$-module $B/A\otimes_{W(k)}k$ has grading degrees 
greater than or equal to $p$ and is supported at the origin, {\it i.e.},~every element $x\in B/A \otimes k$ is killed by a power of the maximal ideal $\mathfrak{m}\subset A \otimes k$. Thus, $B/A \otimes k$ admits an exhaustive filtration $L_0 \subset L_1 \subset \cdots \subset B/A \otimes k$ by $A \otimes k$-submodules
with each $L_i/L_{i-1}$ of the form $k\{j\}$ with $j\geq p$. 
That $\cN^{\bullet}$ is effective implies, using Remark \ref{remark-two-notions-effectivity}, that $(\cN^{\bullet}\overset{\mathbb{L}}{\otimes}_Ak)^i$ is acyclic for $i< 0$. Equivalently, $(\cN\overset{\mathbb{L}}{\otimes}k\{j\})^i$ is acyclic for $i<j$. Thus  
 $(\cN^{\bullet}\overset{\mathbb{L}}{\otimes}_A (B/A \otimes _{W(k)}k))^i$ is acyclic for $i< p$.

To compute $\Tor_1^{W(k)}(B/A,k)=\ker(B/A\xrightarrow{p}B/A)$: by Lemma \ref{Bi-free-module} again, one has 
$\ker(B^i/A^i\xrightarrow{p}B^i/A^i)=\frac{1}{p}v_+^ik$ for $i\geq p$ and $0$ otherwise.~Thus $\Tor_1^{W(k)}(B/A,k)= \left(A/(v_-)\right)\{p\}$ as an $A$-module.~We have that
$$(\cN^{\bullet}\overset{\mathbb{L}}{\otimes}_A \Tor_1^{W(k)}(B/A,k))^i\xrightarrow{\sim} \Cone(\cN^{i-p+1} \rar{v_-} \cN^{i-p}).$$  
Using effectivity, the cone above is acyclic for $i<p$.
\end{proof}

\begin{cor}\label{cor-to-truncation}
For every  $M\in  \cD_{qc,[0, +\infty]}(k^{\cN})$,
the natural transformation $\Id \to a_*a^*$ induces an isomorphism $w_{\leq p-1} M\xrightarrow{\sim} w_{\leq p-1}a_*a^* M$. 
\end{cor}
Now we can prove that \eqref{eq:fully_faithful} is fully faithful. For any $M,M'\in \cD_{qc,[0,p-1]}(k^{\cN})$, 
\begin{align*}
\mathrm{RHom}(M,M')&=\mathrm{RHom}(M,w_{\leq p-1}M')\simeq \mathrm{RHom}(M,w_{\leq p-1}a_*a^*M')\\
&\simeq \mathrm{RHom}(M,a_*a^*M')\simeq  \mathrm{RHom}(a^*M,a^*M'),
\end{align*}
where the second and fourth isomorphisms follow by adjunction, and the third isomorphism follows by Corollary \ref{cor-to-truncation}. This proves fully faithfulness and thus completes the proof of the theorem. 
\end{proof}
Let us record a generalization of Theorem \ref{thm-crho-equal-b-pullback}. 
\begin{pr}  There exists
an isomorphism $w_{\leq p-1} t^*\mathfrak{p}^*_{\bar{\dR}}\simeq w_{\leq p-1} \Cris^*$ of functors $\cD_{qc,  [0, + \infty]}(W(k)^{\cN})\to \cD_{qc, [0, p-1]}(k^\cN).$ 
\end{pr}
\begin{proof} Using Theorem \ref{Drinfeld-prop}, we have $a_*a^*t^*\mathfrak{p}^*_{\bar{\dR}}\cF\simeq a_* a^*\Cris^*\cF$.
Applying $w_{\leq p-1}$ and using Corollary \ref{cor-to-truncation} we obtain the result. 
\end{proof}    
\begin{rem}\label{remark:cris-Hodge-filtration-recovered-from-Nygaard}
 In geometric context, {\it i.e.},~for the prismatic gauge $ \mathcal{H}_{\cN}(X) \in \cD_{qc}(W(k)^\cN)$ associated to a smooth $p$-adic formal scheme $X$ over $W(k)$, \eqref{cor_of_Drinfeld-prop} gives
\begin{equation}\label{cor_of_Drinfeld-prop_geom}
a^*t^*\mathfrak{p}^*_{\bar{\dR}} (\mathcal{H}_{\cN}(X))\simeq a^*\Cris^* (\mathcal{H}_{\cN}(X)),
\end{equation}
which expresses a relation between the Nygaard filtration on $R\Gamma_{\mathrm{cris}}(X_k)$, where $X_k:=X\times\Spec k$, and the Hodge filtration on $R\Gamma_{\dR}(X)$. In fact, $a^*t^*\mathfrak{p}^*_{\bar{{\dR}}} (\mathcal{H}_{\cN}(X))$, viewed as  a complex of filtered modules over the filtered algebra $(W(k), (p)^{[\cdot]})$ (see Remark \ref{rem:the_category_of_B-modules}), can be identified, using \cite[Theorem 7.2]{Berthelot-Ogus} with the crystalline Hodge filtration on $R\Gamma_{\mathrm{cris}}(X_k)$. Thus, \eqref{cor_of_Drinfeld-prop_geom} says that the crystalline Hodge filtration can be obtained from 
the Nygaard filtration by ``downgrading'' using the functor $a^*$.
\end{rem}

\subsection{Construction of the functor}
Define a functor 
\begin{equation}\label{functor}
  \Phi_{\mathrm{Maz}}:\cD_{qc,[0,p-1]}(W(k)^{syn}) \to \mathscr{DMF}^{\mathrm{big}}_{[0,p-1]}(W(k))  
\end{equation}
as follows. Let $\cF\in \cD_{qc,[0,p-1]}(W(k)^{syn})$ be a prismatic $F$-gauge with Hodge-Tate weights in $[0,p-1]$. We have the following diagram (\cite[Remark 1.4.2]{bhatt-lecture-notes} or \S \ref{subsec:preliminary-syntomification-section})
\begin{equation}\label{dia:a_very_important_dia}
\begin{tikzcd}
&&&\bA^1_{-}/\bG_m\arrow[]{d}{\mathfrak{p}_{\bar{\dR}}}\\
\Spf W(k)\arrow[shift left]{r}{j_-}\arrow[shift right]{r}[swap]{j_+ \circ  F^{-1}}&\Spf A/\bG_m\arrow[]{r}{\simeq }&k^{\cN}\arrow[]{d}{\overline{j}_{\cN}}\arrow[]{r}{\Cris}\arrow[]{ru}{t}&W(k)^{\cN}\arrow[]{d}{j_{\cN}}\\
&&k^{syn}\arrow[]{r}{}&W(k)^{syn}\end{tikzcd}
\end{equation}
where $j_-$ (resp.~$j_+ \circ F^{-1}$) is an open embedding onto the substack given by $v_-\neq 0$ (resp.~$v_+\neq 0$), 
$\bar j_\cN: k^{\cN} \to k^{syn}$ exhibits the target as the coequalizer of  $j_{\pm}:
\begin{tikzcd}\Spf W(k)=k^{\Prism} \arrow[shift left]{r}{}\arrow[shift right]{r}[swap]{}&k^{\cN}\end{tikzcd}$ (see, \cite[Def. 6.1.1]{bhatt-lecture-notes}),  
and the square is commutative. The morphism $\Spf W(k) \xrightarrow{v_- = 1} \bA^1_{-}/\bG_m$ is isomorphic to $t \circ j_-$ and $\Spf W(k) \xrightarrow{v_- = p} \bA^1_{-}/\bG_m$ is isomorphic to $t \circ j_{+} \circ F^{-1}$; see the end of \S \ref{subsec:preliminary-syntomification-section}. By definition 
of $\bar j_\cN$
\begin{equation}\label{j+j-same-after-jbar-uptoF}
\overline{j}_{\cN}\circ j_+ \simeq \overline{j}_{\cN}\circ j_-.
\end{equation}
Let $\cM:=\mathfrak{p}_{\bar{\dR}}^*\circ j_{\cN}^*\cF$. Note that $F^* (\cM|_{v_-=p})\simeq j_+^*\circ\Reesparameter^*\cM$  and $\cM|_{v_-=1}\simeq j_-^*\circ\Reesparameter^*\cM$. 
Then using Theorem \ref{thm-crho-equal-b-pullback}, one has that $\Cris^*\circ j_{\cN}^*\cF\simeq\Reesparameter^*\cM$. Hence we obtain  $F^*(\cM|_{v_-=p})\simeq j_+^*\circ\Cris^*\circ j_{\cN}^*\cF$ and $\cM|_{v_-=1}\simeq j_-^*\circ\Cris^*\circ j_{\cN}^*\cF$. 
By \eqref{j+j-same-after-jbar-uptoF}, we get an isomorphism $F^*(\cM_{v_-=p})\simeq \cM_{v_-=1}$, which we denote by $\varphi$. Define the functor $\Phi_{\mathrm{Maz}}$ in (\ref{functor}) by sending $\cF$ to $(\cM, \varphi )$. 
\begin{rem}\label{remark-construction-functor-Psi-Mazur}
   Observe that
    $\Psi_{\mathrm{Maz}}$ from Theorem \ref{Drinfeld-prop} gives a functor 
    \begin{equation}\label{first-displayed-functor-Phi}
    \cD_{qc}(W(k)^{syn})\to \cD_{qc}(k^{syn})\times_{\cD_{qc}(\Spf B/\bG_m)}\cD_{qc}(\bA^1_-/\bG_m)\end{equation}
    sending $\cD_{qc,[a,b]}(W(k)^{syn})$ to $\cD_{qc,[a,b]}(k^{syn})\times_{\cD_{qc}(\Spf B/\bG_m)}\cD_{qc,[a,b]}(\bA^1_-/\bG_m)$.
We assert that there is a functor 
   \begin{equation}\label{second-displayed-functor-Phi}
\cD_{qc,[0,p-1]}(k^{syn})\times_{\cD_{qc}(\Spf B/\bG_m)}\cD_{qc,[0,p-1]}(\bA^1_-/\bG_m)\to \mathscr{DMF}_{[0,p-1]}^{\mathrm{big}}(W(k))
\end{equation}  
whose precomposition with  \eqref{first-displayed-functor-Phi} is $\Phi_{\mathrm{Maz}}$. 
To see this, observe that the functor
\begin{align*}&\cD_{qc,[0,p-1]}(k^{syn})\underset{\cD_{qc, [0,p-1]}(k^\cN)}{\times}\cD_{qc,[0,p-1]}(\bA^1_-/\bG_m)\to \\
&\cD_{qc,[0,p-1]}(k^{syn})\underset{\cD_{qc}(\Spf B/\bG_m)}{\times}\cD_{qc,[0,p-1]}(\bA^1_-/\bG_m)
\end{align*}
induced by the fully faithful embedding \eqref{eq:fully_faithful} is an equivalence.
We define 
$$ \cD_{qc,[0,p-1]}(k^{syn})\underset{\cD_{qc, [0,p-1]}(k^\cN)}{\times}\cD_{qc,[0,p-1]}(\bA^1_-/\bG_m)\to  \mathscr{DMF}_{[0,p-1]}^{\mathrm{big}}(W(k))$$
 sending an object $(\cF_k,\cM, \upsilon: \bar j^*_\cN \cF_k \iso t^* \cM )$ of the fiber product to
 $(\cM, \varphi )$, where $\varphi: F^*(\cM_{v_-=p})\simeq \cM_{v_-=1}$ is constructed from diagram 
\eqref{dia:a_very_important_dia} and equation \eqref{j+j-same-after-jbar-uptoF}. This defines \eqref{second-displayed-functor-Phi}. The fact that its precomposition with \eqref{first-displayed-functor-Phi} is $\Phi_{\mathrm{Maz}}$ is straightforward. 
\end{rem}
The following remarks will not be used in the remainder of the paper. However, we believe that they are important in their own right.

\begin{rem}[A refinement of the crystalline Frobenius]\label{remark:Mazur-module}
For every object $N\in  \cD_{qc,[0, +\infty]}(k^{\cN}) \subset \widehat{\cD}_{gr}(A)$, we have a canonical morphism 
\begin{equation}\label{c-stack_map}
   N\to  \pi ^* \circ  (j_+\circ F^{-1})^* N,
   \end{equation}
where $\pi: k^{\cN} \to \Spf W(k)$ is the structure map.  
One can define \eqref{c-stack_map} by observing 
that $\pi^*$ is the right adjoint to $(j_+\circ F^{-1})^*\colon \cD_{qc,[0, +\infty]}(k^{\cN}) \to \cD_{qc}(\Spf W(k))$.
Alternatively, \eqref{c-stack_map} can be constructed using Drinfeld's observation that every effective gauge $\cF$ extends to a contravariant $\cO$-module on the $c$-stack $k^{\cN, c}$ and that, for any scheme $S$, the category $k^{\cN, c}(S)$ has initial object given by $p_+$ (see
Remark \ref{rem:Drinfeld-commutativity-interpretation}).

Given an effective $F$-gauge  $\cF\in  \cD_{qc,[0, +\infty]}(W(k)^{syn})$, we have functorial morphisms 
\begin{equation}\label{eq:mazur_prelim}
    a^*t^*{\mathfrak{p}}^*_{\bar{{\dR}}}(\cF) \underset{\eqref{cor_of_Drinfeld-prop}}{\iso} a^*\Cris^* (\cF) \underset{\eqref{c-stack_map}}{\to}  a^* \pi ^* (j_+\circ F^{-1})^*  \Cris^* (\cF) \underset{\eqref{j+j-same-after-jbar-uptoF}}{\xrightarrow{\sim}} a^* \pi ^*   F_*  j^*_- \Cris^* (\cF).
\end{equation}
To rewrite it more suggestively, let $F^{\bullet}= \oplus F^i\in \widehat {\cD}_{gr}(W(k)[v_-])$ be the complex of graded  $W(k)[v_-]$-modules corresponding to $\mathfrak{p}^*_{\bar{{\dR}}}(\cF)$ under the equivalence \eqref{identifying-CohA1modGm-with-modules}. Then the composition \eqref{eq:mazur_prelim}
yields a homomorphism in $\widehat {\cD}_{gr}(W(k)[v_-])$ 
\begin{equation}\label{eq:mazur}
   \varphi_{\bullet}:  F^{\bullet}   \to F_* M \otimes_{W(k)} B,
    \end{equation}
 where $M:=F^0$ and the action of   $W(k)[v_-]$ on $ F_* F^0 \otimes_{W(k)} B$ comes from the homomorphism $W(k)[v_-] \mono B$. Explicitly, \eqref{eq:mazur} amounts to specifying $W(k)$-linear maps
 $\varphi_i:F^i\to F_* M$ for each $i\geq 0$ such that $p^{[i+1]-[i]}\varphi_{i+1}=\varphi_iv_-$ (see Definition \ref{def:Mazur numbers}). We claim that $\varphi_0: M \to F_* M$ coincides with the crystalline Frobenius. Indeed, taking $S=\Spf B/\bG_m$ in Remark \ref{rem:Drinfeld-commutativity-interpretation}, the composition \eqref{eq:mazur_prelim} is obtained by taking fibers of $\cF$ along $p_+ \to \mathfrak{p}_{\bar{\dR}}(t)$ and using $\eqref{j+j-same-after-jbar-uptoF}.$  The map $\Spf B / \bG_m \to \mathfrak{D}$ from \eqref{eqn:fibre-product-cN-version} is induced by the natural map $B^{\flat} \to B$ from the PD envelope of $(v_+)\subset A$
 to $B$; see Lemma \ref{lem:3.10}.  
 Since $B$ has no $p$-torsion, this map factors through the PD envelope $B^{\dagger}$ of $(v_+)\subset A$
 {\it relative} to  $(\bZ_p, (p))$\footnote{Indeed, the map from the PD-envelope of $(p)\subset \bZ_p$ to $\bZ_p$ is an isomorphism after inverting $p$.}.
 It follows that the map $\Spf B / \bG_m \to \mathfrak{D}$ factors through the closed substack $\mathfrak{D}_{0}$.
 Now our claim follows from the commutativity of diagram \eqref{diagram:drinfeld}\footnote{More generally, we infer that $\varphi_{\bullet}$ lifts to 
 a homomorphism of $W(k)[v_-]$-modules  $F^{\bullet}   \to (F_* M \otimes_{W(k)} B^{\flat})^{\wedge}$, where the superscript $\wedge$ stands for the graded $p$-completion.
 In degree zero, the composition $F^{\bullet}   \to (F_* M \otimes_{W(k)} B^{\dagger})^{\wedge}$ of the latter 
  with $(F_* M \otimes_{W(k)} B^{\flat})^{\wedge} \to (F_* M \otimes_{W(k)} B^{\dagger})^{\wedge}$ is homotopic to  the crystalline Frobenius.}.

 Denote by $\mathrm{Mod}_{\mathrm{Maz}}(W(k))$ the $\infty$-category formed of pairs $(F^{\bullet}, \varphi)$, where $F^{\bullet} \in \widehat {\cD}_{gr}(W(k)[v_-])$ is effective and $\varphi: F^{\bullet} \to F_* F^0 \otimes_{W(k)} B$ is a map in $\widehat {\cD}_{gr}^{\eff}(W(k)[v_-])$. More formally, it is defined as the lax equalizer $\operatorname{LEq}(\Id, F_*F^0 \otimes B)$ as in (\cite[Def. II.1.4]{MR3904731}). Here by $F_*F^0 \otimes_{W(k)} B$ we mean the endofunctor of $\widehat \cD_{gr}^{\eff} (W(k)[v_{-}])$ taking $F^{\bullet}$ to $F_* F^0 \otimes_{W(k)} B.$ We refer to objects of $\mathrm{Mod}_{\mathrm{Maz}}(W(k))$ as Mazur modules.
 The above construction determines a functor
 \begin{equation}\label{eq:functor_to_Mazur_modules}
   {\bf \Phi}_{\mathrm{Maz}}:   \cD_{qc,[0, \infty]}(W(k)^{syn})\to \mathrm{Mod}_{\mathrm{Maz}}(W(k)), \quad \cF \mapsto (\mathfrak{p}^*_{\bar{\dR}}(\cF), \varphi_{\bullet}).
 \end{equation}
 Next, we shall construct a functor 
 \begin{equation}\label{eq:FL_embeds_to_MM}
     \mathscr{DMF}^{\mathrm{big}}_{[0, \infty]}(W(k)) \to \mathrm{Mod}_{\text{Maz}}(W(k)),
 \end{equation}
 whose precomposition with \eqref{functor} is isomorphic to
\eqref{eq:functor_to_Mazur_modules}. To do this,
 observe that the $\infty$-category $\cC_1$ formed of pairs $(F^{\bullet}, \varphi ')$, where $F^{\bullet} \in \widehat {\cD}_{gr}^{\eff}(W(k)[v_-])$ and $\varphi': F^{\bullet} \to F_* F^0 \otimes_{W(k)} A$ is a map in $\widehat {\cD}_{gr}^{\eff}(W(k)[v_-])$, is equivalent to the lax equalizer $$\cC_2 := \mathrm{LEq}\begin{tikzcd}\Big(\cD_{qc, [0, \infty]}(\bA^1_-/\bG_m)\arrow[shift left]{r}{i_1^*}\arrow[shift right]{r}[swap]{F^*\circ i_p^*}&\cD_{qc}(\Spf W(k))\Big)\end{tikzcd}.$$ The latter category consists of $\cM\in \cD_{qc,[0, \infty]}(\bA^1_{-}/\bG_m)$ together with a morphism $\varphi': F^*(i^*_{p}\cM)\to i^*_{1}\cM$
 in $\cD_{qc}(\Spf W(k))$ ({\it cf.} Definition \ref{Defn-big-DMF-category}). Note that, for every $\cM\in \cD_{qc,[0, \infty]}(\bA^1_{-}/\bG_m)$, we have
 $$ \Hom (F^* i^*_{p}\cM, i^*_{1}\cM) \iso  \Hom (\cM, i_{p\, *}F_* i^*_{1}\cM) \iso \Hom (\cM, w_{\geq 0} i_{p\, *}F_* i^*_{1}\cM),$$
 where $w_{\geq 0}$ is the right adjoint to the embedding   $ \cD_{qc,[0, \infty]}(\bA^1_{-}/\bG_m)\mono  \cD_{qc}(\bA^1_{-}/\bG_m)$ ({\it cf.} Lemma \ref{lm:existenceofrightadjoint}). The right-hand side of the above is naturally isomorphic to $\Hom_{\widehat {\cD}_{gr}^{\eff}(W(k)[v_-])}(F^{\bullet}, F_* F^0 \otimes_{W(k)} A)$, where $F^{\bullet}$ is the graded module corresponding to $\cM$\footnote{Indeed, for every $\mathcal{V} \in \cD_{qc}(\Spf W(k))$, the sheaf $w_{\geq 0} i_{p\, *} \mathcal{V}$ corresponds, under the equivalence $\cD_{qc,[0, \infty]}(\bA^1_{-}/\bG_m) \iso \widehat {\cD}_{gr}^{\eff}(W(k)[v_-])$, to $\mathcal{V} \otimes A$.}. This gives 
 $\cC_1 \iso \cC_2$.

 The category $\cC_2$ receives an obvious functor $(\cM, \varphi')\mapsto (\cM, \varphi')$ from  $\mathscr{DMF}^{\mathrm{big}}_{[0, \infty]}(W(k))$, see Definition \ref{Defn-big-DMF-category}. This functor is fully faithful by \cite[Proposition II.1.5]{MR3904731}. To complete the construction of \eqref{eq:FL_embeds_to_MM}, note that the morphism $A \to B$ induces a functor $\cC_1 \to \mathrm{Mod}_{\text{Maz}}(W(k))$ sending $(F^\bullet, \varphi')$ to $(F^\bullet, \varphi)$, where $\varphi$ is the composition $ F^{\bullet} \xrightarrow{\varphi'} F_* F^0 \otimes A \to F_* F^0 \otimes B.$ Thus, in the notation of \S \ref{subsection:Fontaine-Laffaille-modules}, $\varphi _i = p^{i-[i]}\varphi_i'$. 
 
 Note that the restriction 
 of \eqref{eq:FL_embeds_to_MM} to 
  $\mathscr{DMF}^{\mathrm{big}}_{[0, p-1]}(W(k)) \to \mathrm{Mod}_{\text{Maz}}(W(k))$ is fully faithful. Indeed, using \cite[Proposition II.1.5]{MR3904731}, it is enough to show that, for $M^{\bullet}, N^{\bullet} \in {\cD_{qc, [0, p-1]}}(\bA^1_{-}/\bG_m)$, the natural map $\Hom_{\widehat {\cD}_{gr}^{\eff}(W(k)[v_-])}(M^{\bullet}, F_*N_0 \otimes A) \to \Hom_{\widehat {\cD}_{gr}^{\eff}(W(k)[v_-])}(N^{\bullet}, F_*N_0 \otimes B)$, is an equivalence. This follows since $\cD_{qc, (-\infty, p-1]}(\bA^1_{-}/\bG_m) \mono \cD_{qc}(\bA^1_{-}/\bG_m)$ admits a right adjoint $w_{\le p-1}$ ({\it cf.} Lemma \ref{lm:existenceofrightadjoint}) and $F_*N^0 \otimes A \to F_*N^0 \otimes B$ is an equivalence after applying $w_{\le p-1}$, see Corollary \ref{cor:AtoB}.  
\end{rem}
\begin{rem}[Relation to the Mazur-Ogus construction]\label{remark:Phi-Maz-tilde-isom-PhiMaz-Hsyn}
In the geometric context, \eqref{eq:mazur} encodes divisibility properties of the crystalline Frobenius discovered by Mazur. 
Namely, following the idea of Mazur \cite{Mazur-1973}, we define a functor 
\begin{equation}\label{eq:sm_var_maz_mod}
    {\bf \widetilde \Phi}_{\mathrm{Maz}}:   \widehat{\mathrm{Sm}}_{W(k)}^{\op}\to \mathrm{Mod}_{\mathrm{Maz}}(W(k))
   \end{equation} 
   as follows.  Let $X$ be a smooth $p$-adic formal scheme over $W(k)$. First, assume that there exists a closed 
   embedding $X\mono Y$, where $Y$ is a smooth $p$-adic formal scheme over $W(k)$ equipped with a Frobenius lift $\tilde F_Y$ (that is  not required to preserve  $X$). 
    Denote by $\cD_X(Y)$  the $p$-completed PD-hull of $X$ in $Y$ and by $J$ the sheaf of ideals of the closed embedding $X \mono \cD_X(Y)$.
     Note  $\cD_X(Y)$ can be also viewed as the $p$-completed PD-hull of $X_k$ in $Y$; in particular, it follows that 
      $\tilde F_Y$ extends uniquely to an endomorphism $\cD_X(Y)$. 
    Set 
    $\Omega_{\cD_X(Y)}^j$ to be the tensor product $\cO_{\cD_X(Y)} \otimes_{\cO_Y} \Omega_{Y}^j$. For each $i$, define a subcomplex $\Fil^i (\Omega_{\cD_X(Y)}^\bullet, d_{\dR})$ of $ (\Omega_{\cD_X(Y)}^\bullet, d_{\dR})$ to be
    $$J^{[i]} \to J^{[i-1]}\Omega_{\cD_X(Y)}^1 \to \cdots \to \Omega_{\cD_X(Y)}^i \to  \Omega_{\cD_X(Y)}^{i+1} \to \cdots ,$$
    where $J^{[m]}$ stands for the $m$-th divided power of $J$. The restriction map $(\Omega_{\cD_X(Y)}^\bullet, d_{\dR}) \to
   (\Omega_X^\bullet, d_{\dR})$ extends to a map of filtered complexes 
   \begin{equation}\label{eq:filt.Poincare}
     \Fil^\cdot (\Omega_{\cD_X(Y)}^\bullet, d_{\dR}) \to
   (\Omega_X^\bullet, d_{\dR})^{\geq \cdot}  
   \end{equation}
    and the filtered PD Poincar\'e Lemma of Berthelot  (\cite[p.~61]{Mazur-1973}) asserts this is an isomorphism in the filtered derived category. A key observation (\cite[p.~63]{Mazur-1973}) is that, for every $i$, 
   the chain map   
   $(\Omega_{\cD_X(Y)}^\bullet, d_{\dR})\rar{\tilde F_Y^*}  (\Omega_{\cD_X(Y)}^\bullet, d_{\dR})$
 restricted to $\Fil^i (\Omega_{\cD_X(Y)}^\bullet, d_{\dR})$
 is (uniquely)  divisible by $p^{[i]}$. 
 Combining this with \eqref{eq:filt.Poincare},  we obtain $W(k)$-linear maps $\varphi_i: F^i R\Gamma_{\dR}(X) \to  F_*  R\Gamma_{\dR}(X)$
 together with homotopies  $p^{[i+1]-[i]}\varphi_{i+1}\simeq \varphi_i  v_-$, where $v_-:  F^{i+1} R\Gamma_{\dR}(X) \to F^i R\Gamma_{\dR}(X)$
 denotes the usual ``inclusion'' map. Note that, by construction, the map $\varphi_{0}$ is the crystalline Frobenius.
 
The above construction determines a contravariant functor from the category $\cC$ of tuples $(s: X \mono Y, \tilde F_Y)$ to $\mathrm{Mod}_{\mathrm{Maz}}(W(k))$ that carries
every morphism in $\cC$, which is an isomorphism on $X$'s, to an isomorphism in $\mathrm{Mod}_{\mathrm{Maz}}(W(k))$.~A standard argument shows that $\cC^{\op} \to \mathrm{Mod}_{\mathrm{Maz}}(W(k))$ descends uniquely to 
\eqref{eq:sm_var_maz_mod}\footnote{Here is the argument: the fiber of $\cC^{\op} \to  \widehat{\mathrm{Sm}}_{W(k)}^{\op}$ over each affine $X$ is non-empty and admits finite products. Therefore, the nerve of the fiber is contractible. This proves the independence of the Mazur module structure on 
$R\Gamma_{\dR}(X)$ on the choice of $(s: X \mono Y, \tilde F_Y)$.}
({\it cf.}~\cite[\S 8, p.~23]{Berthelot-Ogus}). 

For the sake of completeness, let us remark that, for every $n<p$, the maps $\varphi_{i}$, 
($0\leq i \leq n$), can be assembled to a morphism in the derived category of sheaves
$$\frac{\varphi}{p^n}: (p^n\cO_X \rar{d_{\dR}} p^{n-1} \Omega_X^1  \rar{d_{\dR}} \cdots \rar{d_{\dR}} \Omega_X^n  \rar{d_{\dR}} \cdots) \rar{}  (\Omega_X^\bullet, d_{\dR}),$$
such that $p^n \frac{\varphi}{p^n}$ is the crystalline Frobenius precomposed with the embedding of the source into $(\Omega_X^\bullet, d_{\dR})$.
In particular, we obtain morphism \eqref{intro-Mazur-isom}.
The above construction of \eqref{intro-Mazur-isom} is due to Ogus;
we refer the reader to \cite[Theorem 6.8]{Ogus-log-dR-Witt} for 
a detailed exposition and a proof of the isomorphism property of \eqref{intro-Mazur-isom}. 

We claim that 
\begin{equation}\label{uniqueness_of_functor}
     {\bf \widetilde \Phi}_{\mathrm{Maz}} \iso  {\bf \Phi}_{\mathrm{Maz}} \circ \cH_{syn}.
\end{equation}
In fact, any functor $ \widehat{\mathrm{Sm}}_{W(k)}^{\op}\to \mathrm{Mod}_{\mathrm{Maz}}(W(k))$ equipped with an isomorphism 
between its post-composition with the forgetful functor $\mathrm{Mod}_{\mathrm{Maz}}(W(k)) \to \widehat {\cD}_{gr}(W(k)[v_-])$  and the 
 Hodge filtered de Rham functor such that $\varphi_0$ is homotopic to the crystalline Frobenius is isomorphic to 
 ${\bf \Phi}_{\mathrm{Maz}} \circ \cH_{syn}$.
To see this, consider the functor $F^i: \widehat{\mathrm{Sm}}_{W(k)}^{\op} \to \hat \cD(W(k))$
sending a smooth $p$-adic formal scheme $X$ to the $i$-th term of the Hodge filtration on $\mathrm{R}\Gamma_{\dR}(X)$. It suffices to show that, for every
$i$, the space $\operatorname{Map}(F^i, F^0)$ of morphisms between the functors is discrete,  $\pi_0 \operatorname{Map}(F^i, F^0)$
is $p$-torsionfree,  and the map $\pi_0 \operatorname{Map}(F^i, F^0)\to \pi_0 \operatorname{Map}(F^{i+1}, F^0)$ induced by the canonical morphism $v_-: F^{i+1} \to F^i$ is injective. Using that $F^i$ is a sheaf for the Zariski topology, the space  $\operatorname{Map}(F^i, F^0)$
does not change if $\widehat{\mathrm{Sm}}_{W(k)}$ is replaced by the subcategory $\widehat{\mathrm{SmAff}}_{W(k)}$ of smooth affine formal 
schemes. Let $\operatorname{QSyn}_{W(k)}$ be the category of quasi-syntomic affine formal schemes over $W(k)$, and let $LF^i: \operatorname{QSyn}_{W(k)}^{\op} \to  \hat \cD(W(k))$  be the $p$-complete left Kan extension along the embedding $\mathrm{SmAlg}_{W(k)}^{\op} \hookrightarrow \operatorname{QSyn}_{W(k)}^{\op}$. Then by the universal property of the 
 left Kan extension, we have a homotopy equivalence
$\operatorname{Map}(F^i, F^0)\xrightarrow{\sim} \operatorname{Map}(LF^i, LF^0)$ \cite[Proposition 4.3.2.17.]{HHT}. Our claim follows from the following facts:
 $LF^i$ is a sheaf for the quasi-syntomic topology on $\operatorname{QSyn}_{W(k)}$ (\cite[Theorem 3.1]{bhatt2019topological}); the inclusion 
 $\operatorname{QRSPerf}_{W(k)} \mono \operatorname{QSyn}_{W(k)}$ of the category of quasi-regular semi-perfectoid affine formal schemes into the category of all quasi-syntomic affine formal schemes induces an equivalence between the corresponding categories of sheaves (\cite[Proposition 4.31]{bhatt2019topological}); for any $X \in \operatorname{QRSPerf}_{W(k)}$, the complex $F^i(X)$ is concentrated in degree zero, $H^0(F^i(X))$ is  $p$-torsionfree and the map $v_-: H^0(F^{i+1}(X)) \to H^0(F^i(X))$ is injective (\cite[Construction 6.6]{Antieau-Mathew-Morrow-Nikolaus}).

Using \eqref{eq:FL_embeds_to_MM}, the category $\mathscr{DMF}^{\mathrm{big}}_{[0,p-1]}(W(k))$ is a full subcategory of $\mathrm{Mod}_{\mathrm{Maz}}(W(k))$. In particular, 
 we derive from \eqref{uniqueness_of_functor} that the 
 functors $$\widetilde{\Phi}_{\mathrm{Maz}}, \Phi_{\mathrm{Maz}} \circ \cH_{syn}: \widehat{\mathrm{Sm}}^{\le p-1}_{W(k)} \to \mathscr{DMF}^{\mathrm{big}}_{[0,p-1]}(W(k))$$ are isomorphic.
 \end{rem}
\begin{rem}[A ``non-abelian'' refinement of the crystalline Frobenius]\label{remark:refined-crystalline-Frob-categories}
Let $X$ be a $p$-adic formal scheme over $W(k)$. Let $X_k:=X\times_{\Spf W(k)} \Spec k$ be its special fibre. 
By Theorem \ref{Drinfeld-prop}, we obtain a canonical isomorphism 
\begin{equation}\label{remark:after-Drinfeld-prop-can-morphism}
X_k^{\cN}\times_{k^{\cN}}\Spf B/\bG_m\simeq X^{\cN}\times_{W(k)^{\cN}}\Spf B/\bG_m\simeq X^{\dR,+}\times_{\bA^1_-/\bG_m}\Spf B/\bG_m
\end{equation}
of stacks over $\Spf B/\bG_m$. Recall from \cite[Remark 3.3.4]{bhatt-lecture-notes}  a canonical morphism of stacks over $k^{\cN}$:
\begin{equation}\label{remark:after-Drinfeld-prop-can-morphism-2}
X_k^{\mathrm{cris}}\times k^{\cN}=X^{\dR}\times k^{\cN}\to X_k^{\cN},
\end{equation}
which is an isomorphism over the open substack $\Spf W(k)\overset{j_+}{\hooklongrightarrow}k^{\cN}$ given by $v_+\neq 0$. Likewise, 
we have a map of stacks over $\bA^1_-/\bG_m$
\begin{equation}\label{remark:after-Drinfeld-prop-can-morphism-3}
X^{\dR,+}\to X^{\dR}\times \bA^1_-/\bG_m,
\end{equation}
which is an isomorphism for $v_-\neq 0$. The map can be constructed by restricting the structure map $X^{\cN} \to X^\Prism \times W(k)^\cN$ (see \cite[Definition 5.3.10 (3)]{bhatt-lecture-notes})
to $\bA^1_{-}/\bG_m \rar{\mathfrak{p}_{\bar{\dR}}} W(k)^\cN$.
Composing \eqref{remark:after-Drinfeld-prop-can-morphism} with \eqref{remark:after-Drinfeld-prop-can-morphism-2}, we obtain  
\begin{equation}\label{eqn:refinement-of-crystalline-Frobenius}
F_{\mathrm{ref}}: X^{\dR}\times\Spf B/\bG_m\to X^{\dR,+}\times_{\bA^1_-/\bG_m}\Spf B/\bG_m,
\end{equation}
whose postcomposition with the map \eqref{remark:after-Drinfeld-prop-can-morphism-3}, pulled back to $\Spf B/\bG_m$, is the crystalline Frobenius. 
\end{rem}
\begin{rem}[The category of Fontaine-Laffaille modules over a scheme]\label{The_category_of_Fontaine_Laffaille_modules_over_a_scheme}
Let $X$ be a smooth $p$-adic formal scheme over $W(k)$. We shall explain how the refined Frobenius morphism $F_{\mathrm{ref}}$ in \eqref{eqn:refinement-of-crystalline-Frobenius} can be used to define a certain stable $\infty$-category $\mathscr{DMF}^{\mathrm{big}}_{[0,p-1]}(X)$. By definition, an object of $\mathscr{DMF}^{\mathrm{big}}_{[0,p-1]}(X)$ is a pair $(\cM,\varphi)$, where $\cM\in\cD_{qc,[0,p-1]}(X^{\dR,+})$ and $\varphi:(F_{\mathrm{ref}}^*c^*\cM_{})^{p-1}\xrightarrow{\sim}\cM_{v_-=1}$ is an isomorphism in $\cD_{qc}(X^{\dR})$. Here $c: X^{\dR,+}\times_{\bA^1_-/\bG_m}\Spf B/\bG_m\to X^{\dR,+}$ is the projection to the first component, and $F_{\mathrm{ref}}^*$ is the pullback along \eqref{eqn:refinement-of-crystalline-Frobenius}. The superscript $(\cdot)^i:\cM\mapsto \cM^i$ refers to the functor $\cD_{qc}(X^{\dR}\times \Spf B/\bG_m)\to \cD_{qc}(X^{\dR})$ given by $\cM\mapsto \mathrm{pr}_{X^{\dR}*}(\cM\otimes\cO(i))$, where $\mathrm{pr}_{X^{\dR}}: X^{\dR}\times \Spf B/\bG_m\to X^{\dR}$.     
\end{rem}

\section{The equivalence $\cD_{qc,[0,p-2]}(W(k)^{syn}) \to \mathscr{DMF}^{\mathrm{big}}_{[0,p-2]}(W(k))$} 
\subsection{The main result}
\begin{Th}\label{main-thm}
    The functor \eqref{functor} induces an equivalence of categories 
    \begin{equation}\label{Phi-equiv}
     \Phi_{\mathrm{Maz}}\colon \cD_{qc,[0,p-2]}(W(k)^{syn}) \xrightarrow{\sim} \mathscr{DMF}^{\mathrm{big}}_{[0,p-2]}(W(k)).\end{equation}
    Moreover, it restricts to an equivalence on perfect complexes 
    \begin{equation}\label{Phi-equiv-perfect}
    \Perf_{[0,p-2]}(W(k)^{syn}) \xrightarrow{\sim} \cD^b\left(\mathscr{MF}_{[0,p-2]}(W(k))\right)\end{equation}
    that identifies the subcategories\footnote{By definition, $\Coh_{[0,p-2]}(W(k)^{syn})$ is a full subcategory of $\Coh(W(k)^{syn})$ that consists of objects lying in $\cD_{qc,[0,p-2]}(W(k)^{syn})$.} $\Coh_{[0,p-2]}(W(k)^{syn})\xrightarrow{\sim}\mathscr{MF}_{[0,p-2]}(W(k))$.
    \end{Th}

We now give the outline for the proof of Theorem \ref{main-thm}.  

\textit{Step 1:} we reduce \eqref{Phi-equiv} to proving the equivalence 
\begin{equation}\label{Phi-equiv_mod_p}
 \Phi_{\mathrm{Maz}} \otimes \bF_p:\cD_{qc,[0,p-2]}(W(k)^{syn})\otimes\bF_p \xrightarrow{\sim}\mathscr{DMF}^{\mathrm{big}}_{[0,p-2]}(W(k))\otimes\bF_p.\end{equation} 
We say that $\cF\in\cD_{qc}(W(k)^{syn}\otimes\bF_p)$ has weights in $[a,b]$ if and only if the pullbacks of $\cF$ along both composites $\Spec k[v_{\pm}]/\bG_{m,k}\to k^{\cN}\otimes\bF_p\to W(k)^{syn}\otimes\bF_p$ lie in $\cD_{qc,[a,b]}((\bA^1_{\pm}/\bG_m)\otimes \bF_p)$ in the sense of Definition \ref{Defn-Dqc-ab-A1modGm-plusminus}. 

We show that the restriction along $W(k)^{syn}\otimes\bF_p\to W(k)^{syn}$ induces an equivalence 
\begin{equation}\label{step-1-equiv-modp}
\cD_{qc,[0,p-2]}(W(k)^{syn})\otimes\bF_p\simeq \cD_{qc,[0,p-2]}\left(W(k)^{syn}\otimes\bF_p\right).\end{equation}

\textit{Step 2:} Let $W(k)^{syn}_{red}$ be the reduced locus of $W(k)^{syn}\otimes\bF_p$. Let $\cD_{qc,[a,b]}(W(k)^{syn}_{red})$ be the full subcategory of $\cD_{qc}(W(k)^{syn}_{red})$ consising of objects with Hodge-Tate weights in $[a,b]$ defined via the restriction along
$(\bA^1_{\pm}/\bG_m)\otimes \bF_p\to W(k)^{syn}_{red}$ as in Step 1. 
We show in Proposition \ref{prop:step-2} that the restriction functor 
\begin{equation}\label{restriction-functor-r}
r_{[0,p-2]}:\cD_{qc,[0,p-2]}(W(k)^{syn}\otimes\bF_p)\to\cD_{qc,[0,p-2]}(W(k)^{syn}_{red})    
\end{equation}
is an equivalence of categories.

\textit{Step 3:} Use description of $W(k)^{syn}_{red}$ from \cite[\S 7]{drinfeld2022prismatization} to complete the proof.

\begin{rem}
    We show in Step 3 that the functor $ \Phi_{\mathrm{Maz}}\otimes\bF_p$ factors through the reduced locus 
  \[\begin{tikzcd}
      \cD_{qc,[0,p-1]}(W(k)^{syn}\otimes\bF_p)\arrow[]{rr}{ \Phi_{\mathrm{Maz}}\otimes\bF_p}\arrow[]{rd}[swap]{r_{[0, p-1]}}&&\mathscr{DMF}^{\mathrm{big}}_{[0,p-1]}(W(k))\otimes\bF_p\\
      &\cD_{qc,[0,p-1]}(W(k)^{syn}_{red})\arrow[]{ru}[swap]{\Phi^{\mathrm{Maz}}_{red}}&
  \end{tikzcd} \] 
 Moreover, we construct an equivalence 
 \begin{equation}\label{reduced-locus-equiv}
 \Phi_{\mathrm{DI}}: \cD_{qc,[0,p-1]}(W(k)^{syn}_{red})\simeq\mathscr{D}\mathscr{MF}^{\mathrm{big}}_{[0,p-1]}(W(k))\otimes\bF_p.
 \end{equation}
Thus, via \eqref{reduced-locus-equiv}, $ \Phi^{\mathrm{Maz}}_{red}$ can be viewed as an endo-functor on $\mathscr{DMF}^{\mathrm{big}}_{[0,p-1]}(W(k))\otimes\bF_p$.    
We shall see that $\Phi^{\mathrm{Maz}}_{red}$ is not an equivalence of categories, but becomes an equivalence when restricted to $\mathscr{DMF}^{\mathrm{big}}_{[0,p-2]}(W(k))\otimes\bF_p$. 
\end{rem}

\subsection{Proof of the Main Theorem}
We now prove Theorem \ref{main-thm}.
\subsubsection{Step 1:}
For a $\bZ$-linear presentable stable $\infty$-category $\cC$ and a commutative ring $A$ we will denote by $\cC \otimes A$ the category $\cC \otimes_{\cD(\bZ)} \cD(A)$ defined via Lurie's tensor product in $\Pr^L_{st}$; see \cite[\S 4.8]{Lurie-Higher-Algebra}. 
Informally, an object of $\cC \otimes_{\cD(\bZ)} \cD(A)$ is merely an $A$-module in $\cC$.
We say that such $\cC$ is $p$-complete if the natural functor $\cC \to \lim \cC \otimes \bZ/p^n$ is an equivalence.
\begin{lm}
    Let $f:\cC_1 \rightarrow \cC_2$ be a map between stable $\bZ$-linear presentable $p$-complete $\infty$-categories. If the induced functor $f\otimes\bF_p:\cC_1\otimes \bF_p \to \cC_2\otimes\bF_p$ is an equivalence, then $f$ is an equivalence.
\end{lm}
\begin{proof}
Observe that, for any $X,Y\in \cC_1$, we have
\begin{equation}\label{pf:mod_p^n}
  \Hom_{\cC_1\otimes \bZ/p^n} (X\otimes \bZ/p^n, Y\otimes \bZ/p^n)\iso 
  \Hom_{\cC_1} (X, Y\otimes \bZ/p^n) \iso \Hom_{\cC_1} (X, Y)\otimes \bZ/p^n.
\end{equation}
Indeed, each mapping spectrum in \eqref{pf:mod_p^n} is naturally identified with the cone of $\Hom_{\cC_1} (X, Y)\rar{p^n}\Hom_{\cC_1} (X, Y)$. 
In particular, $ \Hom_{\cC_1} (X, Y) \iso \lim \Hom_{\cC_1} (X, Y) \otimes \bZ/p^n$ {\it i.e.}, $\Hom_{\cC_1} (X, Y)$ is $p$-complete. The same applies to
$\Hom_{\cC_2} (X', Y')$, for $X', Y' \in \cC_2$.

Let us check that $f$ is fully faithful {\it i.e.}, the map $\Hom_{\cC_1} (X, Y) \to \Hom_{\cC_2} (f(X), f(Y))$ is an equivalence. 
By $p$-completeness, it suffices to check that $\Hom_{\cC_1} (X, Y)\otimes \bF_p \to \Hom_{\cC_2} (f(X), f(Y))\otimes \bF_p$ is an equivalence. Using \eqref{pf:mod_p^n}, this is equivalent to showing that $ \Hom_{\cC_1\otimes \bF_p} (X\otimes \bF_p, Y\otimes \bF_p) \iso \Hom_{\cC_2\otimes \bF_p} (f(X)\otimes \bF_p, f(Y)\otimes \bF_p)$. Consequently, it suffices to show that, for any $X$ in $\cC_1$ the natural map $f(X) \otimes \bF_p \to f_1(X \otimes \bF_p)$,
where $f_n: =f\otimes\bZ/p^n: \cC_1 \otimes \bZ/p^n \to \cC_2\otimes\bZ/p^n $, is an equivalence in $\cC_2 \otimes \bF_p$. To this end, observe that the functor 
$\cC_2 \otimes \bF_p \to \cC_2$ is conservative. Indeed, $\cC_2 \otimes \bF_p$ is identified with the category $\Fun_{\operatorname{Mod}_{\bZ}(\Pr^L_{st})}(\cD(\bF_p), \cC_2) \in \Pr^L_{st}$ of colimit-preserving $\bZ$-linear functors. 
This equivalence carries the functor $\cC_2 \otimes \bF_p \to \cC_2$ to the evaluation-at-$\bF_p$ functor
$\Fun_{\operatorname{Mod}_{\bZ}(\Pr^L_{st})}(\cD(\bF_p), \cC_2)\to \cC_2$. It remains to observe that $\cD(\bF_p)$ is generated by
$\bF_p$ under colimits and finite limits.  By conservativity, it is enough to show that $f(X) \otimes \bF_p \to f_1(X \otimes \bF_p)$ is an isomorphism in $\cC_2$. But this is clear: each side is naturally identified with the cone of $f(X)\rar{p} f(X)$. 

We claim that $f_n$ is fully faithful,  for every $n$: representing $\cC_i \otimes \bZ/p^n$ as the inner $\Hom$ from $\cD(\bZ/p^n)$ to $\cC_i$ the claim follows
from a general fact that a fully faithful $f$ induces a  fully faithful functor on inner $\Hom$'s, see \cite[TAG 04Q6]{kerodon}.

It remains to check that $f_n$ is essentially surjective. We check it by induction on $n$. 
For any $Z\in \cC_2 \otimes \bZ/p^n$, we have a fiber sequence $Z \otimes_{\bZ/p^n} \bZ/p^{n-1} \to Z \to Z
 \otimes_{\bZ/p^n} \bF_p$. By induction assumption the boundary terms are in the essential image. Since $f_n$ is fully faithful it follows that $Z$ is in the  essential image.
\end{proof}
In particular, the lemma reduces  proving equivalence \eqref{Phi-equiv} to \eqref{Phi-equiv_mod_p}. 

\begin{pr}\label{pr:res-Dqc-syn-special-fib}
The restriction along $W(k)^{syn}\otimes\bF_p\to W(k)^{syn}$ induces an equivalence $\cD_{qc} (W(k)^{syn} \otimes \bF_p) \simeq \cD_{qc} (W(k)^{syn}) \otimes \bF_p$, where the right-hand side stands for the category of $\bF_p$-linear objects in $\cD_{qc}(W(k)^{syn})$. Moreover, this equivalence respects Hodge-Tate weights. Thus, it induces $\cD_{qc, [0, p-2]} (W(k)^{syn} \otimes \bF_p) \simeq \cD_{qc, [0, p-2]} (W(k)^{syn}) \otimes \bF_p$. 
\end{pr}
The proof of this proposition is given at the end of \textit{Step 1}. We start with some preliminary results.

\begin{lm}\label{lm:lemma_on_rings}
    Let $A$ be a ring, $(p, f_1, \cdots, f_n)$ is a regular sequence in $A$, $I\subset A$ the ideal generated by this sequence, and let $\bar{I}\subset A/p$ be its image. Assume that $A$ is $I$-complete. Denote by $\Spf A$ (resp. $\Spf A/p$)  the $I$-adic (resp. $\bar{I}$-adic) formal scheme. 
    Then the restriction along $i: \Spf(A) \otimes \bF_p \to \Spf(A)$ induces $\cD_{qc}(\Spf(A)) \otimes \bF_p \simeq \cD_{qc}(\Spf(A/p)).$
\end{lm}
\begin{proof}
    Identify $\cD_{qc}(\Spf(A))$ with $\cD_{I-comp}(A) \subset \cD(A)$ and $\cD_{qc}(\Spf(A/p))$ with $\cD_{\bar{I}-comp}(A/p) \subset \cD(A/p)$. Under this equivalence, the functor $\Phi: \cD_{qc}(\Spf(A)) \otimes \bF_p \to \cD_{qc}(\Spf(A/p))$ is compatible with the functor $\cD(A) \otimes \bF_p \to \cD(A/p)$ induced by $\Spec(A/p) \to \Spec(A),$ which is an equivalence by \cite[Theorems 4.8.5.11 and 4.8.5.16]{Lurie-Higher-Algebra}. 
    Since $- \otimes \cD(\bF_p) \simeq \Fun_{\operatorname{Mod}_{\bZ}(\Pr^L_{st})}(\cD(\bF_p), -)$ as functors, and the latter preserves fully faithful embeddings, we see that $\Phi$ is fully faithful. Let us show that it is essentially surjective. Observe that $A/p$ is in the image since $\Phi(A)=A/p.$ Thus, it suffices to check that  $A/p$ generates $\cD_{\bar{I}-comp}(A/p).$ Note that the $\bar{I}$-completion functor $\cD(A/p) \to \cD_{\bar{I}-comp}(A/p)$ is essentially surjective and it commutes with colimits since it is left adjoint to the embedding. Since $A/p$  generates $\cD(A/p)$ the claim follows.
    \end{proof}

\begin{lm}\label{lm:lemma_on_stacks}
    Let $\fX$ be a stack which admits a faithfully flat cover by an affine formal scheme $\Spf(A)$. Denote by $\Spf(A)^{\bullet/\fX}$ the Cech nerve of $\Spf(A) \to \fX$ and assume that $\Spf(A)^{\bullet/\fX}$ is represented by an $I^\bullet$-adic affine simplicial formal scheme $\Spf(A^{\bullet})$
    with each $I^\bullet \subset A^\bullet$ satisfying the assumption of Lemma \ref{lm:lemma_on_rings}.
    Then the natural map $\cD_{qc} (\fX) \otimes \bF_p \to \cD_{qc}(\fX \otimes \bF_p)$ is an equivalence.
\end{lm}

\begin{proof}\label{Dqc-tensor-Fp-for-stacks}
    The assumption implies that the functor $\cD_{qc} (\fX) \to \lim \cD_{qc} (\Spf(A^{\bullet}))$ is an equivalence. The natural map $\cX \otimes \bF_p \to \cX \otimes^L \bF_p$ is an isomorphism since $A$ is flat over $\bZ_p$ and flatness is fpqc-local. Thus, by the base change stability of Cech nerves we get the following diagram 
\[\begin{tikzcd}
	\cD_{qc}(\cX) \otimes \bF_p &&& (\lim \cD_{qc} (\Spf A^\bullet)) \otimes \bF_p \\
	\\
	\cD_{qc}(\cX \otimes \bF_p) &&& \lim \cD_{qc} (\Spf(A^\bullet \otimes_{\bZ_p} \bF_p))
	\arrow[from=1-1, to=1-4]
	\arrow[from=1-1, to=3-1]
	\arrow[from=1-4, to=3-4]
	\arrow[from=3-1, to=3-4].
\end{tikzcd}\]
Note that horizontal maps are equivalences by descent. The right vertical map is an equivalence by combining Lemma \ref{lm:lemma_on_rings} with (self)dualizability of $\cD(\bF_p)$ in $\Pr^L_{st}$ (the latter implies that, for any diagram $\cC_i$ in $\Pr^L_{st}$, the natural map $(\lim \cC_i) \otimes \bF_p \to \lim \cC_i \otimes \bF_p$ is an equivalence since, for every $\cC \in \Pr^L_{st}$, we have $\cC \otimes \bF_p \iso \Fun(\cD(\bF_p), \cC)$). 
   Thus, all maps in the diagram are equivalences. 
\end{proof}

\textit{Proof of Proposition \ref{pr:res-Dqc-syn-special-fib}.}
    Denote by $C$ the $p$-adic completion of an algebraic closure of $\Frac W(k)$ and by $\cO_C\subset C$ the subring of integral elements. Then $W(k) \to \cO_C$ induces a faithfully flat cover $\cO_C^{\cN} \to W(k)^{\cN}$ by \cite[Remark 6.6.12]{bhatt-lecture-notes}. If $\cO_C^{\cN} \to \bA^1_-/\bG_m$ is the Rees map, then  $\cO_C^{\cN} \times_{\bA^1_-/\bG_m} \bA^1_-$ is a formal affine scheme $\Spf A$
    such that the Cech nerve  $\Spf(A)^{\bullet/W(k)^{\cN} }$ satisfies the assumption of Lemma \ref{lm:lemma_on_stacks}; see \cite[Corollary 5.5.11]{bhatt-lecture-notes}.
    Thus, Lemma  \ref{lm:lemma_on_stacks} gives $\cD_{qc}((W(k)^{\cN}) \otimes \bF_p \simeq \cD_{qc} (W(k)^{\cN} \otimes \bF_p)$. Similarly, 
    $\cD_{qc}((W(k)^{\Prism}) \otimes \bF_p \simeq \cD_{qc} (W(k)^{\Prism} \otimes \bF_p)$.
    This implies the first statement since $\cD(W(k)^{syn}) \simeq \operatorname{Eq} (\cD(W(k)^{\cN}) \rightrightarrows \cD(W(k)^{\Prism}).$ To show that it preserves Hodge-Tate weights it is enough to observe that the equivalence $\cD(\bA^1_{-}/\bG_m) \otimes \bF_p \to \cD(\bA^1_{-}/\bG_m \otimes \bF_p)$ restricts to an equivalence $\cD_{qc, [a,b]} (\bA^1_{-}/\bG_m) \otimes \bF_p \to \cD_{qc, [a, b]} (\bA^1_{-}/\bG_m \otimes \bF_p)$ which is clear.

\subsubsection{Step 2:} Recall from \cite[\S 5.10.8 and Prop.~8.11.2]{drinfeld2022prismatization} that $W(k)^{syn}_{red}$ is a Cartier divisor of $W(k)^{syn}\otimes\bF_p$, given by a nonzero section $v_1\in H^0(W(k)^{syn}\otimes\bF_p,\cO\{p-1\})$. Let $(W(k)^{syn}\otimes\bF_p)_{v_1^i=0}$ denote the zero locus of $v_1^i$ inside $W(k)^{syn}\otimes\bF_p$.  
\begin{pr}\label{prop:step-2}
For each $i>1$, the restriction map
\begin{equation}\label{restriction-map-i-to-i-1}
    \cD_{qc,[0,p-2]}((W(k)^{syn}\otimes\bF_p)_{v_1^i=0})\to\cD_{qc,[0,p-2]}((W(k)^{syn}\otimes\bF_p)_{v_1^{i-1}=0})
\end{equation}
is an equivalence. In particular, the functor $r_{[0,p-2]}$ from \eqref{restriction-functor-r} is an equivalence. 
\end{pr}
\begin{proof}
Observe that the closed embedding 
\begin{equation}\label{closed-embedding-iota-i}
\alpha_i: (W(k)^{syn}\otimes\bF_p)_{v_1^{i-1}=0}\hookrightarrow (W(k)^{syn}\otimes\bF_p)_{v_1^i=0}
\end{equation}
is a square-zero extension with sheaf of ideals given by 
\begin{equation}\label{ideal-sheaf-iota-i}
    \mathcal{J}_i=\cO_{W(k)_{red}^{syn}}\{-(i-1)(p-1)\}. 
\end{equation}
We shall use the following generalization of a fact from \cite[Example 6.5.13]{bhatt-lecture-notes}. 
\begin{lm}\label{generalized-effectivity-vanishing-lemma}
Let $n$ be an arbitrary integer. For any $\cE\in\cD_{qc,[n,\infty]}(W(k)^{syn}_{red})$ and any $\cE'\in \cD_{qc,(-\infty,n-1]}(W(k)^{syn}_{red})$, we have $\mathrm{RHom}(\cE',\cE)=0$. 
\end{lm}
\begin{proof}
Recall from \cite[\S 7]{drinfeld2022prismatization} and \cite[\S 6.2]{bhatt-lecture-notes} the geometry of the reduced locus $W(k)^{\cN}_{red}$, which has two irreducible components given by the Hodge-Tate component $D_{\HT}$ and the de Rham component $D_{\dR}$ meeting transversely along the closed substack $D_{Hod}$.~The Hodge-Tate component $D_{\HT}$ is identified with $\bA^{1,\dR}_+/\bG_{m, k}\simeq \bG_{a,k}/(\bG_{a,k}^{\sharp}\rtimes\bG_{m,k})$.~The de Rham component $D_{\dR}$ is identified with the classifying stack of a commutative group scheme $\cG$ over $(\bA^1_-/\bG_m)\otimes \bF_p$.~The closed substack $D_{Hod}=B(F_*\bG_{a,k}^{\sharp}\rtimes \bG_{m,k})$ is embedded into $D_{\dR}$ as the fibre over $B\bG_{m,k}\subset(\bA^1_-/\bG_m)\otimes \bF_p$, and it is embedded into $D_{\HT}$ as $\alpha_p/(\bG_{a,k}^{\sharp}\rtimes\bG_{m,k})$. Consequently, $\cD_{qc}(W(k)^{\cN}_{red})$ is a full subcategory of the fibre product 
$\cD_{qc}(D_{\HT})\times_{\cD_{qc}(D_{Hod})}\cD_{qc}(D_{\dR})$.~We refer the reader to \cite[Theorem 16.2.0.2]{Lurie-sag} for a much more general result the above assertion follows from. 

Set $D_{und}:=B\bG_{m,k}^{\sharp}$.~The stack $W(k)^{\cN}_{red}$ contains two disjoint open substacks isomorphic to $D_{und}$.~The embedding of $D_{und}\subset D_{\HT}=\bG_{a,k}/(\bG_{a,k}^{\sharp}\rtimes \bG_{m,k})$ identifies $D_{und}$ with $(\bG_{a,k}\setminus\{0\})/(\bG_{a,k}^{\sharp}\rtimes \bG_{m,k})$.~The embedding of $D_{und}\subset D_{\dR}$ identifies $D_{und}$ with the complement to $D_{Hod}$ inside $D_{\dR}$.~The stack $W(k)^{syn}_{red}$ is obtained from $W(k)^{\cN}_{red}$ by identifying the two copies of $D_{und}$ inside $D_{\HT}$ and $D_{\dR}$ via the arithmetic Frobenius automorphism $\Id \times \mathrm{Frob}$ of $B\bG_{m,k}^{\sharp}= B\bG^{\sharp}_{m,\bF_p}\times\Spec k$.
 In particular, by this pushout presentation of $W(k)^{syn}_{red}$, we have: 
\begin{align}
\begin{split}
\mathrm{RHom}(\cE',\cE)&=\mathrm{Fib}\Big(\mathrm{RHom}(\cE'|_{D_{\HT}},\cE|_{D_{\HT}})\oplus \mathrm{RHom}(\cE'|_{D_{\dR}},\cE|_{D_{\dR}})\\
&\to \mathrm{RHom}(\cE'|_{D_{und}},\cE|_{D_{und}})\oplus \mathrm{RHom}(\cE'|_{D_{Hod}},\cE|_{D_{Hod}})\Big).
\end{split}
\end{align}
Now we claim that under our assumption on Hodge-Tate weights, the following hold:
\begin{enumerate}
    \item[(i)] $\mathrm{RHom}(\cE'|_{D_{\HT}},\cE|_{D_{\HT}})=0$.
    \item[(ii)] The restriction map induces an isomorphism $\mathrm{RHom}(\cE'|_{D_{\dR}},\cE|_{D_{\dR}})\xrightarrow{\sim}\mathrm{RHom}(\cE'|_{D_{und}},\cE|_{D_{und}})$.\footnote{The argument below will prove this under weaker assumptions on the weights of $\cE$, namely, it has weights $\geq n-1$ instead of $\geq n$.}
    \item[(iii)] $\mathrm{RHom}(\cE'|_{D_{Hod}},\cE|_{D_{Hod}})=0$.
\end{enumerate}
To prove (i), note $\cD_{qc}(D_{\HT})\simeq \cD_{qc}(\bA^{1,\dR}_+/\bG_{m,k})$ is a full subcategory of the derived category $\cD_{gr}(\mathcal{A}_1)$ of graded $\cA_1$-modules,~where $\mathcal{A}_1:=k\{v_+,D\}/(Dv_+-v_+D-1)$ is the algebra of differential operators on the affine line; see 
\cite[\S 6.5.4]{bhatt-lecture-notes}. Let $G'= \bigoplus  G'_i$ (resp.~$G=\bigoplus G_i$) be the graded module corresponding to $\cE'|_{D_{\HT}}$ (resp.~$\cE|_{D_{\HT}}$) under the above inclusion. Using the assumption on weights and Remark \ref{HT-weights-kN-remark}, we conclude that $G_i$ is acyclic for $i\leq n-1$ and $G_i'\xrightarrow{v_+}G_{i+1}'$ is an isomorphism for $i+1\geq n$. 

We claim that $\mathrm{RHom}_{\cD_{gr}(k[v_+])}(G',G(j))=0$, $j\leq 0$, where $G(j)_i:=G_{i+j}$. Indeed, let $G''$ be the free graded $k[v_+]$-module on $G'_{n-1}$, {\it i.e.}, $G''_i=G'_i$ for $i\geq n-1$ and $0$ otherwise.~Then $\mathrm{RHom}_{\cD_{gr}(k[v_+])}(G'',G(j))\xrightarrow{\sim} 
\mathrm{RHom}_{\cD_{gr}(k)}(G'_{n-1},G_{n-1+j})$, which is $0$ for $j\leq 0$ since $G_{n-1+j}=0$. 
It remains to check that for $G''':=\mathrm{cone}(G'' \to G')$, 
we have $\mathrm{RHom}_{\cD_{gr}(k[v_+])}(G''',G(j))=0$ for $j\leq 0$.~Indeed, $G'''_i$ is acyclic for $i>n-2$.~Then we have
$$\mathrm{RHom}_{\cD_{gr}(k[v_+])}(G''',G(j))= \mathrm{Fib}\Big(\mathrm{RHom}_{\cD_{gr}(k)}(G''',G(j)) \rar{\ad_{v_+}}
\mathrm{RHom}_{\cD_{gr}(k)}(G''',G(j+1))\Big),$$
\begin{equation}\label{product-formula-RHom-grk}
\mathrm{RHom}_{\cD_{gr}(k)}(G''',G(j)) =  \prod_i \mathrm{RHom}_k(G'''_i, G_{i+j})=0, \quad\text{for } j\leq 1,\end{equation}
since for every $i$, either $G'''_i$ or $G_{i+j}$ is acyclic (for $j\leq 1$). 
Next, we have
\begin{equation}\label{RHom-DA1-fibre}
\mathrm{RHom}_{\cD_{gr}(\cA_1)}(G',G)=\mathrm{Fib}(\mathrm{RHom}_{\cD_{gr}(k[v_+])}(G',G) \rar{\ad_D }\mathrm{RHom}_{\cD_{gr}(k[v_+])}(G',G(-1))).
\end{equation}
By the vanishing of $RHom_{\cD_{gr}(k[v_+])}(G',G(j))=0$ for $j=0, -1$, we deduce vanishing of the left-hand side in \eqref{RHom-DA1-fibre}.~Thus $RHom(\cE'|_{D_{\HT}},\cE|_{D_{\HT}})=0$ and (i) holds.

To prove (ii), recall from \cite[\S 6.5.3]{bhatt-lecture-notes} that $\cD_{qc}(D_{\dR})$ can be viewed as a full subcategory of $\cD_{gr}(k[v_-,\Theta])$, where $\deg(v_-)=-1$ and $\deg(\Theta)=-p$. Denote by $F'=\bigoplus {F'}^i$ (resp.~$F=\bigoplus F^i$) the module corresponding to $\cE'$ (resp.~$\cE$). Using the assumption on weights and Remark \ref{HT-weights-kN-remark}, we have that ${F'}^i$ is acyclic for $i\geq n$, and $F^i\xrightarrow{v_-}F^{i-1}$ is an isomorphism for $i\leq n$. 

Likewise, $\cD_{qc}(D_{und})$ is a full subcategory of $\cD_{gr}(k[v_-,v_-^{-1},\Theta])=\cD(k[\Theta])$. 
Consider the embedding $j:D_{und}\hookrightarrow D_{\dR}$ lying over $\overline{j}: (\bA_-^1\setminus\{0\})/\bG_m \otimes \bF_p\hookrightarrow \bA_-^1/\bG_m \otimes \bF_p$. Under the identifications above, the pullback along $j$ corresponds to the functor $j^*: \cD_{gr}(k[v_-,\Theta]) \to \cD_{gr}(k[v_-,v_-^{-1},\Theta])$ given by the tensor product with $k[v_-,v_-^{-1},\Theta]$ over $k[v_-,\Theta]$.
We shall prove that under the above assumption on $F$ and $F'$, the map 
\begin{equation}\label{j*-on-RHom-Dund-to-DdR}
j^*: \mathrm{RHom}_{\cD_{gr}(k[v_-,\Theta])}(F',F)\to \mathrm{RHom}_{\cD_{gr}(k[v_-,v_-^{-1},\Theta])}(j^*F',j^*F)
\end{equation}
is an isomorphism. The left-hand side of \eqref{j*-on-RHom-Dund-to-DdR} can be computed as 
\begin{equation*}
    \mathrm{RHom}_{\cD_{gr}(k[v_-,\Theta])}(F',F)=\mathrm{Fib}\Big(\mathrm{RHom}_{\cD_{gr}(k[v_-])}(F',F)\xrightarrow{\ad_{\Theta}}\mathrm{RHom}_{\cD_{gr}(k[v_-])}(F',F(-p))\Big),
\end{equation*}
and the right-hand side of \eqref{j*-on-RHom-Dund-to-DdR} can be computed as 
\begin{align*}
&\mathrm{RHom}_{\cD_{gr}(k[v_-,v_-^{-1},\Theta])}(j^*F',j^*F)\\
=&\mathrm{Fib}\Big(\mathrm{RHom}_{\cD_{gr}(k[v_-,v_-^{-1}])}(j^*F',j^*F)\xrightarrow{\ad_{\Theta}}\mathrm{RHom}_{\cD_{gr}(k[v_-,v_-^{-1}])}(j^*F',j^*F(-p))\Big).
\end{align*}
We will show the stronger statement that, for any $F'\in \cD_{gr}(k[v_-])$ with ${F'}^i$ being acyclic for $i\geq n$, and any $F\in \cD_{gr}(k[v_-])$ with $F^i\xrightarrow{v_-}F^{i-1}$ being an isomorphism for $i\leq n-1$, we have an isomorphism
\begin{equation}\label{RHom-RHomj*}
\overline{j}^*: RHom_{\cD_{gr}(k[v_-])}(F',F)\xrightarrow{\sim} RHom_{\cD_{gr}(k[v_-,v_-^{-1}])}(\overline{j}^*F',\overline{j}^*F).
\end{equation}
This will imply the desired isomorphism in \eqref{j*-on-RHom-Dund-to-DdR}. 

Now, the functor $\overline{j}^*$ admits a right adjoint $\overline{j}_*$ given by the restriction along $k[v_-]\hookrightarrow k[v_-,v_-^{-1}]$.~By adjunction, $\mathrm{RHom}_{\cD_{gr}(k[v_-,v_-^{-1}])}(\overline{j}^*F',\overline{j}^*F)\simeq \mathrm{RHom}_{\cD_{gr}(k[v_-])}(F',\overline{j}_*\overline{j}^*F)$.~Let $F'':=\mathrm{cone}(F\to \overline{j}_*\overline{j}^*F)$.~To prove \eqref{RHom-RHomj*}, it suffices to show $\mathrm{RHom}(F',F'')=0$.~We have ${F''}^i=0$ for $i< n$.~To show the desired vanishing, we use the equivalence $\cD_{gr}(k[v_-])\xrightarrow{\sim} \cD_{gr}(k[v_+])$ given by $F\mapsto G$ with $G_i=F^{-i}$.~Now the vanishing of $\mathrm{RHom}(F',F'')\simeq \mathrm{RHom}(G',G'')$ follows from \eqref{product-formula-RHom-grk}. 

To prove (iii), recall from \cite[\S 6.5.2]{bhatt-lecture-notes} that $\cD_{qc}(D_{Hod})$ can be embedded as a full subcategory of $\cD_{gr}(k[\Theta])$. Denote by $V'=\bigoplus V'_i$ (resp.~$V=\bigoplus V_i$) the corresponding graded modules for $\cE'|_{D_{Hod}}$ (resp.~$\cE|_{D_{Hod}}$). Observe that 
\begin{equation}\label{RHom-step-iii}
    \mathrm{RHom}_{\cD_{gr}(k[\Theta])}(V',V)=\mathrm{Fib}\Big(\mathrm{RHom}_{\cD_{gr}(k)}(V',V)\xrightarrow{\ad_{\Theta}}\mathrm{RHom}(V',V(-p))\Big).
\end{equation}
Using the assumption on weights, $V_i$ is acyclic for $i<n$ and $V_i'$ is acyclic for $i\geq n$. Thus both the source and the target of $\ad_{\Theta}$ on the right-hand side of \eqref{RHom-step-iii} are zero. Therefore, $\mathrm{RHom}_{\cD_{qc}(D_{Hod})}(\cE'|_{D_{Hod}},\cE|_{D_{Hod}})=\mathrm{RHom}_{\cD_{gr}(k[\Theta])}(V',V)=0$. 
\end{proof}
We now check that \eqref{restriction-map-i-to-i-1} is fully faithful.~Let $\cF, \cF'\in \cD_{qc,[0,p-2]}((W(k)^{syn}\otimes\bF_p)_{v_1^i=0})$. 
The morphism 
$\mathrm{RHom}(\cF,\cF')\to \mathrm{RHom}(\alpha_i^*\cF,\alpha_i^*\cF')\simeq \mathrm{RHom}(\cF,\alpha_{i*}\alpha_i^*\cF')$ is induced by 
the map $\cF'\to\alpha_{i*}\alpha_i^*\cF'$. Let $\cG$ be its fibre. By \eqref{ideal-sheaf-iota-i}, we have $\cG:=\cF'\otimes\mathcal{J}_i=\cF'\{-(i-1)(p-1)\}|_{W(k)^{syn}_{red}}$. We need to show that \[0=\mathrm{RHom}(\cF,\cG)=\mathrm{RHom}_{\cD_{qc}(W(k)^{syn}_{red})}\left(\cF|_{W(k)^{syn}_{red}}\ ,\ \cF'\{-(i-1)(p-1)\}|_{W(k)^{syn}_{red}}\right).\]
Note that $\cF|_{W(k)^{syn}_{red}}$ has weights $\leq p-2$, and $\cF'\{-(i-1)(p-1)\}|_{W(k)^{syn}_{red}}$ has weights $\geq (i-1)(p-1)$. Thus the desired vanishing follows by Lemma \ref{generalized-effectivity-vanishing-lemma}. This finishes the proof of fully faithfulness of functor \eqref{restriction-map-i-to-i-1}. 

To show that \eqref{restriction-map-i-to-i-1} is essentially surjective we shall use the following result. 
\begin{lm}\label{deformation-theory-thm}
    Let $i:\mathcal{X}\hookrightarrow \mathcal{Y}$ be a square-zero extension of algebraic stacks over a field $k$. Let $\mathcal{I}$ be the corresponding sheaf of ideals. Let $\cF\in\cD_{qc}(\cX)$. Assume that $\Ext^2_{\cD_{qc}(\cX)}(\cF,\cF\otimes\cI)=0$. Then there exists an $\widetilde{\cF}\in\cD_{qc}(\cY)$ such that $\widetilde{\cF}|_{\cX}=\cF$. 
\end{lm}
\begin{proof}
    We first consider the case where $\cY$ is an affine scheme. Let $R$ be any associative algebra over $k$, and $I\subset R$ a two-sided ideal with $I^2=0$. Set $\overline{R}:=R/I$. Define an object $M\in \cD(\overline{R}\otimes \overline{R}^{\op})$ as follows: $M:=\tau_{\leq 1}(\overline{R}\overset{\mathbb{L}}{\otimes}_R \overline{R})$. We have $H_0(M)=\bar{R}$, $H_1(M)=I$ and $H_i(M)=0$ for all $i\neq 0,1$. For the duration of this proof, we denote by $\Psi$ the functor $\cD(\bar{R})\to \cD(\bar{R})$ given by $\Psi(N):=M\otimes_{\bar{R}} N$. Then we have a fibre sequence of functors
    \begin{equation}\label{eqn:deformation-lemma-proof-fibre-sequence}
    I\otimes_{\bar{R}} N[1]\to \Psi(N)\to N\xrightarrow{\mathfrak{o}(N)} I\otimes_{\bar{R}} N[2].
    \end{equation}
    We refer to $\mathfrak{o}(N)$ as the obstruction class. A choice of a null-homotopy for $\mathfrak{o}(N)$ determines a map $\delta: \Psi(N)\to I\otimes_{\bar{R}} N[1]$ splitting the fibre sequence \eqref{eqn:deformation-lemma-proof-fibre-sequence}. Composing $\delta$ with the map $\bar{R}\otimes_R N\simeq (\overline{R}\overset{\mathbb{L}}{\otimes}_R \overline{R})\otimes_{\bar{R}}N \to \Psi(N)$, we obtain a map $\bar{R}\otimes_R N\to I\otimes_{\bar{R}}N[1]$, which yields a map 
    $\bar{\delta}: N\to I\otimes_{\bar{R}}N[1]$ in $\cD(R)$.
    Let 
    \begin{equation}
    \widetilde{N}:=\mathrm{Fib}(N\xrightarrow{\bar{\delta}}I\otimes_{\bar{R}}N[1])\in \cD(R).
    \end{equation}
The morphism $\widetilde{N}\to N$ in $\cD(R)$ induces a map in $\cD(\bar{R})$:
\begin{equation}\label{eqn:barR-otimes-Ntilde-to-N}
\bar{R}\otimes_R\widetilde{N}\to  N.
\end{equation} 
We claim that the map \eqref{eqn:barR-otimes-Ntilde-to-N} is an isomorphism. To see this, we compute the obstruction class $\mathfrak{o}(N)$ of $N$ as follows. By \cite[B.3 Lemma]{MR2028075} we can assume that $N$ is represented by a semi-free complex of $(N^\cdot, d)$ of $\bar{R}$-modules,  {\it i.e.}, there exists an increasing exhaustive filtration  
$0= F_{0} N^{\cdot} \subset F_{1} N^{\cdot}\subset \cdots \subset (N^\cdot, d)$ such that each quotient $F_{i+1}N^\cdot/F_{i}N^\cdot$ is isomorphic to a direct sum of DG-modules of the form $\bar{R}[n]$.
For each $i$,  pick a lift of the filtered $\bar {R}$-module  $F_{\bullet} N^{i}$ to a filtered $R$-module $F_{\bullet} \widetilde N^{i}$
such that $F_{j+1}\widetilde N^{i}/F_{j}\widetilde N^{i}$ is a free $R$-module, for every $j$.
 We can lift $d$ to a homomorphism of filtered $R$-modules 
$\widetilde d: \widetilde N^i \to  \widetilde N^{i+1}$ such that $\widetilde d$ takes $\widetilde{F}_{\bullet} \widetilde{N}^{\cdot}$ to $\widetilde{F}_{\bullet-1} \widetilde{N}^{\cdot+1}$. In general, $\widetilde d^2\ne 0$.  
Note that $\widetilde d^2$ factors through the projection $\widetilde N^i \to N^i$ and its image lies in  $I \otimes_{\bar{R}} N^{i+2}$. Moreover, since $\widetilde d^2$ commutes with $\widetilde d$, we get a morphism of complexes 
$\widetilde d^2: (N^\cdot, d)  \to  I \otimes_{\bar{R}} (N^\cdot, d)[2]$.
\begin{claim}
    The above morphism $\widetilde{d}^2: N \to I \otimes_{\bar{R}} N[2]$ is equal to $\mathfrak{o}(N)$.
\end{claim}
\begin{proof}
    
Consider the DG-algebra $\bar{R}^{dg} := (I \to R)$ coming from the resolution $0 \to I \to R \to \bar{R} \to 0$. Let us define a DG-$\bar{R}^{dg}$-module $N'$ whose underlying graded $R$-module is $\widetilde{N} \oplus I \otimes_{\bar{R}} N[1]$ and the differential in degree $i$ is given by $D:= \begin{pmatrix}
\widetilde{d} & (-1)^{i} \beta \\
(-1)^i \widetilde{d^2} & d
\end{pmatrix}$, where $\beta: I \otimes_{\bar{R}} N^i \mono \widetilde{N}^i$ is the kernel of $\widetilde{N}^i \to N^i.$ The action 
$I[1]$ on $N'$ is given by the lower triangular matrix whose only non-zero entry is given by $I[1] \otimes_{R} \widetilde{N} \simeq I \otimes_{\bar{R}} N[1]$.
 Note that by construction $N'$ is semi-free over $\bar{R}^{dg}.$ Consider the map $N' \to N$ which takes $(n, x) \in \widetilde{N}^i \oplus I \otimes_{\bar{R}} N^{i+1}$ to the image of $n$ under the map $\widetilde{N}^i \to N^i$. It is a map of DG-$\bar{R}^{dg}$-modules and, moreover, it is a quasi-isomorphism. Under the equivalence $\cD(\bar{R}) \simeq \cD(\bar{R}^{dg})$, the functor $\Psi: \cD(\bar{R}) \to \cD(\bar{R})$ corresponds to the functor $\Psi': \cD(\bar{R}^{dg}) \to \cD(\bar{R})$ given by $\Psi'(N) = \tau_{\le 1} (\bar{R} \otimes_{R} \bar{R}^{dg}) \otimes^{\mathbb L}_{\bar{R}^{dg}} N.$ Note that the natural map $\tau_{\le 1} (\bar{R} \otimes^{\mathbb{L}}_{R} \bar{R}^{dg}) \to \tau_{\le 1} (\bar{R} \otimes_{R} \bar{R}^{dg})$ in $\cD(\bar{R} \otimes (\bar{R}^{dg})^{op})$ is an isomorphism, and $\tau_{\le 1} (\bar{R} \otimes_{R} \bar{R}^{dg})$ is represented by $\bar{R} \oplus I[1]$ with the zero differential. Observe that the right action of $\bar{R}^{dg}$ has the property that $\bar{R} \otimes_{R} (\bar{R}^{dg})^{-1} \to I$ is an isomorphism.
Thus, $\Psi'(N)$ is identified with $C:= (N \oplus I \otimes_{\bar{R}} N[1], \bar{D})$, where $\bar{D}$ in degree $i$ is given by $\begin{pmatrix}
d & 0 \\
(-1)^i \widetilde{d^2} & d
\end{pmatrix}$. Denote $u: I \otimes_{\bar{R}} N[1] \to C$  the inclusion and by $\alpha: C \to N$ the projection. To complete the proof, it is enough to construct a quasi-isomorphism $f: I \otimes_{\bar{R}} N[2] \to \operatorname{Cone}(\alpha)$ such that the two morphisms $\mathfrak{o}(N)$ and $\widetilde{d^2} \circ f: N \to \operatorname{Cone}(\alpha)$ are equal. Consider the following diagram 
\[\begin{tikzcd}
	N && {\operatorname{Cone}(\alpha)} && {C[1]} && {N[1]} \\
	\\
	N && {I \otimes_{\bar{R}} N[2]}
	\arrow["{\mathfrak{o}(N)}", from=1-1, to=1-3]
	\arrow[Rightarrow, no head, from=1-1, to=3-1]
	\arrow[from=1-3, to=1-5]
	\arrow["{\alpha[1]}", from=1-5, to=1-7]
	\arrow["{\widetilde{d}^2}", from=3-1, to=3-3]
	\arrow["f", from=3-3, to=1-3]
	\arrow["{u[1]}", from=3-3, to=1-5]
	\arrow["0", from=3-3, to=1-7].
\end{tikzcd}\]
The composition $\alpha[1] \circ u[1]: I \otimes_{\bar{R}} N[2] \to N[1]$ is zero. Thus, there exists a map $f: I \otimes_{\bar{R}} N[2] \to \operatorname{Cone}(\alpha)$ satisfying $f \circ \widetilde{d}^2 = \mathfrak{o}(N)$. Moreover, $f$ is a quasi-isomorphism.     
\end{proof}
Let us now return to the proof of the isomorphism property of \eqref{eqn:barR-otimes-Ntilde-to-N}. 
Since $(N^{\cdot},d)$ is semi-free, the null-homotopy for  $\mathfrak{o}(N)$ determines a collection $h: N^i \to I \otimes_{\bar{R}} N^{i+1}$ with $h d +d h = \widetilde d^2$. Setting $\widetilde d' = \widetilde d - h:  \widetilde N^i \to  \widetilde N^{i+1}$ we have that $(\widetilde d')^2=0$. Then $\widetilde{N}$ is represented by the complex $(\widetilde N^\cdot, \widetilde d' )$. The isomorphism property of \eqref{eqn:barR-otimes-Ntilde-to-N} follows. 

This implies the lemma in the affine case. Indeed, if $\Ext^2_{\bar{R}}(N,N\otimes I)=0$, then $\mathfrak{o}(N)$ is homotopic to zero. Therefore, there exists $\widetilde{N}\in \cD(R)$ together with a map \eqref{eqn:barR-otimes-Ntilde-to-N} which is an isomorphism.

For the general case, observe that $\cD_{qc}(\mathcal{Y})=\lim \cD_{qc}(R)$ where the limit is taken over all flat maps $\Spec R\to \mathcal{Y}$. Thus applying the previous construction to each $R$, we obtain an obstruction class $\mathfrak{o}(\cF): \cF\to\cF\otimes\cI[2]$. Moreover, a choice of null-homotopy for $\mathfrak{o}(\cF)$ determines an object $\widetilde{\cF}\in\cD_{qc}(\cY)$ together with a map $\widetilde{\cF}|_{\cX}\to \cF$. By the local computations in the affine case, this map is an isomorphism. This completes the proof of the lemma. 
\end{proof}
Let us check that \eqref{restriction-map-i-to-i-1} is essentially surjective. Applying Lemma \ref{deformation-theory-thm} to $\alpha_i$ in \eqref{closed-embedding-iota-i}, it suffices to verify that for every $\cF\in\cD_{qc,[0,p-2]}\left((W(k)^{syn}\otimes\bF_p)_{v_1^{i-1}=0}\right)$, the group $\Ext^2(\cF,\cF\otimes\cI_i)=0$. Using \eqref{ideal-sheaf-iota-i}, we have that $\Ext^2(\cF,\cF\otimes\cI_i)=\Ext^2(\cF|_{W(k)^{syn}_{red}},\cF|_{W(k)^{syn}_{red}}\{-(i-1)(p-1)\})$. Note that $\cF|_{W(k)^{syn}_{red}}$ has weights $\leq p-2$, and $\cF|_{W(k)^{syn}_{red}}\{-(i-1)(p-1)\}$ has weights $\geq (i-1)(p-1)\geq p-1$ (since $i-1\geq 1$). Thus the vanishing for $\Ext^2$ follows from Lemma \ref{generalized-effectivity-vanishing-lemma}. 

For the last assertion of the proposition, it suffices to 
observe that the functor
\begin{equation}\label{restriction-map-r-lim}
    \cD_{qc,[0,p-2]}(W(k)^{syn}\otimes\bF_p)\to \limfrom\cD_{qc,[0,p-2]}((W(k)^{syn}\otimes\bF_p)_{v_1^{i}=0})
\end{equation}
is an equivalence by the definition of reduced locus. 
\end{proof}

\subsubsection{Step 3:}
We shall prove that the functor $(\Phi_{\mathrm{Maz}}\otimes \bF_p) \circ r^{-1}_{[0,p-2]}: \cD_{qc,[0,p-2]}(W(k)^{syn}_{red})  \to \mathscr{DMF}^{\mathrm{big}}_{[0,p-2]}(W(k))\otimes\bF_p$ is an equivalence. Although this result does not hold with $r_{[0,p-2]}$ replaced by $r_{[0,p-1]}$, some intermediate results hold in the larger range. Aimed at some geometric applications (see Remark \ref{rem:Mazur_vs_Deligne_Illusie}), we explain the results in full generality.

We start by showing that the functor $\Phi_{\mathrm{Maz}}\otimes\bF_p: \cD_{qc,[0,p-1]}(W(k)^{syn}\otimes\bF_p)\to  \mathscr{DMF}^{\mathrm{big}}_{[0,p-1]}(W(k))\otimes\bF_p$ factors through the restriction $r_{[0,p-1]}$ to the reduced locus $W(k)^{syn}_{red}$.\footnote{This assertion would be obvious if $[0,p-1]$ is replaced by $[0,p-2]$ because $r_{[0,p-2]}$ is an equivalence by \textit{Step 2}.} 
Since $k^{\cN}\otimes\bF_p$ (resp.~$\bA^1_-/\bG_m \otimes \bF_p$) is reduced, the morphism $\Cris: k^{\cN}\otimes\bF_p\to W(k)^{\cN}\otimes\bF_p$ (resp.~$\mathfrak{p}_{\bar{\dR}} \otimes \bF_p:\bA^1_-/\bG_m \otimes \bF_p \to W(k)^{\cN}\otimes\bF_p$) factors through $W(k)^{\cN}_{red}$. Thus we have the following diagram obtained from \eqref{Drinfeld-diagram} by tensoring with $\bF_p$:
\begin{equation}\label{Drinfeld-diagram-tensor-Fp}
    \begin{tikzcd}
    &\Spec (B\otimes\bF_p)/\bG_{m,k}\arrow[]{ld}[swap]{a}\arrow[]{rd}{\Reesparameter\circ a}&\\
    k^{\cN}\otimes\bF_p\arrow[dotted]{rr}{\Reesparameter}\arrow[]{rd}[swap]{\Cris}&& (\bA^1_-/\bG_m)\otimes\bF_p\arrow[]{ld}{\mathfrak{p}_{\bar{\dR}}}\\
    &W(k)^{\cN}_{red}&
\end{tikzcd}
\end{equation}
By Theorem \ref{Drinfeld-prop}, we have an isomorphism $\Psi_{\mathrm{Maz}}\otimes\bF_p: \Cris\circ a\xrightarrow{\sim} \mathfrak{p}_{\bar{\dR}}\circ t\circ a$ in diagram \eqref{Drinfeld-diagram-tensor-Fp} above. 
Consequently, this gives a functor 
\[
    \cD_{qc}(W(k)^{syn}_{red})\to \cD_{qc}(k^{syn}\otimes\bF_p)\times_{\cD_{qc}(\Spec (B\otimes\bF_p)/\bG_{m, k})}\cD_{qc}((\bA^1_-/\bG_m)\otimes\bF_p)\]
    sending $\cD_{qc,[0,p-1]}(W(k)^{syn}_{red})$ to the corresponding subcategory of the right-hand side (as in Remark \ref{remark-construction-functor-Psi-Mazur}). Combining this with \eqref{second-displayed-functor-Phi}, we obtain 
\begin{equation}\label{eqn:Phi-Maz-reduced}
\Phi^{\mathrm{Maz}}_{red}: \cD_{qc,[0,p-1]}(W(k)^{syn}_{red})\to\mathscr{D}\mathscr{MF}^{\mathrm{big}}_{[0,p-1]}(W(k))\otimes\bF_p
\end{equation}
together with an isomorphism $\Phi_{\mathrm{Maz}}\otimes\bF_p\simeq \Phi^{\mathrm{Maz}}_{red}\circ r_{[0,p-1]}$. Eventually in this step, we shall prove that $\Phi_{red}^{\mathrm{Maz}}$ is an equivalence in the range $[0,p-2]$.

Note that the construction of $\Phi^{\mathrm{Maz}}_{red}$ makes use of a particular choice of an isomorphism $ \Cris\circ a\simeq \mathfrak{p}_{\bar{\dR}}\circ t\circ a$  provided by $\Psi_{\mathrm{Maz}}\otimes\bF_p$.
We will make use of another isomorphism $\Psi_{\mathrm{DI}}: \Cris\circ a\xrightarrow{\sim} \mathfrak{p}_{\bar{\dR}}\circ t\circ a$ constructed as follows.\footnote{Here ``DI'' refers to Deligne--Illusie.} 
Let $C_2$ be the spectrum of the subalgebra of $k[v_+,v_-]/(v_+v_-)$ generated by $v_+^p$ and $v_-$, {\it i.e.},~$C_2:=\Spec k[v_+^p,v_-]/(v_+^pv_-)$.  
We endow the affine curve $C_2$ with a natural action of $\bG_{m,k}$ given by $\deg v_+^p=p$ and $\deg v_-=-1$. Recall from \cite[\S 5.16.10]{drinfeld2022prismatization}
 the morphism $C_2/\bG_{m,k}\to W(k)^{\cN}_{red}$, whose precomposition with $ k^{\cN}\otimes\bF_p\to C_2/\bG_{m,k}$ is $\Cris\otimes\bF_p$ (see also \eqref{map-C2-to-WkNygaard}). Let us recall how the map $C_2/\bG_{m,k} \to W(k)_{red}^{\cN}$ is constructed. Note $C_2/\bG_{m,k}(R)$ is the groupoid of diagrams 
 \begin{equation}\label{points-of-C_2/Gm}
     L \xrightarrow{v_{-}} R \xrightarrow{v_+^p} L^{\otimes p}
 \end{equation}  where $L$ is an invertible  $R$-module and $v_+^pv_{-}=0.$ To any such diagram we associate an admissible Cartier-Witt divisor $M \to W$ over $R$. 
 Here $M$ is defined to be the fiber product of $F_* W_R \xrightarrow{F_* [v_+^p]} F_* [L^{\otimes p}]$, where $[L] \in \Pic(W(R))$ is the Teichm\"uller representative of $L$ as in \cite[\S 3.11]{drinfeld2022prismatization}, and the Frobenius $[L] \to F_* [L^{\otimes p}]$. The map $M \to W$ is the restriction of $[L] \times_{R} F_* W_R \xrightarrow{([v_-], V)} W_R.$     We have a commutative diagram (the commutativity data are explained below)
 \begin{equation}\label{DI-commutative-diagram}
 \begin{tikzcd}
    &\Spec (B\otimes\bF_p)/\bG_{m,k}\arrow[]{ld}[swap]{a}\arrow[]{rd}{\Reesparameter\circ a}&\\
    k^{\cN}\otimes\bF_p\arrow[dotted]{rr}{\Reesparameter}\arrow[]{rd}{\widetilde{\mathfrak{p}}_{\mathrm{cris}}}\arrow[shift right]{rdd}[swap]{\Cris}&& (\bA^1_-/\bG_m)\otimes\bF_p\arrow[]{ld}[swap]{\widetilde{\mathfrak{p}}_{\bar{\dR}}}\arrow[]{ldd}{\mathfrak{p}_{\bar{\dR}}}\\
    &C_2/\bG_{m,k}\arrow[]{d}{}&\\
    &W(k)^{\cN}_{red}&
\end{tikzcd}
\end{equation}
The map $\widetilde{\mathfrak{p}}_{\mathrm{cris}}$ is given by the embedding $k[v_+^p,v_-]/(v_+^pv_-)\hookrightarrow k[v_+,v_-]/(v_+v_-)$, and the map $\widetilde{\mathfrak{p}}_{\bar{\dR}}$ is given by the map $k[v_+^p,v_-]/(v_+^pv_-)\to k[v_-]$. Since $v_+^p=p\cdot \frac{v_+^p}{p}$ is zero in $B\otimes\bF_p$, we obtain an isomorphism $\widetilde{\mathfrak{p}}_{\mathrm{cris}}\circ a\simeq \widetilde{\mathfrak{p}}_{\bar{\dR}}\circ t\circ a$. To see the commutativity of the bottom right triangle, note $\tilde{\mathfrak{p}}_{\bar{\dR}}: ((\bA^1_{-}/\bG_m)\otimes\bF_p) (R) \to (C_2/\bG_{m,k})(R)$ is given by taking $(L, v_-: L \to R) \in ((\bA^1_{-}/\bG_m) \otimes\bF_p)(R)$ to a diagram $L \xrightarrow{v_-} R \xrightarrow{0} L^{\otimes p}.$ Using the commutativity datum for the lower left triangle in \eqref{DI-commutative-diagram} from \cite[\S 7.8]{drinfeld2022prismatization} (which gives rise to a commutativity datum of the lower right triangle), we obtain an isomorphism $\Psi_{\mathrm{DI}}: \Cris\circ a\xrightarrow{\sim} \mathfrak{p}_{\bar{\dR}}\circ t\circ a$. 

The composition $\psi:=\Psi_{\mathrm{DI}}\circ(\Psi_{\mathrm{Maz}}\otimes\bF_p)^{-1}$ is an automorphism of the point $\mathfrak{p}_{\bar{\dR}}\circ t\circ a$. We now describe it explicitly. Recall that the composition of $\mathfrak{p}_{\bar{\dR}}\circ t\circ a$ with the projection $\Spec (B\otimes\bF_p)\to \Spec (B\otimes\bF_p)/\bG_{m,k}$ is given by the morphism $W_{B\otimes\bF_p}^{(F)}\oplus F_*W_{B\otimes\bF_p}\xrightarrow{(v_-,V)}W_{B\otimes\bF_p}$ of admissible $W_B$-modules. 
Recall from \cite[Proposition 5.2.1 (2)]{bhatt-lecture-notes} an isomorphism 
\begin{equation}\label{isom-Hom-FWB-Ker}
\Hom_{W_{B\otimes\bF_p}}(F_*W_{B\otimes\bF_p},W_{B\otimes\bF_p}^{(F)})\xrightarrow{\sim}\ker\left(W^{(F)}(B\otimes\bF_p)\xrightarrow{w_1} B\otimes\bF_p\right),
\end{equation}
where the map $w_1$ is the first Witt coordinate $w_1: W\to \bG_a$. The map in \eqref{isom-Hom-FWB-Ker} is the precomposition with Frobenius $F: W\to F_*W$, {\it i.e.},
\[\Hom_{W_{B\otimes\bF_p}}(F_*W_{B\otimes\bF_p},W_{B\otimes\bF_p}^{(F)})\xrightarrow{-\circ F} \Hom_{W_{B\otimes\bF_p}}(W_{B\otimes\bF_p},W_{B\otimes\bF_p}^{(F)})=W^{(F)}(B\otimes\bF_p).\]
Let $\{\frac{[v_+^p]}{p}\}$ be the image of $\frac{[v_+^p]}{p}\in W(B)$ inside $W(B\otimes\bF_p)$. Recall that by Lemma \ref{divisibility-W(B)-lemma}, the element $[v_+^p]$ is indeed divisible by $p$ inside $W(B)$. 
\begin{lm}\label{DI-Maz}
The map $\psi$, viewed as an automorphism of $W_{B\otimes\bF_p}^{(F)}\oplus F_*W_{B\otimes\bF_p}$, is given by the matrix $\begin{pmatrix}
    \Id&V(\{\frac{[v_+^p]}{p}\})\\
    0& \Id
\end{pmatrix}$, where the entry $V(\{\frac{[v_+^p]}{p}\})$ in the matrix is viewed as an element of $\Hom_{W_{B\otimes\bF_p}}(F_*W_{B\otimes\bF_p},W_{B\otimes\bF_p}^{(F)})$ via the isomorphism \eqref{isom-Hom-FWB-Ker}. 
\end{lm}
\begin{proof}
Recall from \eqref{eqn:preliminary-isom-admissible-modules-PsiDI} that the isomorphism $\Psi_{\mathrm{DI}}$: 
\[\coker\left(W_{B\otimes\bF_p}^{(F)}\to W_{B\otimes\bF_p}^{(F)}\oplus W_{B\otimes\bF_p}\right)\xrightarrow{\sim} W_{B\otimes\bF_p}\underset{F_*W_{B\otimes\bF_p}}{\times}F_*W_{B\otimes\bF_p}=W_{B\otimes\bF_p}^{(F)}\oplus F_*W_{B\otimes\bF_p}\]
is given by the matrix $X:=\begin{pmatrix}
  \Id &    [v_+]\\
 0&  F
\end{pmatrix}$. 
Recall from the proof of Theorem \ref{Drinfeld-prop} that $\Psi_{\mathrm{Maz}}\otimes\bF_p$ is given by 
$Y:=\begin{pmatrix}
    1& f\\
    0&F
\end{pmatrix}$, 
where $f:=[v_+]-V(\{\frac{[v_+^p]}{p}\})$. Our desired lemma amounts to the statement that $\psi Y=X$. This is immediate from our construction of the isomorphism \eqref{isom-Hom-FWB-Ker}.
\end{proof}
\begin{cor}
The restrictions of $\Psi_{\mathrm{DI}}$ and $\Psi_{\mathrm{Maz}}\otimes\bF_p$ to the closed subscheme $\left(\Spec k[v_+,v_-]/(v_+^{p-1},v_+v_-)\right)/\bG_{m,k}\hookrightarrow (\Spec B\otimes\bF_p)/\bG_{m,k}$ are equal. 
\end{cor}
\begin{proof}
This follows from the fact that the image of $\frac{[v_+^p]}{p}\in W(B)$ inside $W\left(k[v_+,v_-]/(v_+^{p-1},v_+v_-)\right)$ is zero.    
\end{proof}
Using the construction from Remark \ref{remark-construction-functor-Psi-Mazur} once again, we obtain a functor 
\begin{equation}\label{eq:Deligne_Illusie_functor}
  \Phi_{\mathrm{DI}}: \cD_{qc,[0,p-1]}(W(k)^{syn}_{red})\to\mathscr{D}\mathscr{MF}^{\mathrm{big}}_{[0,p-1]}(W(k))\otimes\bF_p.  
\end{equation}
\begin{cor}\label{DI=Mar}
The restrictions of $\Phi^{\mathrm{Maz}}_{red}$ and $\Phi_{\mathrm{DI}}$ to the subcategory $\cD_{qc,[0,p-2]}(W(k)^{syn}_{red})$ are isomorphic.     
\end{cor}
The following proposition will complete the proof of the main Theorem \ref{main-thm}.
\begin{pr}\label{prop:DI_is_an_equiv}
$\Phi_{\mathrm{DI}}$ is an equivalence of categories.     
\end{pr}
\begin{proof}
Let us first outline the argument. Recall from the proof of Lemma \ref{generalized-effectivity-vanishing-lemma} a fully faithful embedding $$\cD_{qc} (W(k)^{syn}_{red}) \mono \cD_{qc}(D_{\HT})\underset{\cD_{qc}(D_{und})\times \cD_{qc}(D_{Hod})}{\times}\cD_{qc}(D_{\dR}).$$ The first part of the proof consists of showing that the restriction $\cD_{qc, [0, p-1]} (D_{\HT}) \to \cD_{qc, [0, p-1]} (D_{Hod})$ is an equivalence (see Corollary \ref{Coro-Dqc-DHT-equiv-DHod}). This yields a fully faithful embedding
\begin{equation}\label{embedding-sheaves-red-locus}
\cD_{qc, [0, p-1]} (W(k)^{syn}_{red})) \mono \operatorname{Eq}
\begin{tikzcd}\Big(\cD_{qc,[0,p-1]}(D_{\dR})\arrow[shift left]{r}{\widetilde{i}_1^*}\arrow[shift right]{r}[swap]{T^*\circ \mathfrak{p}_{\bar{\dR}}^*}&\cD_{qc}(D_{und})\Big)\end{tikzcd},
\end{equation}  
where $\widetilde{i}_1^*$ is the restriction and $T$ is the composition $B\bG_{m, \bF_p}^{\sharp} \otimes k \xrightarrow{\id \times \operatorname{Frob}} B\bG_{m, \bF_p}^{\sharp} \otimes k \to B\bG_{m, k} \to (\bA^1_{-}/\bG_m) \otimes \bF_p.$ 
Next, we show that pullbacks along $\mathfrak{p}_{\bar{\dR}}$ and the inclusion $D_{und} \mono D_{\dR}$ induce an equivalence $\cD_{qc, [0, p-1]} (D_{\dR}) \iso \cD_{qc, [0, p-1]} ((\bA^1_{-}/\bG_m) \otimes \bF_p) \times_{\cD(k)} \cD_{qc} (D_{und})$, see Lemma \ref{Dqc-DdR-as-a-fibre-product}.
This, together with a category theory result, Lemma \ref{category-fibre-product-coequalizer-lemma}, 
identifies  \eqref{embedding-sheaves-red-locus} with 
\begin{equation}\label{embedding-sheaves-red-locus-bis}
\cD_{qc, [0, p-1]} (W(k)^{syn}_{red})) \mono \mathrm{Eq}\begin{tikzcd}\Big(\cD_{qc,[0,p-1]}((\bA^1_-/\bG_m) \otimes \bF_p)\arrow[shift left]{r}{{i}_1^*}\arrow[shift right]{r}[swap]{F^*\circ i_0^*}&\cD(k)\Big)\end{tikzcd}.
\end{equation}  
The functor $\Phi_{DI}$ is isomorphic  to \eqref{embedding-sheaves-red-locus-bis}. This proves fully faithfulness of   $\Phi_{DI}$.
For the essential surjectivity we use Lemma \ref{lemma:essential-image-restriction-functor-fibre-product}. 

Let us explain the details. 
Consider the maps $ (C_2)_{v_-=0}/\bG_{m,k}\overset{f'}{\twoheadrightarrow}  D_{\HT}=(\bA^{1,\dR}_+/\bG_m)\otimes\bF_p\xrightarrow{f} B\bG_{m, k}$ and a section $B\bG_{m, k}=(C_2)_{v_+^p=v_-=0}/\bG_{m,k}\overset{g}{\hookrightarrow} (C_2)_{v_-=0}/\bG_{m,k}$. We first show the following.  
\begin{lm}\label{lemma-equiv-BGm-HT-C2locus}
    The pullback functors along $f$ and $f'$ induce equivalences\footnote{By definition, $\cD_{qc,[0,p-1]}\left((C_2)_{v_-=0}/\bG_m\right)$ is a full subcategory of $\cD_{qc}\left((C_2)_{v_-=0}/\bG_m\right)$ formed of objects whose pullbacks to $\bA^1_+/\bG_m\otimes\bF_p$ have weights in $[0,p-1]$.}
$$\cD_{qc,[0,p-1]}(B\bG_{m, k})\xrightarrow{\sim} \cD_{qc,[0,p-1]}(D_{\HT})\xrightarrow{\sim} \cD_{qc,[0,p-1]}\left((C_2)_{v_-=0}/\bG_{m,k}\right).$$ 
The inverse to the composite is given by $g^*$. 
\end{lm}
\begin{proof}
Since $f\circ f' \circ g=\Id$, 
we have that $g^*\circ {f'}^*\circ  f^*\simeq \Id$. Thus, it suffices to check that $f^*$ and ${f'}^* \circ f^*$ are equivalences. 

Let us show that $f^*$ is fully faithful. Let us recall from \cite[\S 6.5.4]{bhatt2022prismatization} that the category $\cD_{qc}(D_{\HT})=
\cD_{qc}(\bG_a^{dR}/\bG_{m,k})$ is identified with the full subcategory $\cD_{gr,D^p\text{-nilp}}(k[D^p,v_+^p])$ of the derived category $\cD_{gr}(k[D^p, v_+^p])$ of graded modules over the polynomial algebra $k[D^p,v_+^p]$ with $\deg v_+^p=-\deg D^p=p$. Objects of  $\cD_{gr,D^p\text{-nilp}}(k[D^p,v_+^p])$ consist of complexes $\cM$ such that the action of $D^p$ on $\bigoplus\limits_iH^i(\cM)$ is locally nilpotent.  The natural map $\bG_a^{dR} := \bG_{a, k}/\bG_{a, k}^{\sharp} \to \bG_{a, k}/\alpha_p = F_* \bG_{a, k}$ coming from the factorization $\bG_{a, k}^{\sharp} \to \alpha_p \to \bG_{a, k}$ exhibits the source as the $BF_* \bG_{a, k}^{\sharp}$-torsor over the target. A flat lift of $\bG_{a, k}$ to $W(k)$ together with a lift of the Frobenius
 splits the torsor; \cite[Remark 5.13]{bhatt2022prismatization}. The standard choice of the Frobenius lift gives a $\bG_{m, k}$-equivariant equivalence $\bG_a^{dR} \simeq F_*\bG_{a, k} \oplus BF_*\bG_{a, k}^{\sharp}$ and $\cD_{qc}((F_*\bG_{a, k} \oplus BF_*\bG_{a, k}^{\sharp})/\bG_{m, k}) \simeq \cD_{gr,D^p\text{-nilp}}(k[D^p,v_+^p])$.  Under the equivalence above, $\cD_{qc,[0,p-1]}(D_{\HT})$ is identified with the subcategory $\cD_{gr,[0,p-1]}(k[v_+^p,D^p])$ of $\cD_{gr,D^p\text{-nilp}}(k[D^p,v_+^p])$ that consists of objects $\cM=\bigoplus\limits_i\cM_i$ with $\cM_i$ acyclic for $i\leq 0$ and $\mathrm{cone}(\cM_{i-p}\xrightarrow{v_+^p}\cM_i)$ acyclic for $i\geq p$. 

The category $\cD_{qc,[0,p-1]}(B\bG_{m, k})$ is identified with the full subcategory of the derived category $\cD_{gr}(k)$ of graded complexes $C_{\bullet}=\bigoplus\limits_iC_i$ such that for any $i\notin[0,p-1]$, the grading degree $i$ component $C_i$ is acyclic. Under the above identifications, the functor $f^*$ sends $C_i$ to $C_i[v_+^p]:= C_i \otimes_k k[v_{+}^p]$ with the trivial action of $D^p$. The complex $\mathrm{RHom}_{\cD_{gr}(k[D^p,v_+^p])}(C_i[v_+^p],C'_j[v_+^p])$ is computed by \begin{equation}\label{eq:133}
\mathrm{Fib}\left(\mathrm{RHom}_{\cD_{gr}(k[v_+^p])}(C_i[v_+^p],C'_j[v_+^p])
\xrightarrow{\ad_{D^p}}\mathrm{RHom}_{\cD_{gr}(k[v_+^p])}(C_i[v_+^p],C'_j[v_+^p](-p))\right),
\end{equation}
where $(-p)$ stands for the degree shift, {\it i.e.},~$\cM(-p)_i=\cM_{i-p}$. 

If $i,j\in [0,p-1]$, then $\mathrm{RHom}_{\cD_{gr}(k[v_+^p])}(C_i[v_+^p],C'_j[v_+^p](-p))=0$, and therefore by \eqref{eq:133} we get
\begin{equation}\label{eq:12}
\mathrm{RHom}_{\cD_{gr}(k)}(C_i,C'_j)\xrightarrow{\sim} \mathrm{RHom}_{\cD_{gr}(k[v_+^p])}(C_i[v_+^p],C'_j[v_+^p]).
\end{equation}
This proves fully faithfulness of $f^*$. Formula \eqref{eq:12} shows fully faithfulness of ${f'}^* \circ f^*$.

Now we show essential surjectivity. It is enough to show that, for every  $i\in [0,p-1]$ and $\cM\in \cD_{gr,[0,p-1]}(k[D^p,v_+^p])$ such that 
the grading degree $i'$ part $\cM_{i'}\in \cD(k)$ of $\cM$ is acyclic for $i'\not\equiv i\mod p$, 
$\cM$ is quasi-isomorphic to $\cM_i[v_+^p]$.
Let $\tilde \cM$ be a graded complex  of $k[D^p,v_+^p])$-modules, representing $\cM$, such that each term $\tilde \cM^j$ of $\tilde \cM$ (we use cohomological notation) is a free graded module over $k[D^p,v_+^p]$ and the grading degree $i'$ part $\tilde \cM_{i'}^j\in \cD(k)$ of $\tilde \cM_j$ vanishes for $i'\not\equiv i\mod p$.\footnote{Here we make use of the following general fact: for a graded ring $R$, every complex of graded $R$-modules is quasi-isomorphic to a complex of free graded $R$-modules. See, for instance, \cite[Lemma 13.3]{MR2028075} for a proof of the analogous statement for ordinary rings and \cite[\S 14.8]{MR2028075} for a very general assertion that applies to the graded setting.}
In particular, the map $v_+^p:\tilde \cM_j\to \tilde \cM_{j+p}$ is a term-wise injective morphism of complexes. Let $\tilde \cM'$ be a graded complex of $k[D^p,v_+^p])$-modules
 defined as follows: in each cohomological degree, the graded $k[D^p,v_+^p]$-module $\tilde {\cM}^{\prime n}$ is a quotient of $\tilde \cM^n$ by the submodule generated by $\tilde \cM^n_j$ for $j\leq 0$. Then the canonical map $\tilde \cM\to \tilde \cM'$ is a quasi-isomorphism. Since $\tilde \cM'_j$ is a term-wise zero complex for $i\geq j$, the operator $D^p$ acts trivially on $\tilde \cM'_i$. Hence, there exists a unique morphism of graded complexes of  $k[D^p,v_+^p]$-modules
 $\tilde \cM_i[v_+^p]\to \tilde \cM'$, which is the canonical projection $\tilde \cM_i\to \tilde\cM_i/v_+^p\tilde \cM_{i-p}=\tilde \cM'_i$ in degree $i$. By the weight assumption on $\cM$, this morphism is a quasi-isomorphism. Thus we have essential surjectivity of $f^*$. 

Since the map $f': (C_2)_{v_-=0}/\bG_{m,k} \to D_{\HT}= B H|_{(C_2)_{v_-=0}/\bG_{m,k}}$ has a left inverse\footnote{Recall that 
$H$ is a group stack over $C_2/\bG_m$. The left inverse $f''$ is defined to be the restriction of the structure map $BH \to C_2/\bG_m$ to the fiber over 
$(C_2)_{v_-=0}/\bG_{m,k}$.} $f''$, {\it i.e.},~
$f'' \circ f' =\Id$, we have the essential surjectivity of $f'^*$ as desired. Since $f^*$ is an equivalence and $f'^* \circ f^*$ is fully faithful, it implies that $f'^* \circ f^*$ is an equivalence.
\end{proof}

\begin{cor}\label{coro-t*-equiv}
The pullback along the projection $t:C_2/\bG_{m,k}\to \bA^1_-/\bG_m\otimes\bF_p$ induces an equivalence $t^*: \cD_{qc,[0,p-1]}(\bA^1_-/\bG_m\otimes\bF_p)\xrightarrow{\sim}  \cD_{qc,[0,p-1]}(C_2/\bG_{m,k})$.     
\end{cor}
\begin{proof}
Using \cite[Theorem 16.2.0.2]{Lurie-sag}, the category $\cD_{qc,[0,p-1]}(C_2/\bG_{m,k})$ is a full subcategory of the fibre product
\[\cD_{qc,[0,p-1]}\left((C_2)_{v_-=0}/\bG_{m,k}\right){\times}_{{\cD_{qc,[0,p-1]}\left((C_2)_{v_+^p=v_-=0}/\bG_{m,k}\right)}}\cD_{qc,[0,p-1]}\left((C_2)_{v_+^p=0}/\bG_{m,k}\right).\]
By Lemma \ref{lemma-equiv-BGm-HT-C2locus}, composition of $t^*$ with the above embedding is an equivalence. 
\end{proof}

\begin{lm}
The pullback along the composition $D_{Hod}\hookrightarrow D_{\HT}\to B\bG_{m, k}$ induces an equivalence of categories 
    $\cD_{qc,[0,p-1]}(B\bG_{m, k})\xrightarrow{\sim} \cD_{qc,[0,p-1]}(D_{Hod})$.
\end{lm}
\begin{proof}
Recall from \cite[\S 6.5]{bhatt-lecture-notes} that $\cD_{qc}(D_{Hod})$ is identified with the full subcategory $\cD_{gr,\Theta\text{-nilp}}(k[\Theta])$ of $\cD_{gr}(k[\Theta])$ that consists of objects $\cM$ such that the action of $\Theta$ on $\bigoplus\limits_i H^i(\cM)$ is locally nilpotent. 
Under this identification, the pullback functor $\cD_{qc}(B\bG_{m, k})\to \cD_{qc}(D_{Hod})$ sends the graded complex $C_{\cdot}=\oplus C_i$ to $C_{\cdot}$ with a trivial action of $\Theta$. The rest is straightforward.

To see fully faithfulness: we have 
\[\mathrm{RHom}_{\cD_{gr}(k[\Theta])}(C_i,C_j)=\mathrm{Fib}\left(\mathrm{RHom}_{\cD_{gr}(k)}(C_i,C_j)\xrightarrow{\ad_{\Theta}=0} \mathrm{RHom}_{\cD_{gr}(k)}(C_i,C_j(-p))\right).\]
For $i,j\in [0,p-1]$, we have $\mathrm{RHom}_{\cD_{gr}(k)}(C_i,C_j(-p))=0$. Fully faithfulness follows. 

To see essential surjectivity: observe that every $\cM\in\cD_{qc,[0,p-1]}(D_{Hod})$ can be represented by a complex of graded $k[\Theta]$-modules with the trivial action of $\Theta$. 
\end{proof}
Combining the previous two lemmas, we obtain the following:
\begin{cor}\label{Coro-Dqc-DHT-equiv-DHod}
The pullback along the closed embedding $D_{Hod}\hookrightarrow D_{\HT}$ induces an equivalence $\cD_{qc,[0,p-1]}(D_{\HT})\xrightarrow{\sim}\cD_{qc,[0,p-1]}(D_{Hod})$. The restriction to the open substack $\cD_{qc,[0,p-1]}(D_{\HT})\to \cD_{qc}(D_{und})$ is isomorphic to the composition 
\[\cD_{qc,[0,p-1]}(D_{\HT})\to\cD_{qc}(D_{Hod})\to\cD_{qc}(B\bG_{m, k})\to \cD_{qc}(D_{und}),\]
which is induced by pullback along 
$D_{und}=B\bG_{m, k}^{\sharp}\to B\bG_{m, k}\to D_{Hod}\hookrightarrow D_{\HT}$.\footnote{For the reader's convenience, let us recall the definitions of the displayed morphisms of stacks. The identification $D_{und}=B\bG_{m, k}^{\sharp}$
is that from \cite[Remark 6.2.5]{bhatt-lecture-notes}; the map $B\bG_{m, k}^{\sharp}\to B\bG_{m, k}$ is induced by the natural homomorphism of group schemes; 
$B\bG_{m, k}\to D_{Hod}= B(F_*\bG_{a,k}^{\sharp}\rtimes \bG_{m,k})$ is given by 
$\bG_{m, k}\rar{(0, \Id)}F_*\bG_{a,k}^{\sharp}\rtimes \bG_{m,k}$, and the last map realizes $D_{Hod}$ as a closed substack of $D_{\HT}$.
}   
\end{cor}
\begin{proof}
    We just need to explain the second assertion. Note that the post-composition of $D_{und} \to B\bG_{m, k} \to D_{Hod} \mono D_{\HT}$ with $D_{\HT} \to B\bG_{m, k}$ is isomorphic to the standard map $B\bG_{m, k}^{\sharp} \to B\bG_{m, k}.$ The same is true for the open immersion $D_{und} \to D_{\HT}$. This is enough since $f^*: \cD_{qc, [0, p-1]}(B\bG_{m, k}) \to \cD_{qc, [0, p-1]}(D_{\HT})$ is an equivalence by Lemma \ref{lemma-equiv-BGm-HT-C2locus}.
\end{proof}

\begin{lm}\label{Dqc-DdR-as-a-fibre-product}
    The commutative diagram of stacks 
\begin{equation}
    \begin{tikzcd}
        ((\bA^1_-\setminus \{0\} )/\bG_m) \otimes \bF_p \arrow[]{r}{}\arrow[]{d}{}&(\bA^1_-/\bG_m) \otimes \bF_p \arrow[hook]{d}{}\\
        D_{und}\arrow[hook]{r}{}&D_{\dR}
    \end{tikzcd}
    \end{equation}
    induces an equivalence of categories 
    \begin{equation}\label{Dqc-DdR-fibre-product-equiv}
   \mathcal{L}: \cD_{qc,[0,p-1]}(D_{\dR})\xrightarrow{\sim}\cD_{qc,[0,p-1]}((\bA^1_-/\bG_m) \otimes \bF_p) \times_{\cD(k)}\cD_{qc}(D_{und}),
    \end{equation}
    where $\cD(k)$ is identified with $\cD_{qc}(((\bA^1_-\setminus \{0\})/\bG_m) \otimes \bF_p)$. 
\end{lm}
\begin{proof}
To show fully faithfulness of the functor \eqref{Dqc-DdR-fibre-product-equiv}: we shall use the embedding $\cD_{qc}(D_{\dR})\hookrightarrow\cD_{gr}(k[v_-,\Theta])$ as in the proof of Lemma \ref{generalized-effectivity-vanishing-lemma} and the references therein. For $\cF,\cF'\in \cD_{qc,[0,p-1]}(D_{\dR})$, let $F,F'\in 
\cD_{gr,[0,p-1]}(k[v_-,\Theta])$ be the corresponding graded modules. We have
\[\mathrm{RHom}(F',F)=\mathrm{Fib}\left(\mathrm{RHom}_{\cD_{gr}(k[v_-])}(F',F)\xrightarrow{\ad_{\Theta}}\mathrm{RHom}_{\cD_{gr}(k[v_-])}(F',F(-p))\right).\]
On the other hand, 
\[\mathrm{RHom}(\mathcal{L}F',\mathcal{L}F)=\mathrm{Fib}\left(\mathrm{RHom}_{\cD_{gr}(k[v_-])}(F',F)\xrightarrow{\ad_{\Theta}}\mathrm{RHom}_{\cD_{gr}(k[v_-])}(\overline{j}^*F',\overline{j}^*F(-p))\right).\]
Using \eqref{RHom-RHomj*} and the assumption on weights, the morphism $$\overline{j}^*:\mathrm{RHom}_{\cD_{gr}(k[v_-])}(F',F(-p)) \to \mathrm{RHom}_{\cD_{gr}(k[v_-])}(\overline{j}^*F',\overline{j}^*F(-p))$$ is an isomorphism. 
This proves fully faithfulness of $\cL$. 

To show essential surjectivity, observe that any object of the right-hand side of \eqref{Dqc-DdR-fibre-product-equiv} can be represented by a complex $F=\bigoplus F^i$ of graded $k[v_-]$-modules, where $F^i$ is term-wise zero for $i\geq p$, and the map $v_-: F^i\to F^{i-1}$ is a term-wise isomorphism for $i\leq 0$, together with a $\theta: F^0\to F^0$. Then there exists a unique $\Theta:F^{\bullet}\to F^{\bullet-p}$ that commutes with $v_-$ and such that $\Theta:F^0\to F^{-p}$ is precisely $v_-^p\theta$. This proves essential surjectivity as desired. 
\end{proof}
Let us return to the proof of the proposition. Consider the maps 
\begin{equation}\label{big-3row-equalizer-diagram}
    \begin{tikzcd}\mathrm{Eq}\Big(\cD_{qc,[0,p-1]}(\bA^1_-/\bG_m\otimes\bF_p)\arrow[]{d}{t^*}[swap]{\simeq}\arrow[shift left]{r}{i_1^*}\arrow[shift right]{r}[swap]{F^*\circ i_p^*}&\cD_{qc}(\Spec k)\arrow[]{d}{\Id}[swap]{\simeq}\Big)\\
    \mathrm{Eq}\Big(\cD_{qc,[0,p-1]}(C_2/\bG_{m,k})\arrow[shift left]{r}{}\arrow[shift right]{r}[swap]{}&\cD_{qc}(\Spec k)\Big)\\
    \mathrm{Eq}\Big(\cD_{qc,[0,p-1]}(W(k)^{\cN}_{red})\arrow[]{u}{}\arrow[shift left]{r}{}\arrow[shift right]{r}[swap]{}&\cD_{qc}(D_{und})\Big)\arrow[]{u}{}\end{tikzcd}
\end{equation}
The equalizer in the first row is precisely the category $\mathscr{DMF}^{\mathrm{big}}_{[0,p-1]}(W(k))\otimes\bF_p$ from Definition \ref{Defn-big-DMF-category}. The two horizontal arrows in the second row are given by the restriction to the point $(v_+^p=0, v_-=1)$, and the restriction to the point $(v_+^p=1,v_-=0)$ post-composed with the Frobenius.~The downward arrows induce an equivalence of equalizers by Corollary \ref{coro-t*-equiv}.~The equalizer at the bottom arrow is $\cD_{qc,[0,p-1]}(W(k)^{syn}_{red})$. The functor $\Phi_{\mathrm{DI}}$ can be interpreted as the composition of the upward map on equalizers composed with the inverse of the downward equivalence on equalizers. 
Thus it remains to check that the upward map on equalizers is an equivalence. 

Recall ({\it cf.} Lemma \ref{lm:ffembedding})  that the category 
$\cD_{qc,[0,p-1]}(W(k)^{syn}_{red})$ is a full subcategory of the fiber product
\begin{equation}\label{eqn:prop-Phi-DeligneIllusie-fibre-product}
\cC:= \cD_{qc,[0,p-1]}(D_{\HT})\underset{\cD_{qc}(D_{und})\times \cD_{qc}(D_{Hod})}{\times}\cD_{qc,[0,p-1]}(D_{\dR}).
\end{equation}
where the functors $\cD_{qc}(D_{\HT}),\; \cD_{qc}(D_{\dR}) \to \cD_{qc}(D_{Hod})$ and
 $\cD_{qc}(D_{\dR})\to \cD_{qc}(D_{und})$ are given by the 
{\it $W(k)$-linear} embeddings of underlying stacks; the functor $\cD_{qc}(D_{\HT})\to \cD_{qc}(D_{und})$  
 is the pullback along the morphism ({\it cf.} the end of \S \ref{subsec:preliminary-syntomification-section})
 $$D_{und}=  B\bG_{m, \bF_p}^{\sharp} \otimes k \xrightarrow{\id \times \operatorname{Frob}}B\bG_{m, \bF_p}^{\sharp}\otimes k= D_{und}\mono
 D_{\HT}.$$ 
 In Corollary \ref{Coro-Dqc-DHT-equiv-DHod}, we showed that the restriction functor $\cD_{qc, [0, p-1]}(D_{\HT})\to \cD_{qc, [0, p-1]}(D_{Hod})$ is an equivalence and described explicitly its inverse 
 composed with the other restriction $\cD_{qc, [0, p-1]}(D_{\HT})\to \cD_{qc, [0, p-1]}(D_{und})$.
Using this result,  we obtain the following description of $\cC$
\begin{equation}\label{eq:fiber_product_made_explicit}
 \cD_{qc,[0,p-1]}(W(k)^{syn}_{red})\hookrightarrow \cC \iso \mathrm{Eq}\begin{tikzcd}\Big(\cD_{qc,[0,p-1]}(D_{\dR})\arrow[shift left]{r}{\widetilde{i}_1^*}\arrow[shift right]{r}[swap]{T^*\circ \mathfrak{p}_{\bar{\dR}}^*}&\cD_{qc}(D_{und})\Big)\end{tikzcd}.   
\end{equation}
Here $\widetilde{i}_1^*$ is the restriction along the open embedding $\widetilde{i}_1: D_{und}\hookrightarrow D_{\dR}$, and $\mathfrak{p}_{\bar{\dR}}^*$ is the pullback along $ \mathfrak{p}_{\bar{\dR}}:(\bA^1_-/\bG_m)\otimes\bF_p\hookrightarrow D_{\dR}$, and $T^*$ is the pullback along the composition $T: B\bG_{m, \bF_p}^{\sharp} \otimes k \xrightarrow{\id \times \operatorname{Frob}}B\bG_{m,\bF_p}^{\sharp} \otimes k \to B\bG_{m,k}\to (\bA^1_-/\bG_m)\otimes\bF_p$. 
 
 The functor $\Phi_{\mathrm{DI}}$ factors through \eqref{eq:fiber_product_made_explicit} and the pullback functor 
\begin{equation}\label{eqn: functor-between-equalizers}
\mathrm{Eq}\begin{tikzcd}\Big(\cD_{qc,[0,p-1]}(D_{\dR})\arrow[shift left]{r}{\widetilde{i}_1^*}\arrow[shift right]{r}[swap]{T^*\circ \mathfrak{p}_{\bar{\dR}}^*}&\cD_{qc}(D_{und})\Big)\end{tikzcd}\to \mathrm{Eq}\begin{tikzcd}\Big(\cD_{qc,[0,p-1]}((\bA^1_-/\bG_m) \otimes \bF_p)\arrow[shift left]{r}{{i}_1^*}\arrow[shift right]{r}[swap]{q^*T^*}&\cD(k)\Big)\end{tikzcd}
\end{equation}
induced by $\mathfrak{p}_{\bar{\dR}}$ and $q: \Spec k\to B\bG_{m, k}^{\sharp}$ is the map classifying the trivial $\bG_{m, k}^\sharp$-torsor. Note that $T \circ q = i_0 \circ F$, and thus the equalizer on the right-hand side of \eqref{eqn: functor-between-equalizers} is precisely the category $\mathscr{DMF}^{\mathrm{big}}_{[0,p-1]}(W(k))\otimes\bF_p$. 
Lemma \ref{Dqc-DdR-as-a-fibre-product}, together with the following easy Lemma \ref{category-fibre-product-coequalizer-lemma}, shows that the functor 
\eqref{eqn: functor-between-equalizers} is an equivalence. 
\begin{lm}\label{category-fibre-product-coequalizer-lemma}
Let 
    \begin{equation}
    \begin{tikzcd}
        &\cC_1\times_{\cC_3}\cC_2\arrow[]{ld}[swap]{\pi_1}\arrow[]{rd}{\pi_2} &\\
        \cC_1\arrow[]{rd}[swap]{p_1}&&\cC_2\arrow[]{ld}{p_2}\\
        &\cC_3&
    \end{tikzcd}\end{equation}
    be a pullback square of $\infty$-categories, and let $F:\cC_1\to \cC_2$ be a functor (not necessarily commuting with $p_1$ and $p_2$). Then the equalizers are equivalent:    \[\mathrm{Eq}\begin{tikzcd}(\cC_1\times_{\cC_3}\cC_2\arrow[shift left]{r}{\pi_2}\arrow[shift right]{r}[swap]{F\circ \pi_1}&\cC_2)\end{tikzcd}\simeq \mathrm{Eq}\begin{tikzcd}(\cC_1\arrow[shift left]{r}{p_1}\arrow[shift right]{r}[swap]{p_2\circ F}&\cC_3)\end{tikzcd}.\]
The functor from left to right is given by sending 
$(X\in \cC_1, Y\in \cC_2, \alpha:p_1 X\simeq p_2Y,\beta: Y\simeq FX)\in \mathrm{Eq}\begin{tikzcd}(\cC_1\times_{\cC_3}\cC_2\arrow[shift left]{r}{\pi_2}\arrow[shift right]{r}[swap]{F\circ \pi_1}&\cC_2)\end{tikzcd}$ 
to $(X, \gamma:=(p_2\beta)\circ \alpha: p_1X\simeq p_2FX)\in \mathrm{Eq}\begin{tikzcd}(\cC_1\arrow[shift left]{r}{p_1}\arrow[shift right]{r}[swap]{p_2\circ F}&\cC_3)\end{tikzcd}$.
~The functor from right to left is given by sending 
$(X,\gamma: p_1X\simeq p_2FX)\in \mathrm{Eq}\begin{tikzcd}(\cC_1\arrow[shift left]{r}{p_1}\arrow[shift right]{r}[swap]{p_2\circ F}&\cC_3)\end{tikzcd}$ to $(X,FX,\gamma,\Id)\in \mathrm{Eq}\begin{tikzcd}(\cC_1\times_{\cC_3}\cC_2\arrow[shift left]{r}{\pi_2}\arrow[shift right]{r}[swap]{F\circ \pi_1}&\cC_2)\end{tikzcd}$. 
\end{lm}
\begin{proof}
    This is immediate. 
\end{proof}
Combining \eqref{eq:fiber_product_made_explicit} and equivalence \eqref{eqn: functor-between-equalizers}, we conclude that $\Phi_{\mathrm{DI}}$ is the composition
of a fully faithful embedding into the fiber product category and an equivalence: 
\begin{equation}\label{eqn: functor-between-equalizers_bis}
 \cD_{qc,[0,p-1]}(W(k)^{syn}_{red})\hookrightarrow \cC \iso  \mathscr{DMF}^{\mathrm{big}}_{[0,p-1]}(W(k))\otimes\bF_p.
\end{equation}
It remains to prove that the composition is essentially surjective. 
Given any $(\cM, \varphi) \in \mathscr{DMF}^{\mathrm{big}}_{[0,p-1]}(W(k))\otimes\bF_p$, let $\cF$ be the corresponding object of $\cC$. By  Lemma \ref{lemma:essential-image-restriction-functor-fibre-product}, it is enough to check that the pullback  of $\cF$ to $\cD_{qc}\left((C_2)_{v_-=0}/\bG_m\right){\times}_{\cD_{qc}\left((C_2)_{v_+^p=v_-=0}/\bG_m\right)}\cD_{qc}\left((C_2)_{v_+^p=0}/\bG_m\right)$
lies in $\cD_{qc}(C_2/\bG_m)$. But by diagram \eqref{big-3row-equalizer-diagram}, the pullback of $\cF$ is isomorphic to $t^* \cM \in \cD_{qc}(C_2/\bG_m)$. This completes the proof of Proposition \ref{prop:DI_is_an_equiv} as well as the proof of equivalence  \eqref{Phi-equiv}. \qed

\begin{rem}
    Note that in \textit{Step 3}, the only result that uses the fact that weights are in $[0, p-2]$ and not merely in $[0, p-1]$, is Corollary \ref{DI=Mar}. Moreover, the corollary is not true without this assumption on weights; see Remark \ref{rem:Mazur_vs_Deligne_Illusie} for an explanation. We also note that this assumption is essential for \eqref{restriction-map-i-to-i-1}, \textit{i.e.}, to show that the restriction $\cD_{qc, [0, p-2]}(W(k)^{syn} \otimes \bF_p) \to \cD_{qc, [0, p-2]} (W(k)^{syn}_{red})$ is an equivalence. However, it is only needed in the first step of the deformation argument, \textit{i.e.}, to deform from $W(k)_{red}^{syn}$ to $(W(k)^{syn} \otimes \bF_p)_{v_{1}^2 =0}$. In particular, the restriction $\cD_{qc, [0, p-1]} ((W(k)^{syn} \otimes \bF_p)_{v_1^3 = 0}) \to \cD_{qc, [0, p-1]} ((W(k)^{syn} \otimes \bF_p)_{v_1^2 = 0})$ is an equivalence.
\end{rem}

\textit{End of proof of Theorem 3.}
Since perfectness is preserved under pullbacks the functor $\Phi_{\mathrm{Maz}}$ restricts to 
$$\Perf_{[0,p-1]}(W(k)^{syn}) \to \cD^b\left(\mathscr{MF}_{[0,p-1]}(W(k))\right).$$
To prove \eqref{Phi-equiv-perfect} we have to show that, for every $\cF \in \Perf_{[0,p-2]}(W(k)^{syn})$,  the object $\cM:= \Phi_{\mathrm{Maz}}(\cF)\in  \cD^b\left(\mathscr{MF}_{[0,p-2]}(W(k))\right)$
is perfect. It suffices to check that the pullback of  $\Phi_{\mathrm{Maz}}(\cF)$ to $W(k)^{\cN}_{red}$ is perfect. Moreover, since the map $C_2/\bG_{m,k} \to W(k)^{\cN}_{red}$ is faithfully flat (see \cite[\S 7.7.1]{drinfeld2022prismatization}),
it suffices to check that the pullback of $\Phi_{\mathrm{Maz}}(\cF)$  to $C_2/\bG_{m, k}$ is perfect. But the latter is isomorphic to the pullback of $\cM$ along the composition 
$C_2/\bG_{m,k} \to \bA_-^1/\bG_m \otimes \bF_p \mono  \bA_-^1/\bG_m$. Indeed, both these maps $C_2/\bG_{m,k} \to \bA^1_{-}/\bG_m$ send a triple $(L, v_{-}: L \to R, u: R \to L^{\otimes p}) \in C_2/\bG_{m,k}(R)$ as in \eqref{points-of-C_2/Gm} to $(v_{-}: L \to R) \in \bA^1_{-}/\bG_m(R)$. This proves \eqref{Phi-equiv-perfect}.

Let us prove the last assertion of Theorem \ref{main-thm}. Recall from \cite{bhatt-lecture-notes} that every coherent sheaf on $W(k)^{syn}$ is perfect.  In particular, $\Phi_{\mathrm{Maz}}$ carries $ \Coh_{[0,p-1]}(W(k)^{syn}) $ into $\cD^b\left(\mathscr{MF}_{[0,p-1]}(W(k))\right)$. 
Let $\cM \in \mathscr{MF}_{[0,p-2]}(W(k))$ with $p \cM=0$. We claim that  $E:=\Phi_{\mathrm{Maz}}^{-1}(\cM)$ is a vector bundle over $W(k)^{syn}\otimes \bF_p$. 
Since $W(k)^{\cN}\otimes \bF_p  \to W(k)^{syn}\otimes \bF_p$ is an \'etale cover it suffices to verify that the pullback of  $E$ to $W(k)^{\cN}\otimes \bF_p$ is a vector bundle. 
An object of $\cD_{qc}(W(k)^{\cN}\otimes \bF_p)$ is a vector bundle if and only if its pullback to $W(k)^{\cN}_{red}$ is a vector bundle. Since 
the  morphism $C_2/\bG_m \to W(k)^{\cN}_{red}$ is faithfully flat, it is enough to show the assertion for the pullback of $E$ to   $C_2/\bG_m$. But the latter is the pullback $\cM$ under $C_2/\bG_m \to \bA_-^1/\bG_m \otimes \bF_p$. Now the claim follows from Propiosition \ref{FL-lemma} (2).

By d\'evissage, we infer that $\Phi_{\mathrm{Maz}}^{-1}(\cM)\in \Coh_{[0,p-2]}(W(k)^{syn})$ for every $\cM \in \mathscr{MF}_{[0,p-2]}(W(k))$.  
Thus the functor
$\Phi_{\mathrm{Maz}}^{-1}: \cD^b(\mathscr{MF}_{[0,p-2]}(W(k))) \to \cD_{qc}(W(k)^{syn})$ is $t$-exact.~Moreover, as we observed above, its essential image contains $ \Coh_{[0,p-2]}(W(k)^{syn})$.
This completes the proof.
\end{proof}
\begin{rem}\label{remark:Phi-Maz-not-an-equivalence}
    The functor $\Phi_{\mathrm{Maz}}$ is not an equivalence. Indeed, there exists a non-zero morphism 
    $v_1: \cO_{W(k)^{syn}}\otimes\bF_p  \to \cO_{W(k)^{syn}}\{p-1\}\otimes\bF_p$ whose restriction to $W(k)^{syn}_{red}$ is equal to $0$.
 Since the functor $\Phi_{\mathrm{Maz}}\otimes\bF_p$ factors through the restriction to the reduced locus (see \eqref{eqn:Phi-Maz-reduced}) we conclude that 
 $\Phi_{\mathrm{Maz}}(v_1)=0.$
 \end{rem}
 \begin{rem}\label{remark:vector-bundles}
 The last part of the proof of Theorem \ref{main-thm} shows that every $\cF \in  \Coh_{[0,p-2]}(W(k)^{syn}\otimes \bF_p)$ is a vector bundle over $W(k)^{syn}\otimes \bF_p$. Consequently, every  $p$-torsionfree $\cG \in  \Coh_{[0,p-2]}(W(k)^{syn})$ is a vector bundle over\footnote{Indeed, recall from \S \ref{subsection:coherent_sheaves_and_vector_bundles} that $W(k)^{syn}$ admits a faithfully flat cover $\Spf R \to W(k)^{syn}$, where $R$ is a  {\it Noetherian} ring complete with respect to an ideal $p\in J\subset R$, and  $\Spf R$ is the corresponding $J$-adic formal scheme. A quasi-coherent sheaf $\cG$ on $W(k)^{syn}$ is a vector bundle (resp.~a coherent sheaf) if and only if its pullback 
 to $\Spf R$ is a finite projective $R$-module (resp.~a finite $R$-module). Thus, our claim follows from the following elementary general statement: a $p$-torsion-free finite module $M$ over a $p$-complete Noetherian ring, 
 $p\in J$, is projective if and only if $M/p$ is a projective $R/p$-module. Indeed, it is enough to check that, for every maximal ideal $\fm\in R$, one has $\Tor_1^R(M, R/\fm)=0$. Since $R$ is $p$-complete, $p\in \fm$. Thus  $\Tor_1^R(M, R/\fm)= \Tor_1^{R/p}(M/p, R/\fm)=0$.}
 $W(k)^{syn}$. The assumption on Hodge-Tate weights cannot be weakened: for example, the sheaf
 $\cO_{W(k)^{syn}_{red}}\{p-1\} = \Coker ( \cO_{W(k)^{syn}}\otimes\bF_p\rar{v_1} 
 \cO_{W(k)^{syn}}\{p-1\}\otimes\bF_p)$ has Hodge-Tate weights $0$ and $p-1$, and it is not a vector bundle over   $W(k)^{syn}\otimes \bF_p$. Also, we have  
\begin{align*}
 \Phi_{\mathrm{Maz}}( \cO_{W(k)^{syn}_{red}}\{p-1\} ) &\xrightarrow{\sim} \Phi_{\mathrm{Maz}}(\cO_{W(k)^{syn}}\otimes\bF_p)[1] \oplus  \Phi_{\mathrm{Maz}}(\cO_{W(k)^{syn}}\{p-1\}\otimes\bF_p) \\
 &\xrightarrow{\sim} k[1] \oplus k(p-1).
 \end{align*}
In particular, $\Phi_{\mathrm{Maz}}(\cO_{W(k)^{syn}_{red}}\{p-1\})$ is not concentrated in cohomological degree zero. Thus, the Hodge-Tate weight assumption in the last assertion of Theorem \ref{main-thm} cannot be weakened. 
\end{rem}
\begin{rem}[The Mazur {\it versus} Deligne-Illusie decompositions]\label{rem:Mazur_vs_Deligne_Illusie}
Recall from the proof of  Theorem \ref{main-thm} the functors
\[\Phi_{\mathrm{DI}}, \Phi^{\mathrm{Maz}}_{red} : \cD_{qc,[0,p-1]}(W(k)^{syn}_{red})\to\mathscr{D}\mathscr{MF}^{\mathrm{big}}_{[0,p-1]}(W(k))\otimes\bF_p.\]
By Proposition \ref{prop:DI_is_an_equiv} $\Phi_{\mathrm{DI}}$ is an equivalence of categories. Consequently, we can consider 
\begin{equation}\label{eq:endofunctor}
\Phi^{\mathrm{Maz}}_{red} \circ \Phi_{\mathrm{DI}}^{-1}:  \mathscr{D}\mathscr{MF}^{\mathrm{big}}_{[0,p-1]}(W(k))\otimes\bF_p \to \mathscr{D}\mathscr{MF}^{\mathrm{big}}_{[0,p-1]}(W(k))\otimes\bF_p.   
\end{equation}
By Corollary \ref{DI=Mar} the restriction of this endofunctor to the subcategory  $\mathscr{D}\mathscr{MF}^{\mathrm{big}}_{[0,p-2]}(W(k))\otimes\bF_p$ is
isomorphic to $\Id$.  However, the functor \eqref{eq:endofunctor} is not even an auto-equivalence. In fact, using  Lemma \ref{DI-Maz}  one can describe 
$\Phi^{\mathrm{Maz}}_{red} \circ \Phi_{\mathrm{DI}}^{-1}$ explicitly.  We shall just state the answer. Recall that an object of
$\mathscr{D}\mathscr{MF}^{\mathrm{big}}_{[0,p-1]}(W(k))\otimes\bF_p$ consists of an effective object $F^\bullet \in \cD_{gr}(k[v_-])$ with $F^p=0$ together with 
a homotopy equivalence 
$$\varphi: \bigoplus_{i=0}^{p-1} F^* \Gra ^i_F F^{-\infty} \iso  F^{-\infty}.$$
Consider the endomorphism 
$$\Theta'= \sum _{i=0}^{p-1} (- i)\Id_{F^* \Gra ^i_F F^{-\infty}}$$ of the left-hand side and  let $\Theta= \varphi \circ \Theta' \circ \varphi ^{-1}$ be the corresponding endomorphism of $F^{-\infty}$. The endomorphism $\Theta$  is called the Sen operator (see \cite[\S 6.5]{bhatt-lecture-notes}).
Consider the composition $\bar \Theta$
\begin{equation}
    F^{p-1} \rar{v_-^{p-1}} F^0\iso F^{-\infty} \rar{\Theta} F^{-\infty} \iso F^0 \to F^{0}/F^1.
\end{equation}
Here we use that the canonical map $F^0 \to \underset{v_-}{\varinjlim}\,\ F^i =  F^{-\infty}$ is an equivalence, which follows from the effectivity of $F^\bullet$.
Define a nilpotent  endomorphism $\alpha$ of $\bigoplus_{i=0}^{p-1} F^* \Gra ^i_F F^{-\infty}$ to be  zero on all 
$F^* \Gra ^i_F F^{-\infty}$,  with $i\ne p-1$,  and $\alpha|_{F^* \Gra ^{p-1}_F F^{-\infty}}$ being the composition 
$$F^* \Gra ^{p-1}_F F^{-\infty}\iso F^* F^{p-1}\rar{F^* \bar \Theta} F^*  \Gra ^{0}_F F^{-\infty} \mono \bigoplus_{i=0}^{p-1} F^* \Gra ^i_F F^{-\infty}.$$
We claim that the functor \eqref{eq:endofunctor} carries each $(F^\bullet, \varphi)$ to  $(F^\bullet, \varphi \circ (\Id - \alpha))$. 

For example, if  $(F^\bullet, \varphi) \in \mathscr{MF}(\bZ_p)$ is a $p$-torsion Fontaine-Laffaille module, which is an extension of $\bF_p(1-p)$ by $\bF_p$ then 
the functor \eqref{eq:endofunctor} carries $(F^\bullet, \varphi)$ to $\bF_p(1-p) \oplus \bF_p$, that is, the endomorphism of  
$\Ext^1_{\mathscr{MF}(\bZ_p)\otimes \bF_p}(\bF_p(1-p), \bF_p)\iso \bF_p$ induced by \eqref{eq:endofunctor} is equal to $0$.

Let $X$ be a smooth $p$-adic formal scheme over $W(k)$ of dimension $\leq p-1$. By Remark \ref{remark:Phi-Maz-tilde-isom-PhiMaz-Hsyn}, we have
$\widetilde{\Phi}_{\mathrm{Maz}}(X) \iso  \Phi_{\mathrm{Maz}} \circ \cH_{syn}(X)$. On the other hand, 
by \cite[Theorem 5.10]{li2022endomorphisms} 
the $p$-torsion Fontaine-Laffaille module $\Phi_{\mathrm{DI}} ( \cH_{syn}(X)|_{W(k)^{syn}_{red}})$ is given by the Deligne-Illusie decomposition (\cite{Deligne-Illusie-87}).
Thus, in the geometric context, the above formula for $\Phi^{\mathrm{Maz}}_{red} \circ \Phi_{\mathrm{DI}}^{-1}$ describes the relation between the modulo $p$ reduction of \eqref{intro-Mazur-isom} and the Deligne-Illusie decomposition. For $\dim X <p-1$ the two decompositions coincide, but in dimension $p-1$ they are different. For example, if $X$ is an elliptic curve over $\bZ_2$ with  ordinary reduction then Fontaine-Laffaille module $H^1_{\dR}(X)\otimes \bF_2$ constructed 
using  \eqref{intro-Mazur-isom} ({\it i.e.}, $H^1(\Phi_{\mathrm{Maz}} \circ \cH_{syn}(X)) \otimes \bF_2$) is isomorphic to $\bF_2 \oplus \bF_2(-1)$ whereas 
the  Fontaine-Laffaille structure on $H^1_{\dR}(X)\otimes \bF_2$ induced by the Deligne-Illusie decomposition is a nonsplit extension of $\bF_2(-1)$ by $\bF_2$ unless $X\times \Spec \bZ/4$ is the {\it canonical} lift of the ordinary elliptic curve $X \times \Spec \bF_2$. 
\end{rem}
\begin{rem}[Syntomic cohomology in small weights]\label{rem:Syntomic cohomology in small weights}
Given an $F$-gauge $\cF \in \cD_{qc,[0,p-2]}(W(k)^{syn})$ one can use Theorem \ref{main-thm}  to compute the syntomic cohomology 
$R\Gamma(W(k)^{syn}, \cF\{i\})$, for $i\leq p-2$. Let us briefly explain how, using the functor 
\begin{equation}\label{eq:functor_to_Mazur_modules_bis}
   {\bf \Phi}_{\mathrm{Maz}}:   \cD_{qc,[0, +\infty]}(W(k)^{syn})\to \mathrm{Mod}_{\mathrm{Maz}}(W(k)), \quad \cF \mapsto (\mathfrak{p}^*_{\bar{\dR}}(\cF), \varphi_{\bullet})
 \end{equation}
constructed in Remark \ref{remark:Mazur-module},  the computation can be generalized to all effective $F$-gauges. Let $\cF \in \cD_{qc,[0, \infty]}(W(k)^{syn})$ be an effective $F$-gauge, ${\bf \Phi}_{\mathrm{Maz}}(\cF)= (F^\bullet, \varphi_{\bullet})$. We claim that, for  $i\leq p-2$, the functor ${\bf \Phi}_{\mathrm{Maz}}$
 induces equivalences 
\begin{equation}\label{eq:syn.coh}
   R\Gamma(W(k)^{syn}, \cF\{i \})  \iso \Hom_{\mathrm{Mod}_{\mathrm{Maz}}(W(k))} ({\bf \Phi}_{\mathrm{Maz}}(\cO_{W(k)^{syn}}\{-i\}), {\bf \Phi}_{\mathrm{Maz}}(\cF)))  
     \simeq
\mathrm{Eq}\begin{tikzcd}\Big(F^i \arrow[shift left]{r}{\varphi_i}\arrow[shift right]{r}[swap]{v_-^i}& F^{0}\Big)\end{tikzcd}.  
\end{equation}
To see the last isomorphism, note that by \cite[Proposition II.1.5]{MR3904731}, the mapping spectrum $\Hom_{\Mod_{\text{Maz}}(W(k))} (W(k)(-i), F^{\bullet})$ is identified with the fiber of $\Hom(W(k)(-i), F^{\bullet}) \xrightarrow{t} \Hom(W(k)(-i), F_* F^0 \otimes B)$, where mapping spectra are computed in $\widehat \cD_{gr}^{\eff}(W(k)[v_{-}])$ and $t$ takes $f: W(k)(-i) \to F^{\bullet}$ to the difference $\varphi_{\bullet} \circ f - f \circ \varphi_{\bullet}.$ Note that the source of $t$ can be identified with $F^i$ and the target with $F^0$.  Under this identification one always has $\varphi_{\bullet} \circ f = \varphi_i(f)$ and, provided $i<p$, one has $f \circ \varphi_{\bullet} = v_{-}^i(f)$.  For a proof,  it suffices to check that the map is an equivalence after reducing modulo $p$. The proof of the mod $p$ statement is explained\footnote{Bhatt and Lurie do not use the functor ${\bf \Phi}_{\mathrm{Maz}}$; instead they construct $\varphi _i$, $i< p-1$, with $\bZ/p$-coefficients, using the Sen operator.} in 
\cite[Remark 6.5.15]{bhatt-lecture-notes}. 

For $\cF$ of the form $\cH_{syn}(X)$, the formula \eqref{eq:syn.coh}  recovers \cite[Theorem F (2)]{Antieau-Mathew-Morrow-Nikolaus}.
\end{rem}
\begin{rem}[Lattices in crystalline representations]\label{rem:Lattices in crystalline representations}
Let $\mathscr{MF}_{[0,p-2]}^f(W(k))\subset  \mathscr{MF}_{[0,p-2]}(W(k))$ be the full subcategory, whose objects
are $p$-torsionfree Fontaine--Laffaille modules. In (\cite[\S 7.14]{Fontaine-Laffaille}), Fontaine and Laffaille constructed an exact  fully faithful functor  
\begin{equation}\label{eq:Fontaine-Laffaille_functor_to_Galois}
    \mathcal{T}_{\et}^{FL}: \mathscr{MF}_{[0,p-2]}^f(W(k)) \mono \mathrm{Rep}_{G_K}(\bZ_p)
\end{equation}
to the category of continuous $\bZ_p$-linear representations 
of the absolute Galois group $G_K$ of  $K:=\Frac (W(k))$. 
Moreover, Breuil identified, in  \cite[Proposition 3]{breuil}, the essential image of \eqref{eq:Fontaine-Laffaille_functor_to_Galois}
with the full subcategory $\mathrm{Rep}_{G_K, [0, p-2]}^{crys,f}(\bZ_p)\subset \mathrm{Rep}_{G_K}(\bZ_p)$, formed by all $\Lambda\in  \mathrm{Rep}_{G_K}(\bZ_p) $, such that $\Lambda$ is a finite free as a $\bZ_p$-module and   $\Lambda \otimes \bQ_p$ is a crystalline representation with Hodge--Tate weights lying  in $[0,p-2]$. 

On the other hand,  using results from \cite{Bhatt_Scholze_Galois}, Bhatt and Lurie established in  \cite[Theorem 6.6.13]{bhatt-lecture-notes} an equivalence of categories 
$$\mathcal{T}_{\et}^{BL}: \Coh^{\text{refl}}(W(k)^{syn}) \iso \mathrm{Rep}_{G_K}^{crys,f}(\bZ_p)$$
between the category \textit{reflexive} coherent sheaves on $W(k)^{syn}$ and
the category of lattices in crystalline representations of $G_K$ with  arbitrary Hodge--Tate weights.

By Remark \ref{remark:vector-bundles} the functor \eqref{Phi-equiv-perfect} restricts to an equivalence of categories
\begin{equation}
   \Phi_{\mathrm{Maz}}: \Vect_{[0,p-2]}(W(k)^{syn}) \iso \mathscr{MF}_{[0,p-2]}^f(W(k))
\end{equation}
between the subcategory  $\Vect_{[0,p-2]}(W(k)^{syn})\subset  \Coh_{[0,p-2]}(W(k)^{syn})$ of vector bundles and the subcategory 
$\mathscr{MF}_{[0,p-2]}^f(W(k))\subset \mathscr{MF}_{[0,p-2]}(W(k))$ of $p$-torsionfree Fontaine-Laffaille modules.

We claim that 
\begin{equation}\label{eq:main_result_on_Galois_rep}
  \mathcal{T}_{\et}^{FL} \iso \mathcal{T}_{\et}^{BL} \circ \Phi_{\mathrm{Maz}}^{-1}: \mathscr{MF}_{[0,p-2]}^f(W(k)) \to \mathrm{Rep}_{G_K, [0, p-2]}^{crys,f}(\bZ_p).
\end{equation}
The argument below was communicated to us by Akhil Mathew. Recall from \cite[Construction 6.3.1 and Corollary 6.6.5]{bhatt-lecture-notes} and \cite[Theorem 5.2]{Bhatt_Scholze_Galois} that, for any vector bundle $E\in  \Vect_{(-\infty, 0]}(W(k)^{syn})$ with non-positive weights, the corresponding Galois representation  is computed by the formula
\begin{align} \label{eq:applying_projection_formula_1}
  \mathcal{T}_{\et}^{BL}(E)= &H^0(\cO_C^{syn}, f^*_{syn}E) \simeq H^0(W(k)^{syn},  f_{syn \, *} f^*_{syn}E )\simeq \\
  \label{eq:applying_projection_formula_2}
&\simeq H^0(W(k)^{syn},  E \otimes  f_{syn\, *}\cO_{ \cO_C^{syn}}), 
\end{align}
where $\cO_C$ stands for the ring of integers of completed algebraic closure of $K$ and $f_{syn}: \cO_C^{syn} \to W(k)^{syn}$ is induced  by $\Spf \cO_C \to \Spf W(k)$. Note that $f_{syn\, *}\cO_{ \cO_C^{syn}}$ is an effective $F$-gauge, $f_{syn\, *}\cO_{ \cO_C^{syn}}\in \cD_{qc, [0, \infty ]}(W(k)^{syn})$. Thus, if $E\in  \Vect_{[2-p, 0]}(W(k)^{syn})$, we have   $E \otimes  f_{syn\, *}\cO_{ \cO_C^{syn}} \in  \cD_{qc, [2-p, \infty ]}(W(k)^{syn})$.
Consequently, \eqref{eq:applying_projection_formula_2} can be computed using formula \eqref{eq:syn.coh}. 
Note that by \cite[Proposition 9.9]{bhatt2012padic} the  $p$-completed derived de Rham cohomology of $\Spf \cO_C$ is identified with Fontaine's period ring $A_{crys}$ matching the Hodge filtration on the derived de Rham cohomology and the filtration by divided powers of $I:=\ker(A_{crys} \epi \cO_C)$. Now, let $(\cM, \varphi)\in \mathscr{MF}_{[0,p-2]}^f(W(k))$ and let $M:= \cM_{v_-=1}$ be the associated $W(k)$-module equipped with the Hodge filtration (see \eqref{identifying-CohA1modGm-with-modules-minus}).
 We apply the above computation to $E =  \big(\Phi_{\mathrm{Maz}}^{-1}(\cM,\varphi)\big)^*$.
By the construction of $\Phi_{\mathrm{Maz}}$, the de Rham realization  $\mathfrak{p}_{\bar{\dR}}(E)$, regarded as a filtered $W(k)$-module, is $M^*$ with its Hodge filtration.
Using \eqref{eq:syn.coh}, we have 
$$  \mathcal{T}_{\et}^{BL}(\Phi_{\mathrm{Maz}}^{-1}(\cM,\varphi)) \iso   \big(\mathcal{T}_{\et}^{BL}(E)\big)^* \iso 
\Big (\ker\big(F^0(M^* \otimes A_{crys}) \rar{\varphi -\Id} M^* \otimes A_{crys}\big)\Big) ^*.$$
Here $M^* \otimes A_{crys}$ is equipped with the tensor product filtration, $F^0(M^* \otimes A_{crys})=\sum_{i\geq 0} F^{-i}M^* \otimes I^{[i]}$,
and $\varphi = \varphi_{M^*} \otimes \varphi_{ A_{crys}}.$ The right-hand side of the above formula is $ \mathcal{T}_{\et}^{FL}(\cM, \varphi)$. This proves \eqref{eq:main_result_on_Galois_rep}.

As a consequence of \eqref{eq:main_result_on_Galois_rep}, we conclude that the Bhatt--Lurie equivalence carries every object of  $\mathrm{Rep}_{G_K, [0, p-2]}^{crys,f}(\bZ_p)$ to a \textit{vector bundle} over $W(k)^{syn}$.
\end{rem}

\bibliographystyle{amsalpha}
\bibliography{bibfile}

\end{document}